\documentclass[letterpaper,11pt,reqno]{article}

\usepackage[T1]{fontenc}
\usepackage{tgtermes}
\usepackage{amsthm}
\usepackage{mathtools, amssymb, amsfonts,mathrsfs,graphicx}
\usepackage[usenames,dvipsnames]{xcolor}

\usepackage[mathscr]{eucal}

\usepackage[all]{xy}
\usepackage{tikz}

\usetikzlibrary{matrix,arrows,calc,decorations.pathmorphing}

\usepackage[margin=1in]{geometry}
\usepackage[shortlabels]{enumitem}
\newlist{pfcases}{itemize}{1}
\setlist[pfcases]{label=--}
\newlist{pfsteps}{itemize}{1}
\setlist[pfsteps]{label={}}

\newcommand{\myindent}{\hspace{.5cm}}
\newcommand{\myspace}{\vspace{.25cm}}

\theoremstyle{plain}
\newtheorem{theorem}{Theorem}[section]
\newtheorem*{theorem*}{Theorem}

\newtheorem{proposition}[theorem]{Proposition}
\newtheorem{lemma}[theorem]{Lemma}
\newtheorem*{lemma*}{Lemma}

\newtheorem{cor}[theorem]{Corollary}
\newtheorem{claim}[theorem]{Claim}

\theoremstyle{definition}
\newtheorem{definition}[theorem]{Definition}

\newtheorem{ass}[theorem]{Assumption}
\newtheorem*{ass*}{Assumption}

\newtheorem{notation}[theorem]{Notation}
\newtheorem*{notation*}{Notation}

\theoremstyle{remark}
\newtheorem{remark}[theorem]{Remark}
\newtheorem{example}[theorem]{Example}

\numberwithin{equation}{section}

\DeclareMathOperator{\pr}{pr}
\DeclareMathOperator{\ppr}{Pr}
\DeclareMathOperator{\coker}{coker}

\DeclareMathOperator{\Hom}{Hom}

\newcommand{\sse}{\subseteq}

\newcommand*{\id}{\textup{id}}

\newcommand{\N}{\ensuremath{\mathbb N}}
\newcommand{\Z}{\ensuremath{\mathbb Z}}
\newcommand{\C}{\ensuremath{\mathbb C}}
\newcommand{\R}{\ensuremath{\mathbb R}}

\newcommand{\emb}{\hookrightarrow}

\def\smallint{{\begingroup\textstyle \int\endgroup}}

\makeatletter
\providecommand{\leftsquigarrow}{%
  \mathrel{\mathpalette\reflect@squig\relax}%
}
\newcommand{\reflect@squig}[2]{%
  \reflectbox{$\m@th#1\rightsquigarrow$}%
}
\makeatother

\newcommand{\sint}{\smallint}
\newcommand{\tlt}[1]{\tau_{< #1}}
\newcommand{\tleq}[1]{\tau_{\leq #1}}

\newcommand{\ti}[1]{\tilde{#1}}
\newcommand{\ha}[1]{\hat{#1}}

\newcommand{\Sh}{\mathrm{Sh}}
\newcommand{\si}{\sigma}

\newcommand{\Sn}{\mathbb{S}}
\newcommand{\maps}{\colon}
\newcommand{\tensor}{\otimes}
\renewcommand{\deg}[1]{\left \lvert #1 \right \rvert}

\newcommand{\sulreal}[1]{\bigl \langle #1 \bigr\rangle}

\renewcommand{\S}{\bar{S}}

\newcommand{\rDelta}{\bar{\Delta}}
\newcommand{\rdDelta}[1]{\bar{\Delta}^{(#1)}}

\newcommand{\del}{\partial}

\newcommand{\bs}{\mathbf{s}}                 
                  
\newcommand{\ds}{\mathbf{s}^{-1}}

\newcommand{\xto}[1]{\xrightarrow{#1}}

\newcommand{\CE}{\mathrm{CE}}

 \newcommand{\Q}{\mathbb{Q}}

\newcommand{\kk}{\Bbbk}

\newcommand{\dgla}{\mathsf{dgla}_{\geq 0}}
\newcommand{\Ch}{\mathsf{Ch}}
\newcommand{\Chain}{\mathsf{Ch}_{\geq 0}}
\newcommand{\Linf}{\mathsf{L}_{\infty}\mathsf{Alg}}

\newcommand{\Set}{\mathsf{Set}}

\newcommand{\cD}{\mathsf{D}}

\renewcommand{\C}{\mathsf{C}}

\newcommand{\g}{\mathfrak{g}}

\newcommand{\h}{\mathfrak{h}}

\newcommand{\bb}[1]{\vec{#1}}

\newcommand{\LnA}[1]{\mathsf{Lie}_{#1}\mathsf{Alg}}
\newcommand{\dgcocom}{\mathsf{dgCoCom}_{\geq 0}}
\newcommand{\cocom}{\mathsf{CoCom}_{\geq 0}}
\newcommand{\dgcocomu}{\mathsf{dgCoCom}}
\newcommand{\cocomu}{\mathsf{CoCom}}

\newcommand{\lt}{<}
\newcommand{\vv}{\vee}
\renewcommand{\odot}{\vv}

\newcommand{\antish}{\text{\rotatebox[origin=c]{180}{$!$}}}

\DeclareMathOperator{\End}{\mathrm{End}}
\DeclareMathOperator{\im}{\mathrm{im}}

\DeclareMathOperator{\diag}{\mathrm{diag}}

\DeclareMathOperator{\ft}{\mathrm{fin}}

\DeclareMathOperator{\cdga}{\mathsf{cdga}}
\DeclareMathOperator{\bcdga}{\mathsf{cdga}^{\mathrm{bnd}}_{\geq 0}}
\DeclareMathOperator{\poly}{\mathrm{poly}}
\DeclareMathOperator{\curv}{\mathrm{curv}}
\DeclareMathOperator{\MC}{\mathrm{MC}}

\DeclareMathOperator{\proj}{\mathrm{proj}}
\DeclareMathOperator{\tanch}{\mathrm{tan}_{\geq 0}}

\newcommand{\vphi}{\varphi}
\newcommand{\tha}{\theta}

\newcommand*{\emptycomment}[1]{}

\newcommand{\lnaft}{\LnA{n}^{\ft}}

\newcommand{\el}{\ell}

\newcommand{\sgn}[1]{ (-1)^{\frac{#1(#1-1)}{2}}}

\usepackage[breaklinks=true,colorlinks=true,citecolor=blue,urlcolor=blue,linkcolor=black]{hyperref}

\title{ An explicit model for 
the homotopy theory of finite type Lie $n$-algebras } 
\renewcommand\footnotemark{}
\author{Christopher L.\  Rogers 
\thanks{\hspace{-.625cm} Department of Mathematics \& Statistics, University of Nevada,
  Reno. 1664 N. Virginia Street Reno, NV 89557-0084 USA. |  chrisrogers@unr.edu, chris.rogers.math@gmail.com}
}
\date{}
\begin{document}
\maketitle
\begin{abstract}
Lie $n$-algebras are the $L_\infty$ analogs of chain Lie algebras from rational homotopy theory.
Henriques showed that finite type Lie $n$-algebras can be integrated to produce certain simplicial Banach manifolds, known as Lie $\infty$-groups, via a smooth analog  
of Sullivan's realization functor.
In this paper, we provide an explicit proof that the category of finite type Lie $n$-algebras and (weak) $L_\infty$-morphisms admits the structure of a category of fibrant objects (CFO) for a homotopy theory. Roughly speaking, this CFO structure can be thought of as the transfer of the classical projective CFO structure on non-negatively graded chain complexes via the tangent functor. In particular, the weak equivalences are precisely the $L_\infty$ quasi-isomorphisms. Along the way, we give explicit constructions for pullbacks and factorizations of $L_\infty$-morphisms between finite type  Lie $n$-algebras. We also analyze  Postnikov towers and Maurer--Cartan/deformation functors associated to such  Lie $n$-algebras. The main application of this work is our joint paper \cite{Rogers-Zhu:2018} with C.\ Zhu which characterizes the compatibility of Henriques' integration functor with the homotopy theory of Lie $n$-algebras and that of Lie $\infty$-groups.   
\end{abstract}

\setcounter{tocdepth}{2}
\tableofcontents

\section{Introduction}
The motivation for this paper lies within the various algebraic formalisms developed decades ago for classifying rational homotopy types. In \cite{Quillen:1969}, Quillen established the relationship between the rational homotopy theory of simply-connected spaces, and the homotopy theory of connected chain Lie algebras over $\Q$. By a chain Lie algebra, we mean a chain complex $L$ of vector spaces concentrated in non-negative degrees equipped with a differential graded Lie algebra (dgla) structure. 
A chain Lie algebra $L$ is connected if it is strictly positively graded, i.e. $L_0=0$. Quillen formalized the homotopy theory of connected chain Lie algebras using a model structure 
in which a weak equivalence is defined to be a dgla morphism whose underlying chain map induces an isomorphism in homology, and in which a fibration is defined to be a dgla morphism whose underlying chain map is surjective in all degrees.

From here, we can make a quick leap to the approach developed by Sullivan in \cite{Sullivan:1978}, provided we restrict our attention to ``finite type'' chain Lie algebras, i.e.\ those algebras whose underlying chain complex is finite-dimensional in each degree. 
First, recall that the Chevalley-Eilenberg algebra $\CE(L)$ associated to a chain Lie algebra $L$ is the commutative dg algebra (cdga) obtained by taking the linear dual of the bar construction of $L$.  
If $L$ is finite type and connected, then 
$\CE(L)$ is simply connected, and admits the structure of a Sullivan algebra. Roughly,  
a cdga is a Sullivan algebra if, as a graded commutative algebra, it is freely generated by a 
filtered graded vector space that satisfies certain compatibility conditions with the differential. Sullivan algebras are the cofibrant cdgas in the model structure introduced
by Bousfield and Gugenheim \cite{BG:1976}. In particular, the Sullivan algebra $\CE(L)$ is a  model for the rational homotopy type of its realization: the Kan simplicial set $\sulreal{\CE(L)}:=\hom_{\cdga}(\CE(L), \Omega^{\ast}_{\poly}(\Delta^\bullet))$, where 
$\Omega^{\ast}_{\poly}(\Delta^n)$ is the cdga of polynomial de Rham forms on the geometric $n$-simplex. 

A more direct path from chain Lie algebras to Kan complexes is via deformation theory. Associated to any $\Z$-graded dgla $(L,d, [\cdot,\cdot])$ is its set of Maurer-Cartan elements $\MC(L):=\{ a \in L_{-1} ~\vert ~ da + \frac{1}{2}[a,a]=0\}$. Furthermore, the dgla structure on $L$ induces a natural simplicial dgla structure on the tensor product $L \tensor_{\Q} \Omega^{\ast}_{\poly}(\Delta^\bullet)$. As noted by Getzler \cite{Ezra-infty}, if $L$ is finite type, then there is a natural isomorphism of simplicial sets: 
\[
\sulreal{\CE(L)} \cong \sint L 
\]
where $(\sint L)_n:=\MC \bigl(L \tensor \Omega^{\ast}_{\poly}(\Delta^n) \bigr)$.
Hence,  if $L$ is finite type connected, then $L$ is a Lie model for the rational homotopy type of the space $\sint L$. Moreover, the ``simplicial Maurer-Cartan functor'' $\int$ is compatible with the respective homotopy theories. Indeed, by combining Quillen's work \cite{Quillen:1969} with the results of Bousfield and Gugenheim \cite{BG:1976}, it follows that $\int$ preserves both weak equivalences and fibrations. In particular, $\int L$ is a Kan complex for every finite type connected chain Lie algebra $L$.

This leads to an obvious question: What is the analog of this story for arbitrary finite type chain Lie algebras  over a field of characteristic zero? That is, suppose there are no constraints imposed on degree zero elements and no assumptions involving nilpotency or completeness. To begin with, Quillen's model for the homotopy theory of connected rational chain Lie algebras extends in a straightforward way to the more general case over any field of characteristic zero. It follows from the work of Getzler and Jones \cite[Thm.\ 4.4]{Getzler-Jones:1994} that the category of chain Lie algebras over such a field admits a model structure induced by the projective model structure on non-negatively graded chain complexes. Hence, a weak equivalence is, as before, a dgla morphism whose underlying chain map is a quasi-isomorphism, and a fibration is defined to be a dgla morphism whose underlying chain map is surjective in positive degrees. This is a natural generalization of Quillen's model structure for the rational connected case.

However, the spatial realization of finite type chain Lie algebras with no constraints on connectivity is a more subtle endeavor. Indeed, all finite dimensional non-nilpotent Lie algebras are examples of such chain Lie algebras. Consequently, as demonstrated by Sullivan \cite[``Theorem 8.1'']{Sullivan:1978}, the realization of chain Lie algebras over $\R$ necessarily involves the diffeo-geometric integration of Lie algebras to Lie groups (i.e.\ Lie's Third Theorem). 

The existence of a smooth realization functor for such chain Lie algebras is a special case of 
the more general problem addressed by Henriques in his work \cite{Henriques:2008} on the integration of Lie $n$-algebras. Lie $n$-algebras are $L_\infty$-algebras (or ``strong homotopy Lie algebras'') whose underlying chain complexes are non-negatively graded. Thus they include chain Lie algebras as a special case. However, the ``correct'' notion of morphism between Lie $n$-algebras is significantly weaker than just a linear map which preserves the $L_\infty$-structure on the nose. This implies that the category of Lie $n$-algebras and weak $L_\infty$-morphisms is not a category of algebras over the $L_\infty$ operad. Hence, a model for the homotopy theory of Lie $n$-algebras does not follow from the aforementioned result of Getzler and Jones. The main result (Thm.\  \ref{thm:LnA_CFO}) of this paper resolves this issue by explicitly providing such a model. 

Henriques' integration procedure for finite type Lie $n$-algebras involves replacing the polynomial de Rham forms in Sullivan's realization functor with the dg Banach algebra of $C^r$-differential forms. The output of this procedure is a group-like simplicial Banach manifold, 
or ``Lie $\infty$-group'', which satisfies a diffeo-geometric analog of the horn filling condition for Kan simplicial sets. Simplicial manifolds of this kind have been used as
geometric models for the higher stages of the Whitehead tower of the orthogonal group. The most famous example of such a model is called the ``String Lie 2--group'', which Henriques showed can be obtained by integrating its infinitesimal analog, the ``string Lie 2-algebra''.  

The results of this paper, when combined with our joint work with Zhu in the companion paper \cite{Rogers-Zhu:2018}, address the compatibility of Henriques' integration functor with the homotopy theories of  Lie $n$-algebras and Lie $\infty$-groups. This can be understood as the smooth analog of the aforementioned results of Quillen and Bousfield-Gugenheim which characterize the homotopical properties of the realization functor for connected Lie models for rational homotopy types.

\subsection*{Overview and main results}
After reviewing standard facts concerning $L_\infty$-algebras in the sections leading up to Sec.\ \ref{sec:lna}, we consider in Def.\ \ref{def:LnA_morphs} two classes of morphisms in the category $\LnA{n}$ of Lie $n$-algebras.
We say a $L_\infty$-morphism $(f_1,f_2,\ldots) \maps (L,\el_1,\el_2, \el_3 , \ldots) \to (L',\el'_1,\el'_2, \el'_3 , \ldots)$ is a \textbf{weak equivalence} iff the chain map $f_1 \maps (L,\el_1) \to (L',\el'_1)$ is a quasi-isomorphism, and we say it is a \textbf{fibration} iff the chain map $f_1$ is a surjection in all positive degrees. Thus weak equivalences coincide with $L_\infty$ quasi-isomorphisms. Although every weak equivalence between Lie $n$-algebra induces a quasi-isomorphism between their associated Chevalley-Eilenberg (co)algebras, the converse is not true, in contrast with the simply connected case mentioned in the above introduction. (See Remark \ref{rmk:acyclic_fibs} for an explicit counterexample.)

In Sec.\ \ref{sec:LnA_fact}, we prove several useful technical results concerning 
morphisms in $\LnA{n}$. We show in Prop.\ \ref{prop:strict_factor} that every strict $L_\infty$-morphism (in the sense of Sec. \ref{sec:strict}) can be factored into a fibration followed by a weak equivalence. In particular, the diagonal map $L \to L \oplus L$ in $\LnA{n}$ admits such a factorization. Therefore, it follows from our main theorem (see below) and Brown's Factorization Lemma (see Lemma \ref{lem:fact}) that every weak $L_\infty$-morphism in $\LnA{n}$ admits such a factorization.     

Next, in Lemma \ref{lem:strict_fib}, we show that every fibration can be factored into an isomorphism followed by a strict fibration. The proof is a simple modification of a result of Vallette \cite{Vallette:2014} concerning the factorization of ``$\infty$-epimorphisms''. (See, for example, Def.\ \ref{def:quasi-iso}). This ``strictification of fibrations'' is a very useful tool which we use repeatedly throughout this paper and in the companion paper \cite{Rogers-Zhu:2018}. 

In Prop.\ \ref{prop:strict_pullback} and Cor.\ \ref{cor:fib_pback}, we explicitly construct pullbacks of fibrations and acyclic fibrations along arbitrary morphisms in $\LnA{n}$.
Moreover, the pullback of a (acyclic) fibration is again a (acyclic) fibration. The proof of these facts involves
a clever use of certain coalgebra endomorphisms, which we learned from studying Vallette's proof of 
his Thm.\ 4.1 in \cite{Vallette:2014}.

Since our main application is integration, in Sec.\ \ref{sec:LnA_CFO}, we restrict our attention to the full subcategory $\lnaft$ of finite type Lie $n$-algebras. The category $\lnaft$ does not admit a model structure, since it does not have all limits and colimits. So instead, we work within Brown's framework \cite{Brown:1973} of a category of fibrant objects, or ``CFO'', for a homotopy theory (Def.\ \ref{def:cfo}). The main result (Thm.\  \ref{thm:LnA_CFO}) of the paper is:

\begin{theorem*}
Let $n \in \N \cup \{\infty\}$.  The category $\LnA{n}^{\ft}$ of finite type Lie $n$-algebras 
over a field of characteristic zero and weak $L_\infty$-morphisms has the structure of a category of fibrant objects, in which 
a morphism $$(f_1,f_2,\ldots) \maps (L,\el_1,\el_2, \el_3 , \ldots) \to (L',\el'_1,\el'_2, \el'_3 , \ldots)$$ is:
\begin{itemize}
\item[-]a weak equivalence iff the chain map $f_1 \maps (L,\el_1) \to (L',\el'_1)$ is a quasi-isomorphism of chain complexes, and
\item[-] a fibration iff the chain map $f_1 \maps (L,\el_1) \to (L',\el'_1)$ is a surjection in all positive degrees (i.e.\ in all degrees $\geq 1$).
\end{itemize}
\end{theorem*}

Hence, the homotopy theory that we consider for Lie $n$-algebras is inherited from the projective model structure for non-negatively graded chain complexes (Sec.\ \ref{sec:prelim}), via the  
tangent functor which assigns to a $L_\infty$-morphism $(f_1,f_2,\ldots)$ as above, the chain map $f_1  \maps (L,\el_1) \to (L',\el'_1)$. In particular, our results imply that the tangent functor is an exact functor  between categories of fibrant objects (Cor.\ \ref{cor:tanexact}). We also note that this CFO structure is compatible with the one induced on the category of chain Lie algebras
by the aforementioned Getzler-Jones/Quillen model structure. (Recall that chain Lie algebras form a non-full subcategory of $\LnA{\infty}$.) 

Also in Sec.\ \ref{sec:LnA_CFO}, we compare the category of fibrant objects structure on finite type Lie $n$-algebras with  Vallette's CFO structure (Thm.\ \ref{thm:vallette_cfo}) on $\Z$-graded $L_\infty$-algebras.  

In Sec.\ \ref{sec:MCfunc}, we analyze Maurer-Cartan (MC) sets of certain $\Z$-graded $L_\infty$-algebras which are constructed by tensoring Lie $n$-algebras with bounded commutative dg algebras. This is a familiar procedure used in deformation theory for constructing deformation functors out of pronilpotent $L_\infty$-algebras. Maurer-Cartan sets are also used to define Henriques' integration functor.  We prove that this construction sends pullback diagrams of fibrations in $\LnA{n}$ to pullback diagrams of MC sets. In Cor.\ \ref{cor:MC-pullbacks}, we show that an analogous statement holds in the smooth category when the MC sets are equipped with a Banach manifold structure. The proofs crucially depend on the explicit description of pullbacks given in Sec.\ \ref{sec:pullbacks}.

Finally, in Sec.\ \ref{sec:postnikov}, we analyze a very useful Postnikov tower construction for Lie $n$-algebras which was introduced by Henriques in \cite{Henriques:2008}.
They play a key role in his proof that Lie $n$-algebras integrate to fibrant simplicial manifolds. We also introduce an important class of fibrations in $\LnA{n}$ called ``quasi-split fibrations'' (Def.\ \ref{def:split_fib}). Such fibrations naturally arise in applications, e.g.\ string extensions. 
We show that a morphism of Postnikov towers induced by a quasi-split fibration admits a convenient functorial decomposition (Prop.\ \ref{prop:tower_decomp1} and Prop.\ \ref{prop:tower_decomp2}).
We use this result and the aforementioned results concerning MC sets in \cite{Rogers-Zhu:2018} to show that Henriques' integration functor is an exact functor with respect to the class of quasi-fibrations. (See Thm.\ 9.16 in \cite{Rogers-Zhu:2018}.)

In this paper, we work with $L_\infty$-algebras within the context of dg cocomutative coalgebras, rather than commutative dg algebras, even though we are ultimately interested in finite type objects. This is because our main result Thm.\  \ref{thm:LnA_CFO}, as well as 
many of the auxiliary results, also hold for infinite dimensional Lie $n$-algebras. (See Remark 
\ref{rmk:inf_dim}.) Furthermore, many of the technical tools we develop in order to prove our main results are the ``non-negatively graded'' variations of Vallette's work \cite{Vallette:2014} on the homotopy theory of $\Z$-graded dg $\mathcal{P}^{\antish}$ coalgebras (where $\mathcal{P}^{\antish}$ is the Koszul dual cooperad of an operad $\mathcal{P}$).  For the convenience of the reader, we quickly recall in the next section various facts concerning cocommutative coalgebras, as well as establish the notation and conventions that we use throughout the paper.

\vspace{-.25cm}
\subsection*{Acknowledgments}
The author thanks: Aydin Ozbek and Chenchang Zhu for very helpful discussions during the preparation of this manuscript, Ezra Getzler for several informative conversations on the homotopy theory of $L_\infty$-algebras, and the anonymous referees for their valuable suggestions for improvements. 
The author is also indebted to Bruno Vallette for writing the inspiring preprint \cite{Vallette:2014}. This work was supported by a grant from the Simons Foundation/SFARI (585631,CR).
\vspace{-.25cm}

\section{Preliminaries and notation} \label{sec:prelim}
\subsection{Graded linear algebra} \label{sec:lin_alg}
Throughout, $\kk$ denotes a field of characteristic zero.
For a $\Z$-graded vector space $V$ we denote 
by $\bs V$ (resp. by $\bs^{-1} V$) the suspension (resp. the 
desuspension) of $V$\,. In other words, 
\[
\bs V_{i}:=V[-1]_{i}=V_{i-1} \qquad \bs^{-1} V_{i}:=V[1]_{i}=V_{i+1}
\]
We denote by $\deg{x}$ the degree of a homogeneous element $x \in V$. Similarly, $\deg{f}$ denotes the degree of a linear map $f \maps V \to V'$ between graded vector spaces $V$  and $V'$.

Let $x_{1}, \hdots,  x_{n}$ be
elements of $V$ and $\sigma \in \Sn_n$ a permutation. The notation
$\epsilon(\sigma)=\epsilon(\sigma ;
x_{1},\hdots,x_{n})$ is reserved for the Koszul sign, which is defined by the equality
\[
x_{1} \vv  \cdots \vv  x_{n} := \epsilon(\sigma) x_{\sigma(1)} \vv \cdots \vv  x_{\sigma(n)} \in S(V)
\]
which holds in the free graded commutative algebra $S(V)$ generated by
$V$. Note that $\epsilon(\sigma)$ does not include the sign $(-1)^{\sigma}$ of the permutation $\sigma$. 

The notation 
$\sigma \cdot (x_1 \tensor x_2 \tensor \cdots \tensor x_n)$ is reserved for the left action of $\Sn_n$ on $V^{\tensor n}$, i.e.
\begin{equation} \label{eq:Sn_action}
\sigma \cdot x_1 \tensor x_2 \tensor \cdots \tensor x_n:= \epsilon(\sigma)
x_{\si^{-1}(1)} \tensor x_{\si^{-1}(2)} \tensor \cdots \tensor x_{\si^{-1}(n)}.
\end{equation}

We denote by $\Sh_{p_1, \dots, p_k} \sse \Sn_n$ the subset of $(p_1, \dots, p_k)$-shuffles 
in $\Sn_n$, i.e.  $\Sh_{p_1, \dots, p_k}$ consists of 
elements $\si \in \S_n$, $n= p_1 +p_2 + \dots + p_k$ such that 
$$
\begin{array}{c}
\si(1) <  \si(2) < \dots < \si(p_1),  \\[0.3cm]
\si(p_1+1) <  \si(p_1+2) < \dots < \si(p_1+p_2), \\[0.3cm]
\dots   \\[0.3cm]
\si(n-p_k+1) <  \si(n-p_k+2) < \dots < \si(n)\,.
\end{array}
$$

We use \underline{homological} conventions for all differential graded (dg) structures except for the cochain algebras that appear in Sec.\ \ref{sec:MCfunc}.

\subsection{Notation} \label{sec:cats}
\subsubsection{The model category $\Chain^{\proj}$} \label{sec:chproj}
We refer the reader to the introductory article \cite{Dwyer-Spal:1995} for basic facts concerning model categories. Throughout our paper, $\Chain$ denotes the category of chain complexes over $\kk$ concentrated in non-negative degrees. The notation $\Chain^{\proj}$ is reserved for the category $\Chain$ equipped with 
the projective model structure \cite[Thm.\ 7.2]{Dwyer-Spal:1995}. Weak equivalences are the quasi-isomorphisms, and fibrations are those chain maps which are surjective in all {positive} degrees. Since we work over a field, the cofibrations are those chain maps which are injective in all degrees. 

Since all objects in $\Chain^{\proj}$ are (bi)fibrant, we will also consider $\Chain^{\proj}$ as a category of fibrant objects (Def.\ \ref{def:cfo}) whenever it is convenient to do so.   

\subsubsection{Notation for categories}
We will consider several closely related categories. We list them here as a convenient reference for the reader as they traverse through the paper.

\begin{itemize}

\item[-] $\Ch$ denotes the category of $\Z$-graded (i.e.\ unbounded) chain complexes over $\kk$. 

\item[-] $\cocomu$ (resp.\ $\dgcocomu$) denotes the category of $\Z$-graded (resp.\ dg) conilpotent 
coaugmented cocommutative coalgebras (Sec.\ \ref{sec:cocom}).  

\item[-] $\Linf$ is the category whose objects are $\Z$-graded $L_\infty$-algebras and whose morphisms are weak $L_\infty$-morphisms (Sec.\ \ref{sec:L_inf}). We will tacitly identify $\Linf$ as the full subcategory of $\dgcocomu$ consisting of those dg coalgebras whose underlying graded coalgebras are cofree (Sec.\ \ref{sec:cofree}).  

\item[-] $\cocom$ (resp.\ $\dgcocom$) denotes the full subcategory of $\cocomu$ (resp.\ $\dgcocomu$)
consisting of those graded (resp.\ dg) coalgebras whose underlying graded vector spaces are concentrated in non-negative degrees.  

\item[-] Fix $n \in \N \cup \{\infty\}$. We denote by $\LnA{n}$ the category of Lie $n$-algebras: the full subcategory of $\Linf$ consisting of those $L_\infty$-algebras whose underlying chain complex is concentrated in degrees $0, 1, \ldots, n-1$. (Sec.\  \ref{sec:lna}). Hence, the morphisms in $\LnA{n}$ are always taken to be \underline{weak} $L_\infty$-morphisms. Again we will usually identify $\LnA{n}$ as a full subcategory of $\dgcocom$. 

The notation $\lnaft$ is reserved for the full subcategory of $\LnA{n}$ consisting of finite type Lie $n$-algebra (Sec.\ \ref{sec:lna}.)

\item[-] We denote by $\bcdga$ the category of bounded cochain algebras. (Sec.\ \ref{sec:MCfunc}.)
Its objects are \underline{cohomologically} graded unital commutative dg $\kk$-algebras 
whose underlying graded vector spaces are concentrated in non-negative degrees and
bounded from above.  The morphisms in $\bcdga$ are unit preserving cdga morphisms.

\end{itemize}

\subsection{Conilpotent cocommutative coalgebras} \label{sec:cocom}
The following facts and notational conventions concerning dg coalgebras are standard, and the reader already familiar with treatments of $L_\infty$-algebras as dg coalgebras, e.g.\ \cite{Lada-Markl:1995} or \cite[Sec.\ 10.1.6]{LoVa} can likely skip this section.

For a basic introduction to dg coalgebras and their morphisms, we suggest Sections 3d and 22a of \cite{FHT}, Appendix B of \cite{Quillen:1969},  and Section 2 of \cite{Hinich:2001}.
We recall that a  graded counital cocommutative coalgebra $(C,\Delta,\epsilon)$  
is \textbf{coaugmented} iff it is  equipped with a distinguished degree zero element $\mathbf{1} \in C_0$ satisfying $\Delta \mathbf{1} = \mathbf{1} \tensor \mathbf{1}$, and $\epsilon(\mathbf{1})=1$. For such a coalgebra, we have a decomposition of vector spaces
\[
C \cong \kk \cdot \mathbf{1} \oplus \bar{C}, \quad \bar{C} := \ker \epsilon.
\] 
Let $\rDelta \maps \bar{C} \to \bar{C} \tensor \bar{C}$ denote the \textbf{reduced comultiplication}
defined as $\rDelta(c):= \rDelta c = \Delta c - c \tensor \mathbf{1} - \mathbf{1} \tensor c$.
We call the non-counital cocomutative coalgebra $(\bar{C}, \rDelta)$ the associated \textbf{reduced coalgebra}.
The reduced diagonal $\rdDelta{n}$ is recursively defined by the formulas:
\[
\rdDelta{0}:=\id, \qquad 
\rdDelta{1}:=\rDelta, \qquad 
\rdDelta{n}:= \bigl( \rDelta \tensor \id^{\tensor(n-1)} \bigr) \circ
\rdDelta{n-1} \maps \bar{C} \to \bar{C}^{\tensor(n+1)}.
\]

A coaugmented counital cocommutative coalgebra $(C,\Delta,\epsilon, \mathbf{1})$ is \textbf{conilpotent} iff $\bar{C} = \bigcup_{n} \ker \rdDelta{n}$.    
We define $\cocomu$ to be the category whose objects are $\Z$-graded conilpotent coaugmented counital cocommutative coalgebras. Similarly, we define $\dgcocomu$ to be the category whose objects are coalgebras $(C,\Delta,\epsilon, \mathbf{1})$ in $\cocomu$ equipped with a degree $-1$ codifferential $\delta$ satisfying $\delta(\mathbf{1})=0$.

The full subcategories $\cocom \sse \cocomu$ and $\dgcocom \sse \dgcocomu$ are defined analogously.

\subsubsection*{Reduced non-counital dg coalgebras $(\bar{C},\rDelta,\delta)$ and their morphisms}
The codifferential $\delta$ of any conilpotent dg coalgebra $(C,\Delta,\epsilon, \mathbf{1},\delta) \in \dgcocomu$ is uniquely determined by its restriction to the corresponding reduced coalgebra $(\bar{C},\rDelta)$. We will use the same notation for $\delta$ and its restriction to $\bar{C}$. Similarly, since we are over a field, morphisms in $\dgcocomu$ are uniquely determined by their restriction to the associated reduced dg coalgebras 
\cite[Sec.\ 2.1]{Hinich:2001}. From now on, when dealing with the categories $\dgcocom$, $\dgcocomu$, etc., we will tacitly work with the associated reduced dg coalgebras and their morphisms, and make no mention of counits or coaugmentations.


\subsection{Cofree conilpotent coalgebras and their morphisms} \label{sec:cofree}
Let $V$ be a graded vector space. The symmetric algebra generated by $V$
\begin{equation*}
S(V)= \kk \oplus \S(V), \qquad\S(V):=V \oplus S^{2}(V) \oplus S^{3}(V) \oplus \hdots
\end{equation*}
is naturally a graded conilpotent cocommutative coalgebra with comultiplication $\Delta$ defined
as the unique morphism of algebras such that $\Delta(v) = v \tensor 1 + 1
\tensor v$ for all $v\in V$. The comultiplication for the corresponding reduced coalgebra $(\S(V),\rDelta)$ is explicitly:
\begin{equation*}
\begin{split}
 \rDelta(v_{1}, v_{2}, \ldots, v_{m})=&
  \sum_{1 \leq p \leq m-1} \sum_{\sigma \in \Sh(p,m-p)} \epsilon(\sigma)
  \left(v_{\sigma(1)} \vv v_{\sigma(2)} \vv \cdots \vv v_{\sigma(p)} \right)\\
  & \tensor \left ( v_{\sigma(p+1)} \vv v_{\sigma(p+2)} \vv \cdots \vv v_{\sigma(m)}
  \right).
\end{split}
\end{equation*}

Let $V$ and $V'$ be graded vector spaces. Let $\Phi \maps \S(V) \to \S(V')$ be a linear map.
For $p,m \geq 1$ the  notation $\Phi^p_{n}$ is reserved for the restriction-projections 
\begin{equation*}
\Phi^p_{m} \maps \S^{m}(V) \to \S^{p}(V') \qquad  \Phi^p_{m}:= \pr_{\S^{p}(V')} \circ \Phi \vert_{\S^{m}(V)} 
\end{equation*}
$\Phi$ is obviously completely determined by its restriction-projection maps $\{\Phi^r_{m}\}$. 
Furthermore, we denote by $\Phi^1 \maps \S(V) \to V'$ the linear map
\[
\Phi^{1} := \Phi^{1}_{1} + \Phi^{1}_{2} + \cdots
\]

We recall that $(\S(V),\rDelta)$ is cofree over $V$ in the category $\cocomu$ (cf.\ Lemma 22.1 in \cite{FHT}). In particular, a degree zero linear map $F^{1} \maps \S(V) \to V'$ uniquely determines  
a coalgebra morphism $F \maps \S(V) \to \S(V')$ via the following formula
\begin{equation} \label{eq:morph_formula1}
\begin{split}
F(v_{1}, \ldots, v_{m}) &=  F^{1}_{m}(v_{1}, \ldots,
 v_{m}) + \sum_{p=1}^{m-1} ~ 
\sum^{k_{1} + k_{2} +
  \cdots + k_{p+1}=m}_{k_{1},k_{2},\hdots,k_{p+1} \geq 1} \sum_{\sigma
  \in \Sh(k_{1},k_{2},\hdots,k_{p+1})} \frac{\epsilon(\sigma)}{(p+1)!}\\
&\quad F^{1}_{k_{1}}(v_{\sigma(1)},
\ldots, v_{\sigma(k_{1})}) 
\odot F^{1}_{k_{2}}(v_{\sigma(k_{1} + 1)}, \ldots,
v_{\sigma(k_{1}+k_{2})}) \odot \cdots \\
& \quad \odot F^{1}_{k_{p+1}}(v_{\sigma(m-k_{p+1} + 1)}, \ldots, v_{\sigma(m)}).
\end{split}
\end{equation}
This gives explicit formulas for the restriction-projections $F^{p}_{m}$:
\begin{equation} \label{eq:morph_formula2}
\begin{split}
F^{p}_{m}(v_{1}, \ldots, v_{m}) &=
\sum^{k_{1} + k_{2} +
  \cdots + k_{p}=m}_{k_{1},k_{2},\hdots,k_{p} \geq 1} \sum_{\sigma
  \in \Sh(k_{1},k_{2},\hdots,k_{p})} \frac{\epsilon(\sigma)}{p!}
F^{1}_{k_{1}}(v_{\sigma(1)}, \ldots, v_{\sigma(k_{1})}) \\
& \quad \odot F^{1}_{k_{2}}(v_{\sigma(k_{1} + 1)}, \ldots, v_{\sigma(k_{1}+k_{2})}) \odot \cdots 
\odot F^{1}_{k_{p}}(v_{\sigma(m-k_{p} + 1)}, \ldots, v_{\sigma(m)}).
\end{split}
\end{equation}
In particular, we have
\begin{equation} \label{eq:morph_formula3}
F^{m}_{m}(v_{1}, \ldots, v_{m}) = F^{1}_{1}(v_{1}) \odot F^{1}_{1}(v_{2}) \odot \cdots \odot F^{1}_{1}(v_m), \quad F^{p}_{m}(v_{1}, \ldots, v_{m})=0 \quad \forall p >m.
\end{equation}
Hence, a coalgebra morphism between cofree conilpotent coalgebras  $F \maps \S(V) \to \S(V')$ is uniquely determined by its \textbf{structure maps} $F^1_k \maps \S^k(V) \to V'$.

\subsubsection*{Composition} \label{sec:comp}
Let $F \maps \S(V) \to \S(V')$  and $G \maps \S(V') \to \S(V'')$ be coalgebra morphisms. It follows from Eqs.\ \ref{eq:morph_formula1}--\ref{eq:morph_formula3} that the composition $GF \maps \S(V) \to \S(V'')$ is the unique coalgebra morphism whose structure maps $(GF)^1_m \maps \S^m(V) \to V''$ are
\begin{equation}\label{eq:comp}
\begin{split}
(GF)^1_m(v_1,\ldots, v_m) & = \sum^m_{p=1} G^1_p F^p_m(v_1,\ldots, v_m)
\end{split}
\end{equation} 

\subsubsection*{Coderivations} \label{sec:coder}
Analogously, we recall that a degree $-1$ linear map $\delta^{1} \maps \S(V) \to V$ uniquely determines a degree $-1$ coderivation $\delta \maps \S(V) \to \S(V)$ via the following formula
\begin{multline} \label{eq:coder_formula1}
\delta_{m}(v_{1}, \ldots, v_{m}) = 
\delta^1_{m}(v_{1},\ldots,v_{m}) + \\
\sum^{m-1}_{i=1} \sum_{\sigma \in \Sh(i,m-i)}
\epsilon(\sigma) \delta^1_{i}(v_{\sigma(1)}, \ldots, v_{\sigma(i)}) \odot v_{\sigma(i+1)} \odot \cdots \odot v_{\sigma(m)}.
\end{multline}
(See, for example, Lemma 2.4 in \cite{Lada-Markl:1995}).
This gives explicit formulas for the restriction-projections:
\begin{equation} \label{eq:coder_formula2}
\begin{split}
\delta^{p}_{m}(v_{1}, \ldots, v_{m}) = \hspace{-.5cm}\sum_{\sigma \in \Sh(m-p+1,p-1)} \hspace{-.5cm}
\epsilon(\sigma) \delta^1_{m-p+1}(v_{\sigma(1)}, \ldots, v_{\sigma(i)}) \odot v_{\sigma(i+1)} \odot \cdots \odot v_{\sigma(m)}.
\end{split}
\end{equation}
In particular, we have
\begin{equation} \label{eq:coder_formula3}
\delta^{m}_{m}(v_{1}, \ldots, v_{m}) = \bigl(\delta^{1}_{1}\bigr)^{\tensor}(v_{1}  \odot v_{2} \cdots \odot v_m), \quad \delta^{p}_{m}(v_{1}, \ldots, v_{m})=0 \quad \forall p >m,
\end{equation}
where $\bigl(\delta^{1}_{1}\bigr)^{\tensor}$ denotes the usual derivation on $S(V)$ induced by the linear map  $\delta^1_1 \maps V \to V$. Hence, a coderivation on a  cofree conilpotent coalgebra is uniquely determined by its \textbf{structure maps} $\delta^1_m \maps \S^m(V) \to V$. 

It follows from Eqs.\
\ref{eq:coder_formula1}--\ref{eq:coder_formula3} that a  degree $-1$ coderivation $\delta$ on $\S(V)$ is a \textbf{codifferential}, i.e., $\delta^2=0$, if and only if
\begin{equation}\label{eq:codiff}
\begin{split}
\sum^m_{k=1} \delta^1_k \delta^k_m(v_1,\ldots, v_m) =0 \quad \forall m \geq 1.
\end{split}
\end{equation} 
Analogously, a coalgebra morphism of the form $F \maps \S(V) \to \S(V')$ lifts to a dg coalgebra morphism $F \maps (\S(V),\delta) \to (\S(V'), \delta')$ if and only if
\begin{equation}\label{eq:dgmap}
\begin{split}
\sum^m_{k=1} \delta'^1_k F^k_m(v_1,\ldots, v_m) =    
\sum^m_{k=1} F^1_k\delta^k_m (v_1,\ldots, v_m) \quad \forall m \geq 1.
\end{split}
\end{equation}

\section{Lie $n$-algebras} \label{sec:Lie_n-Alg}

\subsection{$L_\infty$-algebras and $L_\infty$-morphisms} \label{sec:L_inf}
An \textbf{$L_\infty$-algebra} $(L, \el)$
is a $\Z$-graded vector space $L$ equipped with 
a collection $\el = \{\el_1, \el_2, \el_3, \ldots \}$
of graded skew-symmetric linear map, or ``brackets'',
\begin{equation*} 
\el_k \maps \Lambda^k L \to L, \quad 1 \leq k < \infty 
\end{equation*}
with  $\deg{\el_k} = k-2$, satisfying an infinite sequence of Jacobi-like identities of the form:
\begin{align} \label{eq:Jacobi}
   \sum_{\substack{i+j = m+1, \\ \sigma \in \Sh(i,m-i)}}
  (-1)^{\sigma}\epsilon(\sigma)(-1)^{i(j-1)} l_{j}
   (l_{i}(x_{\sigma(1)}, \dots, x_{\sigma(i)}), x_{\sigma(i+1)},
   \ldots, x_{\sigma(m)})=0.
\end{align}
for all $m \geq 1$ \cite[Def.\ 2.1]{Lada-Markl:1995}. In particular, Eq.\ \ref{eq:Jacobi} implies that $(L, \el_1) \in \Ch$. 
Equivalently, an $L_\infty$-structure on a graded vector space $L$ is a degree $-1$ codifferential $\delta$ on the coalgebra $\bar{S}(\bs L)$. (See, for example, Thm.\ 2.4 in \cite{Lada-Markl:1995}.)
The correspondence between the structure maps 
\[
\delta^1_m \maps \S^{m}(\bs L) \to \bs L
\]
and the brackets is given by the formula
\begin{equation}\label{eq:struc_skew}
\delta^1_{m} = (-1)^{\frac{m(m-1)}{2}} \bs \circ
\el_{m} \circ {(\ds)}^{\tensor m}.
\end{equation}

Let $L$ and $L'$ be graded vector spaces. We recall that there is a one-to-one correspondence 
between collections $f=\{f_1,f_2,\ldots\}$
of skew-symmetric linear maps 
\begin{equation} \label{eq:morphism1}
f_k \maps \Lambda^k L \to L' \quad 1 \leq k < \infty
\end{equation}
with $\deg{f_k} = k-1$, and coalgebra morphisms
\[
F \maps \S(\bs L) \to \S(\bs L')
\]
whose degree $0$ structure maps $F^1_k \maps \S^k(\bs L) \to \bs L$
are given by the formula
\begin{equation} \label{eq:morph_eq1}
 F^{1}_{k} = (-1)^{\frac{k(k-1)}{2}} \bs \circ
 f_{k} \circ {(\ds)}^{\tensor k}.
\end{equation}

A {\bf morphism} (i.e. a weak $L_\infty$-morphism) of $L_\infty$-algebras 
\[
f \maps (L,\el) \to (L',\el')
\]
is a collection $f=\{f_1,f_2,\ldots\}$ of skew-symmetric linear maps as in \eqref{eq:morphism1}
whose corresponding coalgebra morphism \eqref{eq:morph_eq1} 
satisfies Eq.\ \ref{eq:dgmap} and therefore lifts to a morphism of dg coalgebras
\[
F \maps (\S(\bs L), \delta) \to (\S(\bs L'), \delta') 
\] 
We treat the category $\Linf$ of $L_\infty$-algebras  as a full subcategory of $\dgcocomu$. Hence, composition of morphisms in $\Linf$ is given by Eq.\ \ref{eq:comp}.

A morphism $f \maps (L,\el) \to (L',\el')$ is an \textbf{$L_\infty$-isomorphism} iff the linear map $f_1 \maps L \to L'$ is an isomorphism in $\Ch$. It is easy to show that this condition implies that $f$ corresponds to an actual isomorphism in the category $\dgcocomu$.

\begin{notation}
In contrast with some other conventions found in the literature, we will write weak $L_\infty$-morphisms in $\Linf$ using a single lower-case (Latin or Greek) letter, e.g. 
\[
f \maps (L,\el) \to (L',\el'),
\]
and the $k$-ary map in the collection $f$ will always be denoted by $f_k$. The dg coalgebra morphism encoded by the collection $f$ will always be written using the corresponding upper-case letter, e.g.  
$F \maps \S(\bs L) \to \S(\bs L')$.
\end{notation}

\begin{remark} \label{rmk:h0}
By definition, if $f \maps (L,\el) \to (L',\el')$ is an $L_\infty$-morphism, then the coalgebra morphism $F \maps \S(\bs L) \to \S(\bs L')$ satisfies $F \delta = \delta' F$. 
Therefore, by setting $m=1$ in Eq.\ \ref{eq:dgmap}, we observe that the degree 0 map $f_1$
is necessarily a chain map:
\[
f_1 \maps (L,\el_1) \to (L',\el'_1)
\]
Furthermore, if $x,y \in L$, then by setting $m=2$ in Eq.\ \ref{eq:dgmap} we see that
\begin{equation} \label{eq:morphism2.5}
f_1 \bigl ( \el_2(x,y) \bigr ) -  \el'_{2} \bigl( f_1(x),f_1(y) \bigr)
 = \el'_1 f_2(x,y).
\end{equation}

Since the bilinear bracket $\el_2$ on a $L_\infty$-algebra $(L,\el)$
induces a Lie algebra structure on $H_0(L)$,  it follows from 
Eq.\ \ref{eq:morphism2.5} that $H_0(f_{1}) \maps H_0(L) \to H_0(L')$ is a
morphism of Lie algebras. 
\end{remark}

Next, we recall three useful classes of $L_\infty$-morphisms.
\begin{definition} \label{def:quasi-iso}
Let $f \maps (L,\el) \to (L',\el')$ be a morphism of $L_\infty$-algebras. 

\begin{enumerate}
\item We say $f$ is a
\textbf{$L_\infty$-quasi-isomorphism} if the chain map
$f_1 \maps (L,\el_1) \to (L',\el'_1)$ is a
quasi-isomorphism, i.e.\ the induced map on homology
\[
H(f_{1}) \maps H(L) \to H(L')
\]
is an isomorphism.

\item We say $f$ is a \textbf{$L_\infty$-epimorphism} \cite[Def.\ 4.1]{Vallette:2014} 
if the chain map $f_1 \maps (L,\el_1) \to (L',\el'_1)$ is a
surjection in degree $n$ for all $n \in \Z$.


\end{enumerate}
\end{definition}

\subsubsection{Strict $L_\infty$-morphisms} \label{sec:strict}
Finally, recall that a morphism $f \maps (L,\el) \to (L',\el')$ in $\Linf$  is a \textbf{strict $L_\infty$-morphism}  iff $f_k =0$ for all $k \geq 2$. 
If 
\[
f=f_1 \maps (L,\el) \to (L',\el)
\]
is a strict morphism, then it follows from the definition that the restriction-projections \eqref{eq:morph_formula2} of the coalgebra morphism $F$ satisfy
\[
F^k_m =0, \quad \text{if $k \neq m$}.
\]
Combining this with Eq.\ \ref{eq:dgmap}, it follows that every
$k$-ary bracket $\el_k$ is preserved by the chain map $f_1$:
\[
\el'_k \circ f_1^{\tensor k} = f_1 \circ \el_k \quad \text{for all $k \geq 1$}
\] 

\subsection{The category of Lie $n$-algebras} \label{sec:lna}
Let $(L,\el)$ be a $L_\infty$-algebra. If the underlying graded vector space 
$L$ is concentrated in the first $n-1$ non-negative degrees, i.e.\
\[
L= \bigoplus_{i \geq 0}^{n-1} L_i
\] 
then $L$ is called a \textbf{Lie $n$-algebra} \cite[Def.\ 4.3.2]{Baez-Crans:2004}.

For a fixed $n \in \N \cup \{\infty\}$, we denote by $\LnA{n}$ the full subcategory of $\Linf$ 
whose objects are Lie $n$-algebras. Note that if $n < \infty$ then, for degree reasons, $\el_k=0$ for all $k > n+1$. Similarly, 
if $f \maps (L,\el) \to (L',\el')$ is a morphism in $\LnA{n}$, then $f_k =0$ for all $k > n$.   

\begin{example} \label{ex:lna}
We recall a few elementary, but important, examples. 
\begin{enumerate}
\item A Lie 1-algebra is just a Lie algebra. This gives a full and faithful embedding of the category of Lie algebras over $\kk$ into $\LnA{n}$ for any $n$. 

\item We say a Lie $n$-algebra $(L, \el)$ is \textbf{abelian} iff $\el_k =0$ for all $k \geq 2$.
Hence, an abelian Lie $n$-algebra is the same thing as a chain complex concentrated in degrees $0,\ldots,n-1$.

\item Let $\dgla^{<n}$ denote the category of chain Lie algebras whose underlying chain complex
is concentrated in degrees $0,\ldots,n-1$. There is a functor $\dgla^{<n} \to \LnA{n}$
which sends a chain Lie algebra $(L,d,[\cdot,\cdot])$ to the Lie $n$-algebra $(L,\el)$, where $\el_1=d$, $\el_2=[\cdot,\cdot]$ and $\el_k=0$ for all $k >2$. Under this embedding, dgla morphisms are mapped to strict Lie $n$-algebra morphisms.
\end{enumerate}
A simple but non-trivial example of a Lie 2-algebra is the ``string Lie 2-algebra'' \cite[Def.\ 8.1]{Henriques:2008}. See Sec.\ \ref{sec:postnikov} for further discussion.

\end{example}

\begin{proposition} \label{prop:finprod}
The category $\LnA{n}$ is closed under finite products.
Moreover, the forgetful functor $\LnA{n} \to \dgcocom$ creates products.
\end{proposition}
\begin{proof}
Indeed, the categorical product of any two Lie $n$-algebras $(L,\el)$ and $(L',\el')$ is  the Lie $n$-algebra $(L \oplus L', \el \oplus \el')$ where for all $k\geq 1$
\begin{equation*} 
\el_k \oplus \el_k' \bigl( (x_1,x'_1),\ldots, (x_k,x'_k) \bigr):= \bigl( \el_k(x_1,\ldots,x_k),  
\el'_k(x'_1,\ldots,x'_k) \bigr).  
\end{equation*}
Furthermore, the usual projections $\pr \maps L \oplus L' \to L$, $\pr' \maps L \oplus L' \to L'$ lift to strict $L_\infty$-epimorphisms
\[
(L,\el) \xleftarrow{\pr} (L \oplus L', \el \oplus \el') \xto{\pr'} (L',\el')
\]
The product of Lie $n$-algebras then coincides with the product in $\dgcocom$ via the natural isomorphism of graded vector spaces $S(V \oplus V') \cong  S(V) \tensor S( V')$.
\end{proof}

\begin{definition} \label{def:LnA_morphs}
Let $f\maps (L,\el) \to (L',\el')$ be a morphism of Lie $n$-algebras. 
\begin{enumerate}
\item We say $f$ is a 
\textbf{weak equivalence} if the chain map $f_1 \maps (L,\el_1) \to (L',\el'_1)$ is a quasi-isomorphism.

\item We say $f$ is a 
\textbf{fibration} if the chain map $f_1 \maps (L,\el_1) \to (L',\el'_1)$ is surjective in all positive degrees.

\item We say $f$ is an 
\textbf{acyclic fibration} if $f$ is a weak equivalence and a fibration.
\end{enumerate}
\end{definition}
\myspace
Following the standard terminology from deformation theory, let
\[
\tanch \maps \LnA{n} \to \Chain^{\proj}
\]
denote the \textbf{tangent functor}, which is defined by the assignments:
\begin{equation} \label{eq:tan}
\begin{split}
(L,\el) &\mapsto (L, \el_1) \\
(L, \el) \xto{f} (L', \el') ~ &\mapsto  ~ (L, \el_1) \xto{f_1} (L'_1, \el'_1)
\end{split}
\end{equation}
Then it follows from Sec.\ \ref{sec:chproj} that $f \maps (L,\el) \to (L',\el')$ is a weak equivalence (resp.\ fibration) of Lie $n$-algebras if and only if $\tanch{f}$ 
is a weak equivalence (resp.\ fibration) in $\Chain^{\proj}$. We show later in Cor.\ \ref{cor:tanexact} that the tangent functor is an exact functor.

Before proceeding further, we record below some basic facts 
about weak equivalences and fibrations in $\LnA{n}$.
\begin{remark} \label{rmk:acyclic_fibs}
\mbox{}
\begin{enumerate}

\item Clearly, a morphism $f$ in $\LnA{n}$ is a weak equivalence if and only if it is an $L_\infty$-quasi-isomorphism. In particular, a morphism $f \maps \g \to \h$ between Lie algebras is a weak equivalence in $\LnA{1}$ if and only if $f$ is an isomorphism.

\item A morphism $f$ in $\LnA{n}$ is an acyclic fibration if and only if it is both an $L_\infty$-quasi-isomorphism  
and an $L_\infty$-epimorphism. Indeed, if $(L,\el)$ is a Lie $n$-algebra, then $H_0(L) = L_0/ \im \el_1$.

\item If $f \maps (L,\el) \to (L',\el')$ is a weak equivalence, then its corresponding
dg coalgebra map $F \maps \bigl( \S(\bs L), \delta \bigr) \to \bigl( \S(\bs L'), \delta' \bigr)$ 
is a quasi-isomorphism, i.e.\ $H(F) \maps H \bigl( \S(\bs L) \bigr) \xto{\cong} H\bigl( \S(\bs L') \bigr)$. One can verify this by equipping the chain complexes $\bigl(\S(\bs L),\delta \bigr)$ and $\bigl(\S(\bs L'),\delta' \bigr)$ with their canonical ascending filtrations induced by tensor word-length, and then analyzing the associated spectral sequence.

\item The converse of the above statement is false: there exist
morphisms in $\LnA{n}$ which are not weak equivalences, whose corresponding
coalgebra maps are nevertheless quasi-isomorphisms in $\dgcocomu$. 
The following example of such a morphism is probably well-known.

\myindent Let $\g:= \kk e_1 \oplus \kk e_2$ be the 2 dimensional solvable Lie algebra with bracket $[e_1,e_2]=e_1$, and let $\h:=\kk\ti{e}$ be the 1 dimensional abelian Lie algebra. 
Consider the Chevalley-Eilenberg homology of $\g$ and $\h$. We have
\[
\S(\bs \g) = (\kk \bs e_1 \oplus \kk \bs e_2 ) ~ \oplus ~ \kk \, \bs e_1 \!\! \vee \! \bs e_2, 
\qquad \S(\bs \h)= \kk \bs \ti{e}. 
\]
Obviously, $H \bigl(\S(\bs \h) \bigr) \cong \kk \bs \ti{e}$. Let $\delta_\g$ denote the codifferential encoding the Lie bracket of $\g$. Then Eq.\ \ref{eq:coder_formula1} and Eq.\ \ref{eq:struc_skew} imply that
\[
\delta_\g(\bs e_1) = \delta_\g(\bs e_2)=0, \quad \delta_\g(\bs e_1 \vee \bs e_2) = \bs [e_1,e_2]=\bs e_1.
\]
Hence, $H\bigl(\S(\bs \g) \bigr) \cong \kk \bs e_2$. 

\myindent Now consider the map $f \maps \g \to \h$, defined by $f(e_1):=0$, $f(e_2):=\ti{e}$. Clearly $f$ is a Lie algebra morphism, but it is not a weak equivalence (since it is not an isomorphism). Let $F \maps \bigl( \S(\bs \g), \delta \bigr) \to \bigl( \S(\bs \h), 0 \bigr)$ denote the dg coalgebra morphism corresponding to $f$. Since $f$ is strict, it follows from Eq.\ \ref{eq:morph_formula3} and Eq.\ \ref{eq:morph_eq1} that $F(\bs e_1) = 0$, $F(\bs e_1 \vee \bs e_2) = F(\bs e_1)\vee F(\bs e_2)=0$, and
\[
F(\bs e_2) = \bs \ti{e}.
\]
This last equality, combined with the above homology calculations, implies that $F$ is a quasi-isomorphism in $\dgcocomu$.  
\end{enumerate}
\end{remark}

\subsection{Factoring Lie $n$-algebra morphisms} \label{sec:LnA_fact}
We now give an explicit factorization for morphisms in $\LnA{n}$, which we will use to show the existence of path objects. 

\subsubsection*{Factoring chain maps in $\Chain^{\proj}$}
Suppose $f \maps (V,d_V) \to (W,d_W)$ is a morphism of chain complexes.  We will factor $f$ explcitly as $f=p_f \jmath$, where $\jmath$ is an acyclic cofibration and $p_f$ is a fibration
as defined in Sec.\ \ref{sec:chproj}.
Let $(P(W), d_{P(W)}) \in \Chain$ denote the chain complex with underlying graded vector space
\begin{equation} \label{eq:P(W)}
P(W)_0:= \{0\} \oplus W_1, \qquad P(W)_{i}:=W_i \oplus W_{i+1}, \quad \text{for $i \geq 1$},
\end{equation}
with differential
\begin{equation*} 
d_{P(W)}(x,y):=(0,x).
\end{equation*}
Note that the complex $(P(W), d_{P(W)})$ is acyclic. In particular, the linear map 
\begin{equation} \label{eq:hP(W)}
h \maps P(W) \to P(W)[1], \qquad h(x,y):=(y,0)
\end{equation}
is a contracting chain homotopy. Let $\pi \maps P(W) \to W$ denote the degree zero linear map
\begin{equation*} 
\pi(x,y):= x + d_Wy.
\end{equation*}
It is easy to verify that $\pi$ is a chain map and that it is surjective in all positive degrees. We can therefore factor $f$ into an acyclic cofibration followed by a fibration:
\begin{equation} \label{eq:chain_fac}
\begin{split}
V \xto{\jmath} &V \oplus P(W) \xto{p_f} W,\\ 
\jmath(v):=\bigl(v,(0,0) \bigr), & \qquad p_f \bigl( v, (x,y) \bigr):= f(v) + \pi(x,y)\\
\end{split}
\end{equation}

\subsubsection*{Factoring strict maps in $\LnA{n}$}
Next we extend the above factorization in $\Ch^{\proj}$ to strict morphisms in $\LnA{n}$.
\begin{proposition} \label{prop:strict_factor}
Let $f=f_1 \maps (L,\el) \to (L',\el')$ be a strict morphism of Lie $n$-algebras. Then $f$
can be factored in the category $\LnA{n}$ as
\[
(L,\el) \xto{\jmath} (\ti{L},\ti{\el}) \xto{\phi} (L',\el')
\] 
where $\jmath$ is a weak equivalence, and $\phi$ is a fibration in $\LnA{n}$.
\end{proposition}

For the proof, we'll use the following lemma:

\begin{lemma}[see Theorem A.1 \cite{Vallette:2014}]\label{lem:infty-obstruct}
Let $(\ti{L},\ti{\el})$ and $(L',\el')$ be Lie $n$-algebras. Let $m >1$ and suppose 
$\{\Phi^1_i \maps \S^i(\bs \ti{L}) \to \bs L' \}_{1 \leq i \leq m-1}$ is a collection of linear maps satisfying 
\[
\sum_{i=1}^k \delta^{\prime 1}_{i} \Phi^i_k = \sum_{i=1}^k \Phi^1_i \ti{\delta}^{i}_{k}
\qquad  \forall  k \leq  m-1,
\]
where $\ti{\delta}$ and $\delta'$ are the codifferentials encoding the $L_\infty$-structures on $\ti{L}$ and $L'$, respectively. Then:
\begin{enumerate}
\item The linear map $c_m \maps \S^m(\bs \ti{L}) \to \bs L'$ defined as
\[
c_m:= \sum_{k=1}^{m-1} \Phi^1_k \ti{\delta}^k_m - \sum_{k=2}^m \delta^{\prime 1}_k \Phi^k_m 
\]
is a degree $-1$ cycle in the chain complex $\bigl(\Hom  (\S^m(\bs \ti{L}), \bs L' ), \del \bigr)$, where
\[
\del c_m = \delta^{\prime 1}_1 \circ c_m  - (-1)^{\deg{c_m}} c_m \circ \ti{\delta}^m_m.  
\]

\item There exists a linear map $\Phi^1_m \maps \S^m(\bs \ti{L}) \to \bs L'$ such that the collection $\{\Phi^1_1,\ldots \Phi^1_{m-1},\Phi^1_m\}$ satisfies 
\[
\sum_{i=1}^m \delta^{\prime 1}_{i} \Phi^i_m = \sum_{i=1}^m \Phi^1_i \ti{\delta}^{i}_{m}
\]
if and only the homology class $[c_m]$ is trivial.
\end{enumerate} 
\end{lemma}

The lemma can be verified by direct computation, which we leave to the reader. 
Alternatively, both the lemma and Prop.\ \ref{prop:strict_factor}
follow from applying the obstruction theory developed in \cite{Vallette:2014} to weak $L_\infty$-morphisms between algebras over the $\mathsf{Lie}_\infty$ operad. 

\begin{proof}[Proof of Prop.\ \ref{prop:strict_factor}]
Let $f=f_1 \maps (L,\el) \to (L',\el')$ be a strict morphism of Lie $n$-algebras.
Via \eqref{eq:chain_fac}, we factor the chain map $f \maps (L, \el_1) \to (L', \el^{\prime}_1)$ in $\Chain^{\proj}$ as 
\[
L \xto{\jmath} \ti{L} \xto{p_f} L',
\]
where, for brevity, we denote by $\ti{L}$ the chain complex
\[
(\ti{L}, \ti{\el}_1):=\bigl ( L \oplus P(L'), \el_1 \oplus d_{P(L')} \bigr).
\]
We then extend the differential $\ti{\el}_1$ to the following $L_\infty$-structure:
\[
\ti{\el}_k \bigl( (v_1, \bb{v}^{~ \prime}_1), \ldots, (v_k, \bb{v}^{~ \prime}_k) \bigr ):= \bigl(\el_{k}(v_1,\ldots,v_k), (0,0) \bigr), \quad \forall k \geq 2, 
\]
for all $v_i \in L$ and $\bb{v}^{~\prime}_{i}=(x'_i,y'_i) \in P(L')$.  
Note that, by construction, if $L$ and $L'$ are concentrated in degrees $0,\ldots,n-1$, then $\ti{L}$ is as well. Hence, $(\ti{L}, \ti{\el}) \in \LnA{n}$.

It is easy to see that the inclusion of complexes $\jmath \maps L \to \tilde{L}$ lifts to a strict 
$L_\infty$-morphism $\jmath \maps (L,\el) \to (\ti{L},\ti{\el})$ which is, by construction, a weak equivalence in $\LnA{n}$.

Switching to the coalgebra picture, let $\ti{\delta}$ be 
the codifferential on $S(\bs \ti{L}) \cong  S(\bs L) \tensor S(\bs P(L'))$
that encodes the $L_\infty$-structure on $\ti{L}$.
We have the equality
\[
\ti{\delta} = \delta \tensor \id + \id \tensor \hat{\delta},
\]
where $\hat{\delta}$ denotes the extension of the differential $d_{P(L')}$ on $P(L')$ to $\S(\bs P(L'))$, i.e., $\hat{\delta}^1_1 = \bs d_{P(L')} \bs^{-1}$. 
Let $J$ denote the coalgebra map corresponding to the strict $L_\infty$-morphism $\jmath$.
Our goal is to construct a morphism in $\dgcocom$
\[
\Phi \maps \bigl( \S(\bs \ti{L}), \ti{\delta} \bigr) \to \bigl(\S(\bs L'),\delta'\bigr)
\]
that has the following properties:
\begin{enumerate}
\item $\Phi^1_1 = \bs \circ p_f \circ  \bs^{-1}$, 
\item $\Phi J = F$, and hence $\Phi^1_{k} J^k_k =0$ for all $k \geq 2$, since $f$ is strict.
\end{enumerate}
The above implies that the morphism $\Phi$, combined with $J$, will give us the desired factorization of $f$ in $\Linf$.

We construct $\Phi$ by induction. Let $\Phi^1_1 := \bs \circ p_f \circ  \bs^{-1}$. 
Let $m >1$ and suppose  $\{\Phi^1_i \maps \S^i(\bs \ti{L}) \to \bs L' ~ 
\vert ~1 \leq i \leq m-1 \}$
is a collection of linear maps satisfying 
\begin{equation*} 
\sum_{i=1}^k \delta^{\prime 1}_{i} \Phi^i_k = \sum_{i=1}^k \Phi^1_i \ti{\delta}^{i}_{k}
\qquad  \forall  k \leq  m-1,
\end{equation*}
and
\begin{equation} \label{eq:strict_fac2}
\Phi^1_{k} J^k_k =0, \quad  2 \leq k \leq m-1.
\end{equation} 
Part 1 of Lemma \ref{lem:infty-obstruct} implies that the degree $-1$ linear map
\[
c_m \maps \S^m(\bs \ti{L}) \to \bs L', \quad  
c_m:= \sum_{k=1}^{m-1} \Phi^1_k \ti{\delta}^k_n - \sum_{k=2}^m \delta^{\prime 1}_k \Phi^k_m, 
\]
as an element of the chain complex $\Hom_\kk \bigl(\S^m(\bs \ti{L}), \bs L')$,
satisfies 
\[
\del c_m = \delta^{\prime 1}_1 \circ c_m - (-1)^{\deg{c_m}} c_m \circ \ti{\delta}^m_m =0.
\]
Recalling Eq.\ \eqref{eq:comp}, which describes the composition of morphisms in $\LnA{n}$, 
we observe that Eq.\ \ref{eq:strict_fac2} implies that $c_m$ vanishes when restricted to the subspace 
$\im J^m_m$. Hence, $c_m$ descends to cocycle $\ti{c}_m$ in the subcomplex $\bigl(\Hom_{\kk} ( \coker J^m_{m} , \bs L'), \del \bigr)$ where: 
\begin{equation*} 
\coker J^m_{m} \cong  
\bigoplus_{i=0}^{m-1} S^{i}( \bs L) \tensor S^{m-i}( \bs P(L')) =
\bigoplus_{i+j=m} S^{i}( \bs L) \tensor \S^j( \bs P(L')),
\end{equation*} 
and where $\del$ denotes, by slight abuse of notation, the restriction of the differential on the ambient complex $\Hom_\kk \bigl(\S^m(\bs \ti{L}), \bs L')$.

Now let $h \maps P(L') \to P(L')[1]$ be the contracting chain homotopy defined in Eq.\ \ref{eq:hP(W)}. 
By the symmetric version of the  ``tensor trick'' (e.g. \ \cite[Sec.\ 10.3.6]{LoVa}), we extend $h$ to a contracting chain homotopy $H$ on the complex $\bigl(\S( \bs P(L') ), \hat{\delta} \bigr)$. Explicitly, the restriction of $H$ to length $k$ tensors
$H_k \maps \S^k( \bs P(L') ) \to \S^k( \bs P(L') )[1]$ is defined as:
\[
H_{k}(\bb{v}^{~\prime}_1,\ldots,\bb{v}^{~\prime}_k)
:= \sum_{\sigma \in \Sn_k} \sigma^{-1} \cdot \bigl( \id^{\tensor k-1} \tensor \bs h \bs^{-1} \bigr) \sigma \cdot 
(\bb{v}^{~\prime}_1,\ldots,\bb{v}^{~\prime}_k), \qquad \forall ~ \bb{v}^{~\prime}_i \in \bs P(L'),
\]
where $\sigma \cdot$ denotes the left action defined in Eq.\ \ref{eq:Sn_action}. A direct calculation shows that indeed $\id = \hat{\delta} H + H \hat{\delta}$. 
Since $\ti{\delta}^m_m = \sum_{i+j=m} (\delta^{i}_i \tensor \id + \id \tensor \hat{\delta}^j_j)$,
it follows that the map
\[
\id_{S(\bs L)} \tensor H \maps \coker J^{m}_{m} \to \coker J^{m}_{m}[1]
\]
is a contracting chain homotopy for the complex $(\coker J^m_m, \ti{\delta}^m_m)$. Therefore, the
linear map \newline $\ti{c}_m \circ (\id \tensor H) \maps \coker J^m_m \to \bs L'$ satisfies
\begin{equation} \label{eq:triv_cocycle}
\ti{c}_m = -\del \bigl(\ti{c}_m \circ (\id \tensor H) \bigr).
\end{equation}
Finally, we extend $H$ to all of $S(\bs P(L'))$ by declaring $H(1_\kk):=0$, and we let $\Phi^1_m \maps \S(\ti{L}) \to \bs L'$ be the linear map
\[
\Phi^1_m\bigl( (v_1, \bb{v}^{~ \prime}_1), \ldots, (v_k, \bb{v}^{~ \prime}_m) \bigr )  := -c_m \circ (\id \tensor H)\bigl( (v_1, \bb{v}^{~ \prime}_1), \ldots, (v_k, \bb{v}^{~ \prime}_m) \bigr ).
\] 
Hence Eq.\ \ref{eq:triv_cocycle} implies that $c_m = \del \Phi^1_k$. It then follows from 
part 2 of Lemma \ref{lem:infty-obstruct} that the collection $\{\Phi^1_1,\ldots \Phi^1_{m-1},\Phi^1_m\}$
satisfies
\[
\sum_{i=1}^k \delta^{\prime 1}_{i} \Phi^i_k = \sum_{i=1}^k \Phi^1_i \ti{\delta}^{i}_{k}
\qquad  \forall  k \leq  m,
\]
and $\Phi^1_{k} J^k_k =0$ for  $2 \leq k \leq m$. This completes the inductive step, and therefore the proof of the proposition.
\end{proof}

\begin{remark} 
The factorization recalled in the beginning of this section of a chain map $f \maps (V,d_V) \to (W,d_W)$ in $\Chain^{\proj}$ is functorial. Indeed, a commutative diagram in $\Chain$ of the form:
\begin{equation*}
\begin{tikzpicture}[descr/.style={fill=white,inner sep=2.5pt},baseline=(current  bounding  box.center)]
\matrix (m) [matrix of math nodes, row sep=2em,column sep=2em,
  ampersand replacement=\&]
  {  
V \& W \\
V' \& W' \\ 
}; 
  \path[->,font=\scriptsize] 
   (m-1-1) edge  node[auto] {$f$} (m-1-2)
   (m-2-1) edge  node[auto] {$f'$} (m-2-2)
   (m-1-1) edge  node[auto,swap] {$\alpha$} (m-2-1)
   (m-1-2) edge  node[auto] {$\beta$} (m-2-2)
  ;
\end{tikzpicture}
\end{equation*}
factors as
\begin{equation*}
\begin{tikzpicture}[descr/.style={fill=white,inner sep=2.5pt},baseline=(current  bounding  box.center)]
\matrix (m) [matrix of math nodes, row sep=2em,column sep=2em,
  ampersand replacement=\&]
  {  
V \& V \oplus P(W) \& W \\
V' \& V' \oplus P(W') \& W' \\ 
}; 
  \path[->,font=\scriptsize] 
   (m-1-1) edge  node[auto] {$\jmath$} (m-1-2)
   (m-1-2) edge  node[auto] {$p_f$} (m-1-3)
   (m-2-1) edge  node[auto] {$\jmath'$} (m-2-2)
   (m-2-2) edge  node[auto] {$p_{f'}$} (m-2-3)
   (m-1-1) edge  node[auto,swap] {$\alpha$} (m-2-1)
   (m-1-3) edge  node[auto] {$\beta$} (m-2-3)
   (m-1-2) edge  node[auto] {$\gamma$} (m-2-2)
  ;
\end{tikzpicture}
\end{equation*}
where $\gamma \maps V \oplus P(W) \to V' \oplus P(W')$ is the chain map
\[
 \gamma \bigl(v,(x,y) \bigr):= \bigl( \alpha(v), (\beta(x),\beta(y)). 
\]
It would be convenient if the factorization in Prop.\ \ref{prop:strict_factor} could be made functorial in a similar way, perhaps using the symmetrized homotopy $H \maps \S( \bs P(L') ) \to \S( \bs P(L') )[1]$. Then, Thm.\ \ref{thm:LnA_CFO} in Sec.\ \ref{sec:LnA_CFO} would imply that
$\lnaft$ is equipped with a functorial path object. We leave this as an open question.
\end{remark}

\subsubsection*{Factoring arbitrary maps in $\LnA{n}$}
Let $f \maps (L, \el) \to (L',\el')$ be an arbitrary (not necessarily strict)  $L_\infty$-morphism in $\LnA{n}$. Then $f$ can be factored into a weak equivalence composed with a fibration in the following 
way. First, we observe that the diagonal map $\diag \maps (L', \el') \to (L' \oplus L', \el' \oplus \el')$ is a strict $L_\infty$-morphism. Hence, Prop.\ \ref{prop:strict_factor} gives us an explicit factorization of the diagonal into a weak equivalence followed by a fibration
\[
(L', \el') \xto{\jmath}  (L^{\prime I}, \el^{\prime I}) \xto{\phi} (L' \oplus L', \el' \oplus \el').
\]
The Lie $n$-algebra $(L^{\prime I}, \el^{\prime I})$ is a path object
for $(L',\el')$, in the sense of Def.\ \ref{def:cfo}. When combined
with Cor.\ \ref{cor:fib_pback}, which concerns the existence of
pullbacks of (acyclic) fibrations, the existence of a path object for
$(L',\el')$ implies the existence of a factorization\footnote{This is
  Brown's Factorization Lemma. (See Lemma \ref{lem:fact})} of $f$. We
provide more details later in Cor.\ \ref{cor:factor} for morphisms in
$\lnaft$, our main category of interest.

\subsubsection{Strictification of fibrations} \label{sec:stric_fib}
In the last part of this section, we show that every fibration in $\LnA{n}$ can be factored into an isomorphism followed by a strict fibration. We will use this fact repeatedly throughout the paper.
The proof is a modification of a result of Vallette \cite{Vallette:2014} concerning the factorization of ``$\infty$-epimorphisms'' between $\Z$-graded homotopy algebras .

\begin{lemma} \label{lem:strict_fib}
Let $f \maps (L,\el) \to (L',\el')$ be a fibration between Lie $n$-algebras. 
Then there exists a Lie $n$-algebra $(L,\tilde{\el})$ and an 
isomorphism $\phi \maps (L,\tilde{\el}) \xto{\cong} (L,\el)$ such that
\[
f\phi \maps (L,\tilde{\el}) \to (L',\el') 
\]
is a strict fibration with  $f\phi=(f\phi)_1 = f_1$.
\end{lemma}
\begin{proof}
Let $F \maps \S(\bs L) \to \S(\bs L')$ be the coalgebra morphism corresponding to the fibration $f$.
Then the linear map $F^1_1 \maps \bigoplus_{i \geq 1} \bs L_i \to \bigoplus_{i \geq 1} \bs L'_i$ is surjective in all degrees $i \geq 2$. Let $\sigma \maps \bigoplus_{i \geq 1} \bs L'_i \to \bigoplus_{i \geq 1} \bs L_i$ be a map of graded vector space whose restrictions satisfy the equalities
\[
\sigma \vert_{\bs L'_1} =0, \qquad F^1_1 \circ \sigma \vert_{\bigoplus_{i \geq 2} \bs L'_i} = \id.
\]
Let $\Phi^1_1 :=\id \maps \bs L \to \bs L$. Let $m\geq 2$ and suppose we have defined a sequence of degree 0 linear maps $\{\Phi^1_k \maps \S^k(\bs L) \to \bs L\}^{m-1}_{k \geq 1}$. It follows from Eq.\ \ref{eq:morph_formula2} that this sequence gives us well defined linear maps $\Phi^k_m \maps \S^m(\bs L) \to  \S^k(\bs L)$ for $2 \leq k \leq m$. Now define $\Phi^1_m \maps \S^m(\bs L) \to \bs L$ as:
\[
\Phi^1_m = -\sum_{k \geq 2}^m \sigma F^1_k \Phi^k_m. 
\]
This construction inductively yields a coalgebra isomorphism $\Phi \maps \S(\bs L) \to \S(\bs L)$. 
Now let $\delta$ be the codifferential on $\S(\bs L)$ corresponding to the Lie $n$-algebra structure on $L$. We define a new codifferential  $\ti{\delta}:= \Phi^{-1} \delta \Phi$. Clearly, 
we can promote $\Phi$ to an isomorphism  of dg coalgebras
$\Phi \maps \bigl (\S(\bs L), \ti{\delta} \bigr) \to \bigl (\S(\bs L), \delta \bigr)$. And since $\Phi^1_1 = \id$, it follows from Eq.\ \ref{eq:comp} that we have 
\[
(F\Phi)^1_1 = F^1_1 \maps \bs L \to \bs L'.
\]
It remains to show that $F\Phi$ is strict, i.e.\ $(F\Phi)^1_m=0$ for all $m\geq 2$. Eq.\ \ref{eq:comp} and the construction of $\Phi$ imply that
\[
(F\Phi)^1_m = - F^1_1 \bigl (\sum_{k \geq 2}^m \sigma F^1_k \Phi^k_m \bigr)  + \sum_{k \geq 2}^m F^1_k \Phi^k_m. 
\]
If $m \geq 2$, then for any homogeneous element $y \in \S^m (\bs L)$, we have $\deg{y} > 1$. Hence, $\deg{F^1_k \Phi^k_m(y)} >1$ for any $k \geq 1$. It then follows from the definition of $\sigma$ that $F^1_1 \bigl ( \sigma F^1_k \Phi^k_m(y) \bigr) = F^1_k \Phi^k_m(y)$. Therefore, $(F\Phi)^1_m=0$ for all $m\geq 2$. 
\end{proof}

\section{Pullbacks in $\LnA{n}$} \label{sec:pullbacks}
In this section, we give an explicit description of pullbacks of fibrations and acyclic fibrations in the category $\LnA{n}$. After recalling the construction of pullbacks in $\dgcocom$, 
we focus on the special case of pulling back strict fibrations in $\LnA{n}$. We then use 
Lemma \ref{lem:strict_fib} to address the more general non-strict case.

Consider a diagram of Lie $n$-algebras of the form $(L',\el') \xto{g} (L'',\el'') \xleftarrow{f} (L,\el)$. The category $\dgcocom$ is complete, so the pullback of the corresponding diagram of dg coalgebras exists:
\begin{equation} \label{diag:main_pback}
\begin{tikzpicture}[descr/.style={fill=white,inner sep=2.5pt},baseline=(current  bounding  box.center)]
\matrix (m) [matrix of math nodes, row sep=2em,column sep=3em,
  ampersand replacement=\&]
  {  
(\bar{P},\delta_P)  \& \bigl ( \S(\bs L), \delta  \bigr) \\
\bigl ( \S(\bs L'), \delta' \bigr) \& \bigl ( \S(\bs L''), \delta'' \bigr) \\
}; 
  \path[->,font=\scriptsize] 
   (m-1-1) edge node[auto] {$\ppr$} (m-1-2)
   (m-1-1) edge node[auto,swap] {$\ppr'$} (m-2-1)
   (m-1-2) edge node[auto] {$F$} (m-2-2)
   (m-2-1) edge node[auto] {$G$} (m-2-2)
  ;

  \begin{scope}[shift=($(m-1-1)!.4!(m-2-2)$)]
  \draw +(-0.25,0) -- +(0,0)  -- +(0,0.25);
  \end{scope}
\end{tikzpicture}
\end{equation}
The graded coalgebra $\bar{P}$ is the equalizer of the diagram
\[
\begin{tikzpicture}[descr/.style={fill=white,inner sep=2.5pt},baseline=(current  bounding  box.center)]
\matrix (m) [matrix of math nodes, row sep=2em,column sep=2em,
  ampersand replacement=\&]
  {  
\S( \bs L \oplus \bs L') \&
\S (\bs L'') \\
}; 

   \path[->,font=\scriptsize] 
($(m-1-1.east)+(0,0.2)$) edge node[above] {$F\ppr$} ($(m-1-2.west)+(0,0.2)$) 
($(m-1-1.east)+(0,-0.2)$) edge node[below] {$G\ppr'$} ($(m-1-2.west)+(0,-0.2)$) 
;   
\end{tikzpicture}
\]
which can be characterized as the largest sub-coalgebra of 
$\S(\bs L \oplus \bs L')$ contained in the vector space $\ker (F\ppr - G \ppr')$. 
By generalizing a construction of Sweedler \cite[Sec.\ 16.1]{Sweedler:1969}, we have the following explicit description of $\bar{P}$:
\begin{equation} \label{eq:coalg_pullback}
\begin{split}
\bar{P} = \Bigl \{ v \in \ker (F\ppr - G \ppr') ~ \vert ~ 
(\id \tensor F\ppr ) \rDelta(v) = (\id \tensor G\ppr' ) \rDelta(v)  \\ \in  \S(\bs L \oplus \bs L') 
\tensor \S(\bs L'') \Bigr \}.
\end{split}
\end{equation}
Above $\rDelta$ is the reduced comultiplication on $\S(\bs L \oplus \bs L')$. 
Using arguments analogous to those in the proof of Lemma 16.1.1 in \cite{Sweedler:1969}, it is not difficult to show 
that the restriction $\rDelta \vert_P$ gives $\bar{P}$ the structure of a cocommutative coalgebra. The codifferential $\delta_P$ is, of course, the restriction of the codifferential $\delta_\oplus$ on $\S(\bs L \oplus \bs L')$.

\subsection{Strict fibrations} \label{sec:strict_pullback}

The pullback of a strict fibration in $\LnA{n}$ has a convenient explicit description.

\begin{proposition} \label{prop:strict_pullback}
Suppose $f=f_1 \maps (L,\el) \to (L'',\el'')$ is a strict fibration in $\LnA{n}$ 
and $g \maps (L',\el') \to (L'',\el'')$ is an arbitrary morphism between
Lie $n$-algebras.  Let $(\ti{L},\ti{\el}_1) \in \Chain $ 
denote the pullback of the diagram of chain maps
$(L'_1,\el'_1) \xto{g_1} (L''_1,\el''_1) \xleftarrow{f_1} (L,\el_1)$.
\begin{enumerate}
\item The pullback square in $\Chain$ containing $f_1$ and $g_1$ lifts to a commutative diagram 
in $\LnA{n}$:
\begin{equation} \label{diag:strict_pullback}
\begin{tikzpicture}[descr/.style={fill=white,inner sep=2.5pt},baseline=(current  bounding  box.center)]
\matrix (m) [matrix of math nodes, row sep=2em,column sep=3em,
  ampersand replacement=\&]
  {  
(\ti{L},\ti{\el})  \& (L,\el) \\
(L', \el') \& ( L'',\el'') \\
}; 
  \path[->,font=\scriptsize] 
   (m-1-1) edge node[auto] {$$} (m-1-2)
   (m-1-1) edge node[auto,swap] {$$} (m-2-1)
   (m-1-2) edge node[auto] {$f$} (m-2-2)
   (m-2-1) edge node[auto] {$g$} (m-2-2)
  ;

\end{tikzpicture}
\end{equation}

\item Let $(\bar{P},\delta_P)$ denote the pullback of the diagram 
$(\S(\bs L'),\delta') \xto{G} (\S(\bs L''), \delta'') \xleftarrow{F} (\S(\bs L),\delta)$, where $F$ and $G$ are the dg coalgebra morphisms corresponding to $f$ and $g$, respectively.
Then there exists an isomorphism of dg coalgebras
\[
(\bar{P},\delta_P) \xto{\cong} ( \S( \bs \ti{L}), \ti{\delta})
\]
which identifies $(\S( \bs \ti{L}), \ti{\delta})$ as the pullback of $F$ and $G$ in 
$\dgcocom$. Hence, the diagram \eqref{diag:strict_pullback} is a pullback diagram in $\LnA{n}$.
\end{enumerate}
\end{proposition}

Before we prove Prop.\ \ref{prop:strict_pullback}, we will need to discuss a few technical constructions.

\subsubsection*{Useful endomorphisms of the product coalgebra}
As above, consider a diagram of the form  $(L',\el') \xto {g} (L'',\el'') \xleftarrow{f=f_1} (L,\el)$
in $\LnA{n}$, in which $f$ is a strict fibration. In order to give an explicit description of the $L_\infty$-structure on the pullback, we first construct\footnote{To the best of our knowledge, this construction is due to Bruno Vallette. It is implicit in his proof of Thm.\ 4.1 in \cite{Vallette:2014}.} two auxiliary endomorphisms of the graded coalgebra $\S(\bs L' \oplus \bs L)$.

Throughout this section, we denote elements of the direct sum $\bs L' \oplus \bs L$ as vectors  $\bb{v}:=(v',v)$. 
We first give a convenient description of $\bs \ti{L}$, the suspension of the pullback of $f_1$, as a graded vector space. Since $f$ is a fibration, the linear map $F^1_1 \maps \bigoplus_{i \geq 1} \bs L_i \to \bigoplus_{i \geq 1} \bs L''_i$ is surjective in all degrees $i \geq 2$. Let $\sigma \maps \bigoplus_{i \geq 1} \bs L''_i \to \bigoplus_{i \geq 1} \bs L_i$ be a map of graded vector spaces whose restrictions satisfy the equalities
\begin{equation} \label{eq:sigma}
\sigma \vert_{\bs L''_1} =0, \qquad F^1_1 \circ \sigma \vert_{\bigoplus_{i \geq 2} \bs L''_i} = \id.
\end{equation}
Then we have
\begin{equation} \label{eq:Ltilde}
\bs \ti{L}_1 = \bs L'_1 \times_{\bs L''_{1}} \bs L_{1}, \quad \bs \ti{L}_{i \geq 2} = \bs L'_{i \geq 2} \oplus \ker F_1
\end{equation}
and a pullback diagram of graded vector spaces
\begin{equation} \label{diag:strict_pullback2}
\begin{tikzpicture}[descr/.style={fill=white,inner sep=2.5pt},baseline=(current  bounding  box.center)]
\matrix (m) [matrix of math nodes, row sep=2em,column sep=5em,
  ampersand replacement=\&]
  {  
\bs \ti{L}  \& \bs L \\
\bs L' \&  \bs L'' \\
}; 
  \path[->,font=\scriptsize] 
   (m-1-1) edge node[auto] {$\pr + \sigma G_1 \pr'$} (m-1-2)
   (m-1-1) edge node[auto,swap] {$\pr'$} (m-2-1)
   (m-1-2) edge node[auto] {$F_1$} (m-2-2)
   (m-2-1) edge node[auto] {$G_1$} (m-2-2)
  ;

  \begin{scope}[shift=($(m-1-1)!.4!(m-2-2)$)]
  \draw +(-0.25,0) -- +(0,0)  -- +(0,0.25);
  \end{scope}
\end{tikzpicture}
\end{equation}
Clearly, $\S(\bs \ti{L}) \sse \S(\bs L' \oplus \bs L)$ as graded coalgebras. 
We define the linear maps $H^1_k \maps \S^k(\bs L' \oplus \bs L) \to \bs L' \oplus \bs L$ to be
\begin{equation} \label{eq:H}
 H^1_1(\bb{v}):= (v', \sigma G_1(v') + v ), \quad
H^1_{k}(\bb{v}_1,\ldots,\bb{v}_k):= \bigl(0, \sigma G^1_{k}(v'_1,\ldots,v'_k) \bigr).
\end{equation}
Similarly, let $J^1_k \maps \S^k(\bs L' \oplus \bs L) \to \bs L' \oplus \bs L$ denote the linear maps
\begin{equation} \label{eq:J}
 J^1_1(\bb{v}):= (v', -\sigma G_1(v') + v ), \quad
J^1_{k}(\bb{v}_1,\ldots,\bb{v}_k):= \bigl(0, -\sigma G^1_{k}(v'_1,\ldots,v'_k) \bigr).
\end{equation}

\begin{claim}\label{claim:1}
We have the equalities $HJ = JH =\id_{\S(\bs L' \oplus \bs L)}$.
\end{claim}
\begin{proof}
Indeed, using the definition of $\sigma$, a direct computation verifies that $(HJ)^1_1 = H^1_1J^1_1 = \id_{\bs L' \oplus \bs L}$. Now suppose $m \geq 2$. It remains to show $(HJ)^1_m=0$. It follows from Eq.\ \ref{eq:comp} that 
$(HJ)^1_m = \sum_{k=1}^m H^1_kJ^k_m$.
From Eq.\ \ref{eq:morph_formula2}, we see that the formula for $J^k_m$ involves a summation of tensor products of linear maps of the form
\begin{equation} \label{eq:strict_pullback0}
\sum_{j_1 + j_2 + \cdots + j_k =m}J^1_{j_1} \tensor J^1_{j_2} \tensor \cdots  \tensor J^1_{j_k}. 
\end{equation}
Hence, if $k < m$, then in each term of above sum, there exists a $j_i > 1$, 
and so it follows from the definition \eqref{eq:J} of $J$ that $\pr'J^1_{j_i}=0$. 
Combining this observation with the definition \eqref{eq:H} of $H$, we deduce that 
if $2 \leq k < m$, then $H^1_kJ^k_m =0$. Therefore,
\begin{align*}
& (HJ)^1_m (\bb{v}_1,\ldots,\bb{v}_m) = H^1_1J^1_m(\bb{v}_1,\ldots,\bb{v}_m) + H^1_mJ^m_m(\bb{v}_1,\ldots,\bb{v}_m)  \\
&\quad = H^1_1\bigl(0, -\sigma G^1_m(v'_1,\ldots,v'_m) \bigr) + H^1_m \Bigl(
(v'_1, -\sigma G_1(v'_1) + v_1), (v'_2, -\sigma G_1(v'_2) + v_2) \ldots, \\
& \quad \qquad \ldots, (v'_m, -\sigma G_1(v'_m) + v_m) \Bigr) \\
& \quad =\bigl( 0, -\sigma G^1_m(v'_1,\ldots,v'_m) \bigr) + 
\bigl( 0, \sigma G^1_m(v'_1,\ldots,v'_m) \bigr)  = 0.
\end{align*}
Hence, $HJ =\id_{\S(\bs L' \oplus \bs L)}$. The same proof (\textit{mutatis mutandis}) shows that 
$JH =\id_{\S(\bs L' \oplus \bs L)}$, and so the claim has been verified.
\end{proof}

Now let $\delta_{\oplus}$ denote the codifferential on the product $\S(\bs L' \oplus \bs L)$, and define
\[
\ti{\delta}:= J \circ \delta_{\oplus} \circ H \vert_{\S(\bs \ti{L})}.
\]

\begin{claim} \label{claim:2}
$\ti{\delta}$ induces a well-defined codifferential on $\S(\bs \ti{L})$.
\end{claim}
\begin{proof}
Note that if $\im \ti{\delta} \sse \S(\bs \ti{L})$, then
Claim \ref{claim:1} implies that $\ti{\delta}^2=0$, and hence $\ti{\delta}$ is a codifferential.
Therefore, all we need to show is that $\im \ti{\delta}^1_m \sse \bs \ti{L}$ for $m \geq 1$.  
We proceed by considering a few cases. \\

\underline{Case $m=1$}: If $\bb{v} \in \bs \ti{L}_{i \geq 2}$, then 
it follows from Eq.\ \ref{eq:comp} and the definitions of $H$ and $J$ that:
\[
\ti{\delta}^1_1(\bb{v}) = J^1_1 (\delta_{\oplus})^1_1  H^1_1(\bb{v}) =  
\bigl(\delta'^1_1(v'), - \sigma G_{1}(\delta'^1_1(v')) + \delta^1_1\sigma G_1(v') + \delta^1_1(v) \bigr).
\]

If $\deg{\bb{v}}=2$, then $(\delta_{\oplus})^1_1\bb{v} \in \bs L'_1 \times_{\bs L''_{1}} \bs L_{1}$. Therefore, $\sigma G_1(\delta'^1_1(v'))=0$ and so:
\[
F_1 \bigl (- \sigma G_{1}(\delta'^1_1(v')) + \delta^1_1\sigma G_1(v') + \delta^1_1(v) \bigr) = 
F_1 ( \delta^1_1(v)) = \delta''^1_1(G_1(v')) = G_1 ( \delta'^1_1(v')). 
\]
Hence, $\ti{\delta}^1_1(\bb{v}) \in \bs \ti{L}_1$. If $\deg{\bb{v}} > 2$, then $v \in \ker F_1$ and so:
\[
\begin{split}
 F_1 \bigl (- \sigma G_{1}(\delta'^1_1(v')) + \delta^1_1\sigma G_1(v') + \delta^1_1(v) \bigr) &= 
- G_{1}(\delta'^1_1(v') + \delta''^1_1 \bigl(F\sigma G_1(v') + F(v) \bigr)  \\
&= - G_{1}(\delta'^1_1(v') + \delta''^1_1 G_1(v') = 0
\end{split}
\]
Hence, $\ti{\delta}^1_1(\bb{v}) \in \bs \ti{L}_{ i > 1}$.\\

\underline{Case $m \geq 2$:} Let $\bb{v}_1,\ldots, \bb{v}_m \in \bs \ti{L}$. 
It follows from Eq.\ \ref{eq:comp} that
\begin{equation} \label{eq:strict_pullback0.5}
\ti{\delta}^1_m(\bb{v}_1,\ldots, \bb{v}_m)
= \sum_{k \geq 1}^m \sum_{l \geq 1}^k J^1_l (\delta_{\oplus})^l_k  H^k_m(\bb{v}_1,\ldots, \bb{v}_m).
\end{equation}
First suppose that $\deg{(\delta_{\oplus})^l_k  H^k_m\bigr)(\bb{v}_1,\ldots, \bb{v}_m)}=1$. Then, for degree reasons it must be the case that $m=2$, and $\bb{v}_1, \bb{v}_2 \in \bs \ti{L}_1$ are in lowest degree. This implies that
$(\delta_{\oplus})^2_2 H^2_{2}(\bb{v}_1, \bb{v}_2) = (\delta_{\oplus})^2_2(H^1_1(\bb{v}_1) \vee H^1_1(\bb{v}_2))=0$. 
Therefore, by expanding Eq.\ \ref{eq:strict_pullback0.5} we obtain
the equality
\[
\begin{split}
\ti{\delta}^1_2(\bb{v}_1, \bb{v}_2) & = 
J^1_1 \bigl( (\delta_{\oplus})^1_1 H^1_{2}(\bb{v}_1, \bb{v}_2) + (\delta_{\oplus})^1_2 H^2_{2}(\bb{v}_1, \bb{v}_2) \bigr)\\
&= \bigl(0, \delta^1_1\sigma G^1_2(v'_1,v'_2) \bigr) + \bigl( \delta'^1_2(v'_1,v'_2), \delta^1_2(v_1,v_2) \bigr).
\end{split}
\]
Since $F$ corresponds to a strict $L_\infty$-morphism, we have $F_1 \bigl(\delta^1_2(v_1,v_2) \bigr) = 
\delta''^1_2 \bigl(F_1(v_1),F_1(v_2) \bigr)$. Furthermore, since $\bb{v}_1, \bb{v}_2 \in \bs \ti{L}_1$, we have $\bigl(F_1(v_1),F_1(v_2) \bigr) =  \bigl(G_1(v'_1),G_1(v'_2) \bigr)$. From these two equalities, we deduce that
\begin{equation} \label{eq:strict_pullback1}
F_1 \bigl( \delta^1_1\sigma G^1_2(v'_1,v'_2) + \delta^1_2(v_1,v_2) \bigr) = \bigl(
\delta''^1_{2} G^2_2 +  \delta''^1_{1} G^1_2  \bigr)(\bb{v}_1,\bb{v}_2).  
\end{equation}
Since $G$ is a morphism of dg-coalgebras, we have 
\[
\bigl(\delta''^1_{2} G^2_2 +  \delta''^1_{1} G^1_2 \bigr)(\bb{v}_1,\bb{v}_2)  = 
\bigl( G^1_1 \delta'^1_2 + G^1_2 \delta'^2_2 \bigr)(\bb{v}_1,\bb{v}_2) = 
G^1_1 \delta'^1_2 (\bb{v}_1,\bb{v}_2). 
\]
By substituting this last equality into Eq.\ \ref{eq:strict_pullback1}, we conclude that
\[
G_1(\delta'^1_2(v'_1,v'_2)) = F_1 \bigl( \delta^1_1\sigma G^1_2(v'_1,v'_2) + \delta^1_2(v_1,v_2) \bigr),
\] 
and hence $\ti{\delta}^1_2(\bb{v}_1, \bb{v}_2) \in \bs \ti{L}_1$.

Now we consider the remaining sub-case and suppose $\deg{\bb{v}_1\ + \ldots + \bb{v}_m} > 2$.
From Eq.\ \ref{eq:morph_formula2}, we see that the formula for $H^k_m$ consists of a sum of tensor products of linear maps $H^1_{j_1} \tensor H^1_{j_2} \tensor \cdots  \tensor H^1_{j_k}$, 
just as $J^k_m$ involved the summation \eqref{eq:strict_pullback0}. Hence, we can apply the argument used in the verification of Claim \ref{claim:1} to the present case and deduce that 
if $2 \leq l \leq k < m$, then
\[
J^1_l (\delta_{\oplus})^l_k  H^k_m(\bb{v}_1,\ldots, \bb{v}_m) = 0.
\]
Consequently, we have
\begin{equation} \label{eq:strict_pullback2}
\ti{\delta}^1_m(\bb{v}_1,\ldots, \bb{v}_m)
= J^1_1 \Bigl( 
\sum_{k = 1}^{m-1}(\delta_{\oplus})^1_k  H^k_m(\bb{v}_1,\ldots, \bb{v}_m) \Bigr) + 
\sum_{l = 1}^{m}J^1_l (\delta_{\oplus})^l_m  H^m_m(\bb{v}_1,\ldots, \bb{v}_m)
.
\end{equation}
Note that in order to show $\ti{\delta}^1_m(\bb{v}_1,\ldots, \bb{v}_m) \in \bs \ti{L}$, it suffices to verify that 
\[
\pr \ti{\delta}^1_m(\bb{v}_1,\ldots, \bb{v}_m) \in \ker F_1, 
\]
since $\deg{\bb{v}_1\ + \ldots + \bb{v}_m} > 2$.
Let us focus on the first summation on the right hand side of Eq.\ \ref{eq:strict_pullback2}.
It follows from the definition of $\delta_\oplus$ that we have $(\delta_{\oplus})^1_k  H^k_m = \bigl( \delta'^1_{k} \pr'^{\tensor k}H^k_m, \delta^1_k\pr^{\tensor k}H^k_m \bigr)$. Since $k <m$, we have $\pr'^{\tensor k}H^k_m =0$, and  hence
\[
\pr J^1_1 \Bigl(\sum_{k = 1}^{m-1}(\delta_{\oplus})^1_k  H^k_m(\bb{v}_1,\ldots, \bb{v}_m) \Bigr)
= \sum^{m-1}_{ k = 1} \delta^1_k\pr^{\tensor k}H^k_m(\bb{v}_1,\ldots, \bb{v}_m).
\]
Since $F$ corresponds to a strict $L_\infty$-morphism, applying $F_1$ to the right hand side of the above equality gives
\begin{equation} \label{eq:strict_pullback2.5}
\sum^{m-1}_{ k = 1} F_1\delta^1_k\pr^{\tensor k}H^k_m(\bb{v}_1,\ldots, \bb{v}_m) = 
\sum^{m-1}_{ k = 1} \delta''^1_k(F_1\pr)^{\tensor k}H^k_m(\bb{v}_1,\ldots, \bb{v}_m).
\end{equation}
The expansion of $(F_1\pr)^{\tensor k} H^k_m$ using Eq.\ \ref{eq:morph_formula2} involves sums of the following form
\begin{equation*} 
\sum_{j_1 + j_2 + \cdots + j_k =m}(F_1 \pr) H^1_{j_1} 
\tensor (F_1 \pr) H^1_{j_2} \tensor \cdots  \tensor (F_1 \pr) H^1_{j_k}. 
\end{equation*}
If $j_r > 1$ and $\bb{w}_1, \ldots, \bb{w}_{j_r} \in \{ \bb{v}_1,\ldots, \bb{v}_m \}
\sse \bs \ti{L}$, then
\begin{equation*} 
(F_1\pr)H^1_{j_r}(\bb{w}_1, \ldots, \bb{w}_{j_r}) = F_1 \sigma G^1_{j_r}(w'_1, \ldots, w'_{j_r})
=G^1_{j_r}(w'_1, \ldots, w'_{j_r})
\end{equation*}
where the last equality above follows from the fact that $\deg{\bb{w}_1, \ldots, \bb{w}_{j_r}} > 1$.
For the $j_r = 1$ case, if $\bb{v}_i \in \bs \ti{L}_1$, then the definition of $H^1_1$ implies that 
\begin{equation*} 
F_1\pr H^1_1(\bb{v}_i) = F_1( \sigma G_1(v'_i) + v_i) = F_1(v_i) = G_1(v_i). 
\end{equation*}
And if $\bb{v}_i \in \bs \ti{L}_{j \geq 2}$, then $v_i \in \ker F$, which implies that 
\begin{equation} \label{eq:strict_pullback6}
F_1\pr H^1_1(\bb{v}_i) = F_1( \sigma G_1(v'_i)) = G_1(v'_i).
\end{equation}
Therefore, by combining Eqs.\ \ref{eq:strict_pullback2.5} -- \ref{eq:strict_pullback6}, we conclude that
\begin{equation} \label{eq:strict_pullback7}
F_1 \pr J^1_1 \Bigl(\sum_{k = 1}^{m-1}(\delta_{\oplus})^1_k  H^k_m(\bb{v}_1,\ldots, \bb{v}_m) \Bigr)
= \sum_{k = 1}^{m-1} \delta''^1_k G^k_m(v'_1,\ldots, v'_m).
\end{equation}

Now consider the second summation on the right hand side of Eq.\ \ref{eq:strict_pullback2}. 
Equation \ref{eq:coder_formula2} and the definitions of $\delta_\oplus$ and $J^1_{l \geq 2}$ imply that 
\[
\pr \sum^m_{l=2} J^1_l(\delta_{\oplus})^l_m  H^m_m = - \sum^m_{l=2} \sigma G^1_l \delta'^{l}_m \pr'^{\tensor m}H^m_m. 
\]
Hence, 
\[
\pr \sum^m_{l=2} J^1_l(\delta_{\oplus})^l_m  H^m_m (\bb{v}_1,\ldots, \bb{v}_m) = 
- \sum^m_{l=2} \sigma G^1_l \delta'^{l}_m(v'_1,\ldots,v'_m)
\]
Furthermore, it follows directly from the definition of $J^1_1$ that
\[
\begin{split}
\pr J^1_1(\delta_{\oplus})^1_m  H^m_m (\bb{v}_1,\ldots, \bb{v}_m)& =  \Bigl(-\sigma G_1\delta'^1_m \pr'^{\tensor m} H^m_m + \delta^1_m \pr^{\tensor m}H^m_m \Bigr) (\bb{v}_1,\ldots, \bb{v}_m)\\
&= -\sigma G^1_m \delta'^1_m(v'_1,\ldots,v'_m) + \delta^1_m \bigl( \sigma G_1(v'_1) + v_1, \ldots,  
\sigma G_1(v'_m) + v_m \bigr).
\end{split}
\]
We apply $F_1$, and by using the above equalities, we obtain the following:
\begin{equation} \label{eq:strict_pullback8}
\begin{split} 
F_1 \pr \Bigl(\sum_{l = 1}^{m}J^1_l (\delta_{\oplus})^l_m  H^m_m(\bb{v}_1,\ldots, \bb{v}_m) \Bigr)
&= - \sum^m_{l=2} F_1\sigma G^1_l \delta'^{l}_m(v'_1,\ldots,v'_m) -F_1\sigma G^1_1 \delta'^1_m(v'_1,\ldots,v'_m)\\
& \quad + F_1 \delta^1_m \bigl( \sigma G_1(v'_1) + v_1, \ldots, \sigma G_1(v'_m) + v_m \bigr)\\
 = - \sum^m_{l=2} G^1_l \delta'^{l}_m(v'_1,\ldots,v'_m)& - G^1_1 \delta'^1_m(v'_1,\ldots,v'_m)\\
& + \delta''^1_m (F_1)^{\tensor m} \bigl( \sigma G_1(v'_1) + v_1, \ldots, G_1(v'_m) + v_m \bigr)\\
= - \sum^m_{l=1} G^1_l \delta'^{l}_m(v'_1,\ldots,v'_m)& + \delta''^1_m G^m_m (v'_1,\ldots,v'_m)
\end{split}
\end{equation}
To obtain the last two lines above, we used the fact that $F$ is a strict $L_\infty$-morphism, as well as the definition of $\sigma$, and the fact that either $v_i \in \ker F_1$ or $G_1(v'_i)=F_1(v_i)$ for all $i=1,\ldots,m$. Finally, we combine Eq.\ \ref{eq:strict_pullback7} with Eq.\ \ref{eq:strict_pullback8} to obtain
\[
\begin{split}
F_1\pr \ti{\delta}^1_m(v'_1,\ldots, v'_m) = 
\sum_{k = 1}^{m} \delta''^1_k G^k_m(\bb{v}_1,\ldots, \bb{v}_m) 
- \sum^m_{l=1} G^1_l \delta'^{l}_m(v'_1,\ldots,v'_m). 
\end{split}
\]
Therefore, since $G$ is a dg coalgebra morphism, we conclude that $F_1\pr \ti{\delta}^1_m(v'_1,\ldots, v'_m) = 0$. 
This completes the verification of the claim. 
\end{proof}

\subsubsection*{The proof of Proposition \ref{prop:strict_pullback}}

\begin{proof}[Proof of statement (1)] 
Claims \ref{claim:1} and \ref{claim:2} above imply that 
$(\S(\bs \ti{L}), \ti{\delta})$ is a dg coalgebra, and that $H \maps (\S(\bs \ti{L}), \ti{\delta}) \to (\S(\bs L' \oplus \bs L), \delta_\oplus)$  is a dg-coalgebra morphism. 
We wish to show that the following diagram in $\LnA{n} \sse \dgcocom$ commutes:
\begin{equation} \label{diag:strict_pullback3}
\begin{tikzpicture}[descr/.style={fill=white,inner sep=2.5pt},baseline=(current  bounding  box.center)]
\matrix (m) [matrix of math nodes, row sep=2em,column sep=3em,
  ampersand replacement=\&]
  {  
(\S(\bs \ti{L}),\ti{\delta})  \& \bigl ( \S(\bs L), \delta  \bigr) \\
\bigl ( \S(\bs L'), \delta' \bigr) \& \bigl ( \S(\bs L''), \delta'' \bigr) \\
}; 
  \path[->,font=\scriptsize] 
   (m-1-1) edge node[auto] {$\ppr H$} (m-1-2)
   (m-1-1) edge node[auto,swap] {$\ppr'H$} (m-2-1)
   (m-1-2) edge node[auto] {$F$} (m-2-2)
   (m-2-1) edge node[auto] {$G$} (m-2-2)
  ;

\end{tikzpicture}
\end{equation}
It suffices to show that for all $m \geq 1$, the linear maps
\[
(G \ppr' H)^1_m \maps \S^m(\bs \ti{L}) \to \bs L'', \quad
(F \ppr H)^1_m \maps \S^m(\bs \ti{L}) \to \bs L''
\]
are equal. It follows from Eq.\ \ref{eq:comp} that 
\begin{equation} \label{eq:new:pb1}
\begin{split}
(G \ppr' H)^1_m &= \sum_{k=1}^m \sum_{i=1}^k G^1_i \ppr'^i_k H^k_m = \sum_{k=1}^m G^1_k \pr'^{\tensor k} H^k_m\\
(F \ppr H)^1_m &= \sum_{k=1}^m \sum_{i=1}^k F^1_i \ppr^i_k H^k_m = \sum_{k=1}^m F^1_k \pr^{\tensor k} H^k_m
=F^1_1 \pr H^1_m,
\end{split}
\end{equation}
Note that the last equality above follows from the hypothesis that $F$ is strict. 

We first consider the $m=1$ case. Let $\bb{v} \in \bs \ti{L}$. From the definition \eqref{eq:H} of $H$, we have
\[
F^1_1 \pr H^1_1(\bb{v}) = F^1_1( \si G_1(v') + v).
\]
Hence, the commutativity of diagram \eqref{diag:strict_pullback2} implies that 
$F^1_1 \pr H^1_1(\bb{v}) = G^1_1 \pr' H^1_1(\bb{v})$.

Now suppose $m \geq 2$. From Eq.\ \ref{eq:morph_formula2}, we see that the formula for
$H^k_m$ involves a summation of tensor products of linear maps of the form
\[
\sum_{i_1 + i_2 + \cdots + i_k =m}H^1_{i_1} \tensor H^1_{i_2} \tensor \cdots  \tensor H^1_{i_k}. 
\]
Hence, if $k < m$, then in each term of above sum, there exists a $i_r > 1$, 
and then it follows from the definition of $H$ that $\pr'H^1_{i_r}=0$. 
So we deduce that 
\[
(G \ppr' H)^1_m = G^1_m \pr'^{\tensor m} H^m_m.
\]
Therefore, for any $\bb{v}_1,\ldots, \bb{v}_m \in \bs \ti{L}$, we obtain the following 
equalities:
\[
\begin{split}
(G \ppr' H)^1_m(\bb{v}_1,\ldots, \bb{v}_m) &= 
G^1_m \pr'^{\tensor m} H^m_m(\bb{v}_1,\ldots, \bb{v}_m)\\
&= G^1_m \bigl( \pr'H^1_1(\bb{v}_1) ,\pr'H^1_1(\bb{v}_2), \ldots, \pr'H^1_1(\bb{v}_m) \bigr)\\
&=G^1_m(v'_1,v'_2,\ldots,v'_m).
\end{split}
\]
On the other hand, it follows from  Eq.\ \ref{eq:new:pb1} and the definition of $H$ that
\[
\begin{split}
(F \ppr H)^1_m (\bb{v}_1,\bb{v}_2,\ldots,\bb{v}_m) & =  F^1_1 \pr H^1_m(\bb{v}_1,\bb{v}_2,\ldots,\bb{v}_m)\\
&=F^1_1 \sigma G^1_m(v'_1,v'_2,\ldots,v'_m)\\
&=G^1_m(v'_1,v'_2,\ldots,v'_m).\\
\end{split}
\]
Hence, we conclude that \eqref{diag:strict_pullback3} indeed commutes, and by construction,
it lifts the pullback square in $\Chain$  
to the category $\LnA{n}$.
\end{proof}

\begin{proof}[Proof of statement (2)] 
Let $(\bar{P},\delta_P)$ denote the pullback of diagram \eqref{diag:main_pback} in $\dgcocom$ 
with $F \maps \bigl ( \S(\bs L), \delta  \bigr) \to \bigl ( \S(\bs L''), \delta''  \bigr)$ corresponding to a strict fibration in $\LnA{n}$.

Recall from \eqref{eq:coalg_pullback} that $\bar{P}$ is a subcoalgebra of $\S(\bs L' \oplus \bs L)$.
Let 
\[
J \vert_{\bar{P}} \maps \bar{P} \to \S(\bs L' \oplus \bs L)
\]
denote the restriction of the coalgebra morphism $J$ defined in Eq.\ \ref{eq:J}. We claim that $\im J \vert_{\bar{P}} \sse \S(\bs \ti{L})$. Since $\S(\bs L' \oplus \bs L)$ is cofree in $\cocom$, it suffices to check that the image of the linear map $J^1 \vert_P :=  J \vert_P \maps \bar{P} \to \bs L' \oplus \bs L$ is contained in $\bs \ti{L}$. 

Let $\bb{y} \in \bar{P}$ be an element of homogeneous degree. Write $\bb{y}$ as sum of elements of increasing word length, i.e., $\bb{y} = \bb{y}_1 + \bb{y}_2 + \cdots + \bb{y}_r$, with 
$\bb{y}_i \in \S^{n_i}(\bs L' \oplus \bs L)$ and $ 1 = n_1 < n_2 < \cdots < n_r$.   
Suppose $n_r=1$. Then $\bb{y}=\bb{y}_1=(v',v) \in \bs L' \oplus \bs L$, and  $\bb{y} \in \ker (F\ppr - G \ppr')$ implies that $(v',v) \in \bs \ti{L}$. Therefore  $J(\bb{y}) = (v', -\sigma G_1(v') + v) \in \bs \ti{L}$.

For the higher arity case, suppose $n_r >1$. The equality $F \ppr (\bb{y}) = G \ppr'(\bb{y})$ implies that $F^1 \ppr (\bb{y}) = G^1 \ppr'(\bb{y})$. Since $F$ is a strict $L_\infty$-morphism, $F^1_{m \geq 2}=0$. Therefore, we deduce that
\begin{equation} \label{eq:strict_pullback9}
F_1 \pr (\bb{y}_1) = \sum_{i=1}^r G^1_{n_i} \pr'^{\tensor n_i} (\bb{y}_i).
\end{equation}
Note that $\deg{\bb{y}} >1$, since $n_r >1$. Hence, it suffices to show that $\pr J^1(\bb{y}) \in \ker F_1$. It follows directly from the definition of $J$ that
\[
\pr J^1(\bb{y}) = \sum_{i=1}^r \pr J^1_{n_i}(\bb{y}_i) = \pr \bb{y}_1 - \sum^r_{i=1} \sigma G^1_{n_i} \pr'^{\tensor n_i}(\bb{y}_i).   
\]  
Applying $F_1$ to the above gives
\[
F_1\pr J^1(\bb{y}) = F_1 \pr \bb{y}_1 - \sum^r_{i=1} G^1_{n_i} \pr'^{\tensor n_i}(\bb{y}_i).
\]
It then follows from  Eq.\ \ref{eq:strict_pullback9} that $F_1\pr J^1(\bb{y})=0$.

Thus, $J \vert_{\bar{P}} \maps \bar{P} \to \S(\bs \ti{L})$ is a well-defined coalgebra morphism. Since $\delta_P = \delta_{\oplus} \vert_P$ and $HJ=\id$, the morphism $J \vert_P$ is compatible with the differentials. Furthermore, the following diagram in $\dgcocom$ commutes:
\[
\begin{tikzpicture}[descr/.style={fill=white,inner sep=2.5pt},baseline=(current  bounding  box.center)]
\matrix (m) [matrix of math nodes, row sep=2em,column sep=2em,
  ampersand replacement=\&]
  {  
(\bar{P}, \delta_P) \& [-1 cm] \& \\ [-0.5cm]
\& (\S(\bs \ti{L}),\ti{\delta})  \& \bigl ( \S(\bs L), \delta  \bigr) \\
\& \bigl ( \S(\bs L'), \delta' \bigr) \& \bigl ( \S(\bs L''), \delta'' \bigr) \\
}; 
  \path[->,font=\scriptsize] 
  
  (m-1-1) edge node[auto] {$J$} (m-2-2)
  (m-2-2) edge node[auto] {$\ppr H$} (m-2-3)
  (m-2-2) edge node[auto,swap] {$\ppr'H$} (m-3-2)
  (m-2-3) edge node[auto] {$F$} (m-3-3)
  (m-3-2) edge node[auto] {$G$} (m-3-3)
  (m-1-1) edge [bend left=20]  node[auto] {$\Pr$} (m-2-3)
  (m-1-1) edge [bend right=40]  node[auto,swap] {$\Pr'$} (m-3-2)
  ;

\end{tikzpicture}
\]

To see that $J\vert_{\bar{P}}$ is an isomophism, note that the commutative diagram \eqref{diag:strict_pullback3} implies that the sub-coalgebra $H(\S(\bs \ti{L})) \sse
\S(\bs L' \oplus \bs L)$ is contained in the vector space $\ker (F\ppr - G \ppr')$. Therefore,
the universal property of the coalgebra $\bar{P}$ implies that $H(\S(\bs \ti{L})) \sse \bar{P}$, 
and it then follows by Claim \ref{claim:1} that $H \vert_{\S(\bs \ti{L})}$ is the inverse of $J\vert_{\bar{P}}$. 

Hence, we conclude that $(\S(\bs \ti{L}), \ti{\delta})$ is a pullback in $\dgcocom$ and therefore a pullback in $\LnA{n}$. This completes the proof of statement (2) of the proposition.
\end{proof}



\subsection{Pullbacks of fibrations and acyclic fibrations} \label{sec:fib_pback}
\begin{cor} \label{cor:fib_pback}
Suppose $f \maps (L,\el) \to (L'',\el'')$ is a (acyclic) fibration in $\LnA{n}$ 
and $g \maps (L',\el) \to (L'',\el'')$ is an arbitrary morphism between
Lie $n$-algebras. Then the pullback of the diagram
\[
(L',\el'_k) \xto {g} (L'',\el''_k) \xleftarrow{f} (L,\el_k)
\]
exists in $\LnA{n}$, and the morphism induced by the pullback of $f$ along $g$ is a (acyclic) 
fibration.
\end{cor}
\begin{proof}
Suppose $f$ is a (acyclic) fibration. Lemma \ref{lem:strict_fib} implies that
there exists a Lie $n$-algebra $(L,\hat{\el})$ and an 
isomorphism $\phi \maps (L,\hat{\el}) \xto{\cong} (L,\el)$ such that
$f\phi \maps (L,\hat{\el}) \to (L'',\el'')$ is a strict (acyclic) fibration with  $f\phi=(f\phi)_1 = f_1$. It follows from Prop.\ \ref{prop:strict_pullback}, that there exists 
a pullback diagram in $\LnA{n}$ of the form
\begin{equation*} 
\begin{tikzpicture}[descr/.style={fill=white,inner sep=2.5pt},baseline=(current  bounding  box.center)]
\matrix (m) [matrix of math nodes, row sep=2em,column sep=1em,
  ampersand replacement=\&]
  {  
(\ti{L},\ti{\el})  \& (L,\hat{\el}) \\
(L', \el') \& ( L'',\el'') \\
}; 
  \path[->,font=\scriptsize] 
   (m-1-1) edge node[auto] {$$} (m-1-2)
   (m-1-1) edge node[auto,swap] {$\widetilde{f\phi}$} (m-2-1)
   (m-1-2) edge node[auto] {$f\phi$} (m-2-2)
   (m-2-1) edge node[auto] {$g$} (m-2-2)
  ;

  \begin{scope}[shift=($(m-1-1)!.4!(m-2-2)$)]
  \draw +(-0.25,0) -- +(0,0)  -- +(0,0.25);
  \end{scope}
\end{tikzpicture}
\end{equation*}
Moreover, Prop.\ \ref{prop:strict_pullback} implies that  the image of above diagram under the tangent functor \eqref{eq:tan} is the pullback diagram of  $(f\phi)_1$ along $g_1$ in $\Chain$. Hence,  $\tanch( \widetilde{f \phi})$ is a (acyclic) fibration in $\Chain^{\proj}$, and therefore  $\widetilde{f \phi}$ is a (acyclic) fibration in $\LnA{n}$.

Now let $F$, $G$, $\Phi$ denote the dg coalgebra morphisms corresponding to $f$, $g$, and $\phi$, respectively. Let $(\bar{C}, \delta_C)$ denote the dg coalgebra witnessing the pullback of $F$ along $G$. The second statement of Prop.\ \ref{prop:strict_pullback}, combined with the pasting lemma for pullbacks, implies that the following diagram in $\dgcocom$ commutes:
\[
\begin{tikzpicture}[descr/.style={fill=white,inner sep=2.5pt},baseline=(current  bounding  box.center)]
\matrix (m) [matrix of math nodes, row sep=2em,column sep=3em,
  ampersand replacement=\&]
  {  
(\S(\bs \tilde{L}),\ti{\delta}) \& (\S(\bs {L}), \hat{\delta}) \\
(\bar{C},\delta_C)  \& (\S(\bs L), \delta) \\
(\S(\bs L'),\delta') \& (\S(\bs L''), \delta'')\\
}; 
  \path[->,font=\scriptsize] 
   (m-1-1) edge node[auto] {$$} (m-1-2)
   (m-1-1) edge node[auto] {$\widetilde{\Phi}$} node[sloped,below] {$\cong$} (m-2-1)
   (m-1-2) edge node[auto] {$\Phi$} node[sloped,below] {$\cong$} (m-2-2)
   (m-2-1) edge node[auto] {$$} (m-2-2)
   (m-2-1) edge node[auto,swap] {$\widetilde{F}$} (m-3-1)
   (m-2-2) edge node[auto] {$F$} (m-3-2)
   (m-3-1) edge node[auto] {$G$} (m-3-2)
  ;

  \begin{scope}[shift=($(m-1-1)!.4!(m-2-2)$)]
  \draw +(-0.25,0) -- +(0,0)  -- +(0,0.25);
  \end{scope}

  \begin{scope}[shift=($(m-2-1)!.4!(m-3-2)$)]
  \draw +(-0.25,0) -- +(0,0)  -- +(0,0.25);
  \end{scope}
\end{tikzpicture}
\]
Since $\widetilde{\Phi}$ is an isomorphism, $\bar{C}$ is the cofree coalgebra in $\cocom$ 
cogenerated by $\widetilde{\Phi}(\bs \ti{L})$. Hence, $(\bar{C}, \delta_C)$ is also a Lie $n$-algebra and $\widetilde{F}$ is an (acyclic) fibration. 
\end{proof}

We end this section with the corollary below, whose proof follows immediately from the construction of the pullbacks in Prop.\ \ref{prop:strict_pullback} and Cor.\ \ref{cor:fib_pback}.

\begin{cor} \label{cor:tan_pullback}
Let $f \maps (L,\el) \to (L',\el')$ be a fibration in $\LnA{n}$. Then
the tangent functor \eqref{eq:tan} maps pullback squares in $\LnA{n}$ of the form
\[
\begin{tikzpicture}[descr/.style={fill=white,inner sep=2.5pt},baseline=(current  bounding  box.center)]
\matrix (m) [matrix of math nodes, row sep=2em,column sep=3em,
  ampersand replacement=\&]
  {  
(\ti{L},\ti{\el})  \& (L,\el) \\
(L',\el' ) \& ( L'',\el'') \\
}; 
  \path[->,font=\scriptsize] 
   (m-1-1) edge node[auto] {$$} (m-1-2)
   (m-1-1) edge node[auto,swap] {$$} (m-2-1)
   (m-1-2) edge node[auto] {$f$} (m-2-2)
   (m-2-1) edge node[auto] {$g$} (m-2-2)
  ;

  \begin{scope}[shift=($(m-1-1)!.4!(m-2-2)$)]
  \draw +(-0.25,0) -- +(0,0)  -- +(0,0.25);
  \end{scope}
\end{tikzpicture}
\]
to pullback squares in $\Chain$.
\end{cor}

\section{$\lnaft$ as a category of fibrant objects} 
\label{sec:LnA_CFO}
Let $(L, \el)$ be a Lie $n$-algebra with underlying graded vector space $L=\bigoplus^{n-1}_{i \geq 0} L_i$.
If each $L_i$ is finite-dimensional, we say $(L,\el)$ is \textbf{finite type}. For a fixed $n \in \N \cup \{\infty\}$, we denote by $\LnA{n}^{\ft}$ the full subcategory of $\Linf$ whose objects are finite type Lie $n$-algebras.

\begin{definition}[Sec.\ 1 \cite{Brown:1973}] \label{def:cfo}
Let $\C$ be a category with finite products, with terminal object $\ast
\in \C$, and equipped with two classes of morphisms called
\textit{weak equivalences} and \textit{fibrations}. A morphism which
is both a weak equivalence and a fibration is called an
\textit{acyclic fibration}. Then $\C$ is a
{\bf category of fibrant objects (CFO)} for a homotopy theory iff:
\begin{enumerate}
\item{Every isomorphism in $\C$ is an acyclic fibration.}

\item{The class of weak equivalences satisfies  ``2 out of 3''. That is, if
    $f$ and $g$ are composable morphisms in $\C$ and any two of $f,g, g
    \circ f$ are weak equivalences, then so is the third.}

\item{The composition of two fibrations is a fibration.}

\item{The pullback of a fibration exists, and is a fibration.
That is, if $Y \xto{g} Z \xleftarrow{f} X$ is a diagram in $\C$ with $f$
    a fibration, then the pullback $X \times_{Z} Y$ exists, and
   the induced projection $X \times_{Z} Y \to Y$ is a  fibration.}

\item{The pullback of an acyclic fibration exists, and is an acyclic fibration.
 }

\item{For any object $X \in \C$ there exists a (not necessarily
    functorial) \textbf{path object}, that is, an object
    $X^{I}$ equipped with morphisms
\[
X \xto{s} X^{I} \xto{(d_0,d_1)} X \times X,
\]
such that $s$ is a weak equivalence, $(d_0,d_1)$ is a fibration, and their
composite is the diagonal map.}

\item{All objects of $\C$ are \textbf{fibrant}. That is, for any $X \in \C$ the unique map 
$ X \to \ast$ is a fibration.}
\end{enumerate}
\end{definition}

We now prove the main result of the paper.

\begin{theorem} \label{thm:LnA_CFO}
Let $n \in \N \cup \{\infty\}$.  The category $\LnA{n}^{\ft}$ of finite type Lie $n$-algebras 
and weak $L_\infty$-morphisms has the structure of a category of fibrant objects, in which the 
weak equivalences and fibrations are those morphisms that satisfy the defining criteria given in Def.\ \ref{def:LnA_morphs}. That is, a morphism $f \maps (L,\el) \to (L',\el')$ is
\begin{itemize}
\item a weak equivalence iff $f$ is a  $L_\infty$ quasi-isomorphism,
\item a fibration iff  the associated chain map $f_1 \maps (L,\el_1) \to (L',\el'_1)$ is a surjection in all {positive} degrees.
\end{itemize}
\end{theorem}

\begin{proof}
We begin by noting that Prop.\ \ref{prop:finprod} implies that $\lnaft$ has finite products.
Next, it follows immediately from the definition of weak equivalences and fibrations that: every isomorphism is an acyclic fibration, the weak equivalences satisfy ``2 out of 3'', the composition of two fibrations is again a fibration, and that 
the trivial map $(L,\el) \to 0$ is a fibration for any $(L,\el) \in \lnaft$. Hence
axioms 1,2,3, and 7 in Def. \ref{def:cfo} for a CFO are satisfied.

To verify axiom 4, suppose $(L',\el') \xto {g} (L'',\el'') \xleftarrow{f} (L,\el)$ is a diagram in $\lnaft$ and $f$ is a fibration. It follows from Cor.\ \ref{cor:fib_pback} that the pullback $(\ti{L},\ti{\el})$ of the diagram exists in $\LnA{n}$ and that the morphism induced by the pullback is a fibration. Cor.\ \ref{cor:fib_pback} also implies that the underlying complex of $(\ti{L},\ti{\el})$ is the pullback of the diagram $\tanch \Bigl ( (L',\el') \xto {g} (L'',\el'') \xleftarrow{f} (L,\el) \Bigr)$ in $\Chain$. Hence, $(\ti{L},\ti{\el})$ is clearly finite type, and axiom 4 is satisfied. The same argument also verifies axiom 5.

Finally, recall that any diagonal map $\diag \maps (L,\el) \to (L \oplus L, \el \oplus \el)$ is a strict morphism in $\lnaft$. Hence,  Prop.\ \ref{prop:strict_factor} implies that $\diag$ has a factorization $(L, \el) \xto{\jmath}  (\ti{L}, \ti{\el}) \xto{\phi} (L \oplus L, \el \oplus \el)$
in the category $\LnA{n}$, in which $\jmath$ is a weak equivalence and $\phi$ is a fibration.
Recall from the proof of Prop.\ \ref{prop:strict_factor} that $\ti{L}=\bigl ( L \oplus P(L\oplus L) \bigr)$, where $P(L\oplus L)$ is the graded vector space defined in \eqref{eq:P(W)}. Therefore, $\ti{L}$ is finite type, and $(\ti{L}, \ti{\el})$ is a path object for $(L,\el)$ in $\lnaft$. Hence, 
axiom 6 in Def. \ref{def:cfo} is satisfied, and this completes the proof.
\end{proof}

\begin{remark} \label{rmk:inf_dim}
Clearly, the same proof shows that the category $\LnA{n}$ is also a category of fibrant objects with 
the same weak equivalences, fibrations, and path objects as $\lnaft$.
\end{remark}

We can now complete the discussion that we started in Sec.\ \ref{sec:LnA_fact} concerning the factorization of arbitrary weak $L_\infty$-morphisms in $\lnaft$. 
Let us recall Brown's Factorization Lemma:
\begin{lemma}[Sec.\ 1 \cite{Brown:1973}] \label{lem:fact} 
Let $\C$ be a category of fibrant objects. Let $f\maps X\to Y$ be a morphism in $\C$, and let $Y^I$ be a path object for $Y$. Then $f$ can be factored as 
\[
X \xto{\jmath} X \times_Y Y^I \xto{\phi_f} Y
\]
where $\phi_f$ is a fibration, and $\jmath$ is a 
weak equivalence which is a section (right inverse) of an acyclic fibration. 
\end{lemma}
The morphisms $\jmath$ and $\phi_f$ in the lemma can easily be expressed in terms of the maps  
$Y \xto{s} Y^{I} \xto{(d_0,d_1)} Y \times Y$ that appear in factorization of the diagonal. See, for example, Sec.\ 2.1 of \cite{Rogers-Zhu:2018} for details. 

Hence, Thm.\ \ref{thm:LnA_CFO} and the factorization lemma imply the following:
\begin{cor}\label{cor:factor}
Let $f \maps (L,\el) \to (L',\el')$ be a weak $L_\infty$-morphism between finite type Lie $n$-algebras. Then $f$ can be factored in the category $\lnaft$ as
\[
(L,\el) \xto{\jmath} (\ti{L},\ti{\el}) \xto{p_f} (L',\el')
\] 
where $\jmath$ is a weak equivalence, and $p_f$ is a fibration in $\LnA{n}$.
\end{cor}

The CFO structure on Lie $n$-algebras can be thought of as a lift of the projective CFO structure (Sec.\ \ref{sec:chproj}) on $\Chain$ via the tangent functor \eqref{eq:tan} in the following sense:

\begin{definition}[Def.\ 2.3.3 \cite{BHH:2017}] \label{def:exact_functor}
A functor $F \maps \C \to \cD$ between categories of fibrant objects is a \textbf{(left) exact functor} iff
\begin{enumerate}
\item $F$ preserves the terminal object, fibrations, and acyclic fibrations.
\item Any pullback square in $\C$ of the form
\[
\begin{tikzpicture}[descr/.style={fill=white,inner sep=2.5pt},baseline=(current  bounding  box.center)]
\matrix (m) [matrix of math nodes, row sep=2em,column sep=3em,
  ampersand replacement=\&]
  {  
P \& X \\
Z \& Y \\
}
; 
  \path[->,font=\scriptsize] 
   (m-1-1) edge node[auto] {$$} (m-1-2)
   (m-1-1) edge node[auto,swap] {$$} (m-2-1)
   (m-1-2) edge node[auto] {$f$} (m-2-2)
   (m-2-1) edge node[auto] {$$} (m-2-2)
  ;

  \begin{scope}[shift=($(m-1-1)!.4!(m-2-2)$)]
  \draw +(-0.25,0) -- +(0,0)  -- +(0,0.25);
  \end{scope}
\end{tikzpicture}
\]
in which $f \maps X \to Y$ is a fibration in $\C$, is mapped by $F$ to a pullback square in $\cD$.
\end{enumerate}
\end{definition}
\begin{remark}
Note that axiom 1 in the above definition combined with Lemma \ref{lem:fact} implies that an exact functor $F \maps \C \to \cD$ between CFOs sends weak equivalences to weak equivalences.
\end{remark}

\begin{cor} \label{cor:tanexact}
The tangent functor $\tanch \maps \lnaft \to \Chain^{\proj}$ is an exact functor between categories of fibrant objects.
\end{cor}
\begin{proof}
Theorem \ref{thm:LnA_CFO} and Cor.\ \ref{cor:tan_pullback}.
\end{proof}

\subsection*{Comparison with the Vallette CFO structure on $\Linf$} \label{sec:vallette_cfo}
In \cite{Vallette:2014}, Vallette showed that the category $\Linf$ of $\Z$-graded $L_\infty$-algebras
is exactly the subcategory of bifibrant objects in the Hinich model structure \cite{Hinich:2001} on $\dgcocomu$. This result implies the following theorem:
\begin{theorem}[Thm.\ 2.1 \cite{Vallette:2014}] \label{thm:vallette_cfo}
\mbox{}
\begin{enumerate}
\item The category $\Linf$ of $\Z$-graded $L_\infty$-algebras and weak $L_\infty$-morphisms 
has the structure of a category of fibrant objects in which a morphism $$f \maps (L,\el) \to (L',\el')$$ is a weak equivalence iff it is a  $L_\infty$-quasi-isomorphism, and a fibration iff it is a  $L_\infty$-epimorphism (Def.\ \ref{def:quasi-iso}). 

\item Every $L_\infty$-algebra $(L,\el)$ has a functorial path object whose underlying graded vector space is 
\[
\bigl( L \tensor \Omega^{\ast}_{\poly}(\Delta^1) \bigr)_m \cong L_{m}[t] \oplus L_{m+1}[t]dt
\]
where $\Omega^{\ast}_{\poly}(\Delta^1)$ denotes the commutative dg algebra of polynomial de Rham forms on the geometric 1-simplex.
\end{enumerate}
\end{theorem} 

Let us make a few comparisons between the CFO structure on $\lnaft$ given in Thm.\ \ref{thm:LnA_CFO} and Vallette's CFO structure on $\Linf$. 

First, the weak equivalences in $\lnaft$ and $\Linf$ obviously coincide. Furthermore, it follows from statement 2 in Remark \ref{rmk:acyclic_fibs} that the acyclic fibrations also coincide. 
If $(L,\el) \in \LnA{n}$, then the path object $(L \tensor \Omega^{\ast}_{\poly}(\Delta^1), \el^{\Omega})$
for $(L,\el) \in \Linf$ is not a Lie $n$-algebra. And the 
degree zero truncation of $(L \tensor \Omega^{\ast}_{\poly}(\Delta^1)$ is certainly not finite type.
However, Vallette's Prop.\ 3.3 in \cite{Vallette:2014} implies that $\Linf$ is also equipped with a functorial cylinder object. If $(L,\el)$ is a finite type Lie $n$-algebra, 
then the degree zero truncation of the cylinder object $(L \tensor J, \ti{\el})$ is also finite type. It would be interesting to work out further details and show that the CFO structure on $\lnaft$ satisfies Brown's additional axioms (F) and (G) in \cite[Sec.\ 6]{Brown:1973}. This would imply that $\lnaft$ is ``almost'' a model category, in the sense of Vallette \cite[Sec.\ 4.1]{Vallette:2014}.


\section{Maurer--Cartan sets} \label{sec:MCfunc}
In this section, we analyze Maurer-Cartan (MC) sets of $\Z$-graded $L_\infty$-algebras constructed by tensoring Lie $n$-algebras with bounded commutative dg algebras. 
This construction naturally arises when studying formal deformation problems in characteristic zero. It also appears, as mentioned in the introduction, in the definition of the spatial realization functor for chain Lie algebras. The smooth analog of the Maurer-Cartan set is featured in the defintion
of Henriques' integration functor \cite{Henriques:2008} for Lie $n$-algebras. 

We begin by recalling some basic facts about Maurer-Cartan elements from \cite[Sec.\ 2]{DR:2017} and references within.
 We note that the $L_\infty$-algebras in \cite{DR:2017} are assumed to be complete and filtered, which is not the case in this paper. However, the particular results that we recall below will still hold 
since all $L_\infty$-algebras involved are sufficiently ``tame'' in the following sense:
\begin{definition} \label{def:tame}
A $\Z$--graded $L_\infty$-algebra $(L, \el)$ is \textbf{tame} if there exists an $N \geq 1$ such that for all $k \geq N$
\[
\el_k(x_1, \ldots, x_k) = 0 \quad \forall x_1, \ldots x_k \in L_{-1}.
\] 
A (weak) $L_\infty$-morphism $f \maps (L, \el) \to (L',\el')$ between tame $L_\infty$-algebra 
is a \textbf{tame morphism} if there exists an $N \geq 1$ such that for all $k \geq N$
\[
f_k(x_1, \ldots, x_k) = 0 \quad \forall x_1, \ldots x_k \in L_{-1}.
\]   
\end{definition}
For trivial reasons, every Lie $n$-algebra is tame and every morphism between Lie $n$-algebra is tame for any $n \in \N \cup \{\infty\}$.
Given a tame $\Z$--graded $L_\infty$-algebra $(L,\el)$, the \textbf{curvature} $\curv \maps L_{-1} \to L_{-2}$ is the function
\begin{equation} \label{eq:curv}
\begin{split}
\curv(a) & := \bs^{-1} \sum_{k \geq 1} \frac{1}{k!} \delta^1_k \bigl(\bs a, \bs a, \ldots, \bs a)\\ 
& = \ell_1(a) + \sum_{k \geq 2} \sgn{k} \frac{1}{k!} \ell_k(a,a,\ldots,a) \in L_{-2}.
\end{split}
\end{equation}
where $\delta$ is the corresponding codifferential on $\S(\bs L)$. For any $a \in L_{-1}$, the expression $\exp(\bs a) -1$ is a well-defined element of the completion of $\S(\bs L)$ defined by the corresponding formal power series. By extending $\delta$ in the natural way, a straightforward calculation shows that
\begin{equation} \label{eq:exp1}
\delta \bigl( \exp(\bs a) - 1 \bigr) = \exp( \bs a) (\bs \curv(a))
\end{equation}  
and hence
\begin{equation} \label{eq:exp2}
\bs \curv(a) = \pr_{\bs L} \circ \delta \bigl( \exp(\bs a) - 1 \bigr),
\end{equation}  
where $\pr_{\bs L}$ denotes the canonical projection to the vector space $\bs L$.
The \textbf{Maurer--Cartan elements} of $L$ are the elements of the subset
\[
\MC(L):= \bigl \{ x \in L_{-1} ~\vert ~ \curv(x)=0 \bigr \}.
\]

Similarly, let $f \maps (L, \el_k) \to (L',\el'_k)$ be a tame morphism.
Such a morphism induces a function $f_\ast \maps L_{-1} \to L'_{-1}$ defined as:
\begin{equation} \label{eq:Fstar}
\begin{split}
f_\ast(a) &:= \bs^{-1} \sum_{k \geq 1} \frac{1}{k!} F^1_k( \bs a, \bs a, \ldots, \bs a) \\
& = f_1(a) + \sum_{k \geq 2} \sgn{k} \frac{1}{k!} f_k(a,a,\ldots,a).
\end{split}
\end{equation}
As in Eq.\ \ref{eq:exp1}, we extend $F$ to the completions of $\S(\bs L)$ and $\S(\bs L')$, and 
a straightforward calculation shows that
\begin{equation} \label{eq:exp3}
F \bigl( \exp(\bs a) - 1 \bigr) = \exp( \bs f_\ast(a)) - 1.
\end{equation}  
\begin{remark} \label{rmk:coalg_tame}
\mbox{}
\begin{enumerate}
\item If $f \maps (L,\el) \to (L',\el')$ and $g \maps (L',\el') \to (L'',\el'')$ are tame morphisms, then it follows from the composition formula \eqref{eq:comp} for $L_\infty$-morphisms that  $gf \maps (L,\el) \to (L'',\el'')$ is also tame. Indeed, if $f_k$ vanishes on $(L_{-1})^{\tensor k}$ for all $k \geq N_{L}$ and $g_k$ vanishes on $(L'_{-1})^{\tensor k}$ for all $k \geq N_{L'}$ then
$(gf)_k$ vanishes on $(L_{-1})^{\tensor k}$ for all $k \geq N_LN_{L'}$.

\item Note that the function $f_\ast$ in Eq.\ \ref{eq:Fstar} is also well-defined for any 
\underline{coalgebra} morphism $F \maps \S(\bs L) \to \S(\bs L')$
satisfying
\[
F^1_k(\bs x_1, \ldots, \bs x_k) = 0 \quad \forall x_1, \ldots x_k \in L_{-1}.
\]    
for $k \gg 1$. Compatibility of $F$ with the codifferentials is obviously not necessary.
We will use this in the proof of Prop.\ \ref{prop:MC-pullbacks}.
\end{enumerate}
\end{remark}

We note that the first statement in Remark \ref{rmk:coalg_tame}
along with  Eqs.\ \ref{eq:exp1}, \ref{eq:exp2}, and \ref{eq:exp3} imply the following result: 
\begin{proposition}[cf.\ Prop.\ 2.2 \cite{DR:2017}] \label{prop:MC}
Let $f \maps (L, \el) \to (L',\el')$ be a tame $L_\infty$-morphism between
tame $\Z$-graded $L_\infty$-algebras.
Then the function \eqref{eq:Fstar} restricts to a well-defined function $f_\ast \maps \MC(L) \to \MC(L')$ between the corresponding Maurer-Cartan sets. Moreover, the assignment
\[
(L, \el) \xto{f} (L', \el') \quad \longmapsto \quad  \MC(L) \xto{f_\ast} \MC(L')
\]
is functorial.

\end{proposition}

\subsection{``Deformation  functors''} \label{sec:def_func}
The following construction provides important examples of tame  $L_\infty$-algebras.
We denote by $\bcdga$ the category whose objects are unital, non-negatively and cohomologically graded commutative dg $\kk$-algebras which are bounded from above. Morphisms in $\bcdga$ are unit preserving cdga morphisms. 
Let $n \in \N \cup \{\infty\}$. Let $(L,\el) \in \lnaft$ be a finite type Lie $n$-algebra and let $(B,d_B) \in \bcdga$.  
Denote by 
\[
(L \tensor B, \el^{B}) 
\]
the $\Z$--graded $L_\infty$-algebra whose underlying chain complex is $(L \tensor B, \el^B_1)$ where 
\[
\begin{split}
(L \tensor B)_m &:= \bigoplus_{i+j=m} L_{i} \tensor B^{-j}\\
\el^B_1(x\tensor b) &:= \el_1 x \tensor b + (-1)^{\deg{x}}  x \tensor d_B b
\end{split}
\]
and whose higher brackets are defined as:
\begin{equation} \label{eq:elB}
\el^B_k \bigl(x_1 \tensor b_1, \ldots, x_k \tensor b_k \bigr):= (-1)^{\varepsilon} \el_k(x_1,\ldots,x_k) \tensor b_1b_2 \cdots b_k,
\end{equation} 
with 
\[
\varepsilon :=  \sum_{1 \leq i < j \leq k} \deg{b_i}\deg{x_j}.
\]
If $f \maps (L, \el) \to (L',\el')$ is a morphism of Lie $n$-algebras, then it is easy to verify that the maps $f^B_k\maps \Lambda^{k}(L \tensor B) \to L' \tensor B$ defined as
\begin{equation} \label{eq:fB}
f^B_k\bigl(x_1 \tensor b_1, \ldots, x_k \tensor b_k \bigr):= (-1)^{\varepsilon} f_k(x_1,\ldots,x_k) \tensor b_1b_2 \cdots b_k.
\end{equation} 
assemble together to give a $L_\infty$-morphism $f^B \maps (L\tensor B, \el^B) \to (L' \tensor B, \el^{ \prime B})$ in $\Linf$.

\begin{lemma} \label{lem:MC_CE}
Let $(L,\el) \in \lnaft$ be a finite type Lie $n$-algebra and $(B,d_B) \in \bcdga$ a bounded cdga.
\begin{enumerate}
\item If $f \maps (L, \el) \to (L',\el')$ is a morphism of Lie $n$-algebras, then the induced $L_\infty$-morphism
\[
f^B \maps (L\tensor B, \el^B) \to (L' \tensor B, \el^{\prime B})
\]
is a tame morphism between tame $L_\infty$-algebras. 
\item 
The assignment
\[
(L\tensor B, \el^B_k) \xto{f^B} (L' \tensor B, \el^{\prime B}_k)
\quad \longmapsto \quad \MC(L\tensor B) \xto{f^B_\ast} \MC(L'\tensor B).
\]
defines a functor
\begin{equation} \label{eq:MC-cdga}
\MC(-\tensor B) \maps \lnaft \to \Set
\end{equation}
natural in $B \in \bcdga$.

\end{enumerate}
\end{lemma}

\begin{proof}
Statement (2) is straightforward. For statement (1), note that since the underlying cochain complex of $B$ is bounded, there exists an $N \geq 0$ such that
$B= \bigoplus_{i \geq 0}^N B^i$. Since the underlying chain complex of $L$ is concentrated in non-negative degrees we have $(L \tensor B)_{-1} = \bigoplus_{i \geq 0}^{N-1} L_i \tensor B^{i+1}$. It follows from Eq.\ \ref{eq:elB} and Eq.\ \ref{eq:fB} that for all $k \geq N$, and any $a_1,\ldots, a_k \in (L \tensor B)_{-1}$, we have $\el^B_k(a_1,\ldots,a_k)=0$ and $f^B_k(a_1,\ldots,a_k)=0$. A similar argument shows that $(L'\tensor B, \el^{\prime B})$ is also tame.
\end{proof}

We end this section by showing that the functor $\MC(-\tensor B)$ defined in \eqref{eq:MC-cdga}
preserves certain pullback diagrams. We provide a careful detailed proof of this rather straightforward fact in order to have the analogous statement in the category of Banach manifolds follow as a simple corollary (Cor.\ \ref{cor:MC-pullbacks}).
\begin{proposition} \label{prop:MC-pullbacks}
Let $B \in \bcdga$ be a bounded cdga.
Let $f \maps (L,\el) \to (L'',\el'')$ be a fibration and $g \maps (L',\el') \to (L'',\el'')$ be a morphism in $\lnaft$. Let $(L_P,\el_P)$ be the pullback of the diagram
$(L',\el') \xto{g} (L'',\el'') \xleftarrow{f} (L,\el)$.
Then the induced commutative diagram of sets
\begin{equation} \label{diag:MC-pullbacks}
\begin{tikzpicture}[descr/.style={fill=white,inner sep=2.5pt},baseline=(current  bounding  box.center)]
\matrix (m) [matrix of math nodes, row sep=2em,column sep=3em,
  ampersand replacement=\&]
  {  
\MC(L_P \tensor B) \& \MC({L} \tensor B) \\
\MC(L' \tensor B) \& \MC(L'' \tensor B)\\
}; 
  \path[->,font=\scriptsize] 
   (m-1-1) edge node[auto] {$$} (m-1-2)
   (m-1-1) edge node[sloped,below] {$$} (m-2-1)
   (m-1-2) edge node[auto] {$f^B_\ast$}  (m-2-2)
   (m-2-1) edge node[auto] {$g^B_\ast$} (m-2-2)
  ;


\end{tikzpicture}
\end{equation}
is a pullback square. 
\end{proposition}
\begin{proof}
We first consider the special case in which $f$ is strict, and then use this to prove the general case.
\begin{pfcases}[leftmargin=20pt]
\item \underline{Case 1}: $f=f_1 \maps (L,\el) \to (L'',\el'')$  is a strict fibration.
\myspace

\myindent Let $g \maps (L',\el') \to (L'',\el'')$ be a morphism in $\lnaft$. 
As in the proof of Prop.\ \ref{prop:strict_pullback}, we have the following pullback diagram 
in $\lnaft$:

\begin{equation} \label{diag:new:MCpb1}
\begin{tikzpicture}[descr/.style={fill=white,inner sep=2.5pt},baseline=(current  bounding  box.center)]
\matrix (m) [matrix of math nodes, row sep=2em,column sep=3em,
  ampersand replacement=\&]
  {  
(\ti{L},\ti{\el})  \& \bigl  (L, \el ) \\
(L',\el') \& (L'', \el'' ) \\
}; 
  \path[->,font=\scriptsize] 
   (m-1-1) edge node[auto] {$\pr h$} (m-1-2)
   (m-1-1) edge node[auto,swap] {$\pr'h$} (m-2-1)
   (m-1-2) edge node[auto] {$f$} (m-2-2)
   (m-2-1) edge node[auto] {$g$} (m-2-2)
  ;

  \begin{scope}[shift=($(m-1-1)!.4!(m-2-2)$)]
  \draw +(-0.25,0) -- +(0,0)  -- +(0,0.25);
  \end{scope}
\end{tikzpicture}
\end{equation}
where $\ti{L}$ is the pullback of diagram \eqref{diag:strict_pullback2} in $\Chain$, and
$h \maps (\ti{L},\ti{\el}) \to (L' \oplus L, \el' \oplus \el)$ is the morphism in $\lnaft$ associated to the dg coalgebra morphism $H$ defined in Eq.\ \ref{eq:H}.  
Let $B \in \bcdga$. Then, via Eqs.\ \ref{eq:Fstar} and \ref{eq:fB}, 
the morphism $h$ induces the function $h^{B}_{\ast} \maps \MC(\ti{L} \tensor B) \to \MC \bigl( (L' \oplus L) \tensor B \bigr) \cong \MC(L' \tensor B) \times \MC(L \tensor B)$, where
\begin{equation*} 
 h^B_\ast(a',a) = (a', a + \bigl (\bs^{-1} \sigma \bs \tensor \id_B \bigr) g^B_\ast(a')).
\end{equation*} 
Applying the functor $\MC(-\tensor B)$ to \eqref{diag:new:MCpb1} gives us the following commutative diagram of sets
\begin{equation*} 
\begin{tikzpicture}[descr/.style={fill=white,inner sep=2.5pt},baseline=(current  bounding  box.center)]
\matrix (m) [matrix of math nodes, row sep=2em,column sep=2em,
  ampersand replacement=\&]
  {  
\MC(\ti{L}\tensor B) \& [-1 cm] \& \\ [-0.5cm]
\& E \& \MC({L} \tensor B) \\
\& \MC(L' \tensor B) \& \MC(L'' \tensor B)\\
}; 
  \path[->,font=\scriptsize] 
   (m-2-2) edge node[auto] {$$} (m-2-3)
   (m-2-2) edge node[auto,swap] {$$} (m-3-2)
   (m-2-3) edge node[auto] {$f^B_\ast$}  (m-3-3)
   (m-3-2) edge node[auto] {$g^B_\ast$} (m-3-3)
   (m-1-1) edge node[auto] {$h^B_\ast $} (m-2-2)
   (m-1-1) edge [bend left=20]  node[auto] {$\pr^B_\ast \circ h^B_\ast$} (m-2-3)
   (m-1-1) edge [bend right=40]  node[auto,swap] {$\pr'^B_\ast \circ h^B_\ast$} (m-3-2)
  ;

   \begin{scope}[shift=($(m-2-2)!.3!(m-3-3)$)]
   \draw +(-0.25,0) -- +(0,0)  -- +(0,0.25);
   \end{scope}
\end{tikzpicture}
\end{equation*}
in which the pullback is denoted by
\begin{equation} \label{eq:E}
E: = \MC(L' \tensor B) \times_{\MC(L'' \tensor B)}  \MC(L \tensor B).
\end{equation}
We proceed by explicitly constructing the inverse of $h^B_\ast $.

\myindent Let $J \maps \S( \bs L' \oplus \bs L) \to \S( \bs L' \oplus \bs L)$ denote the inverse of the coalgebra morphism $H$, which was introduced in Eq.\ \ref{eq:J}. 
Even though $J$ may fail to preserve the codifferential, it follows from Remark \ref{rmk:coalg_tame} and Lemma \ref{lem:MC_CE} that the function 
\[
j^B_\ast \maps (L' \tensor B \oplus L \tensor B)_{-1} \to 
(L' \tensor B \oplus L \tensor B)_{-1} 
\]
is well-defined. We denote by $(L' \tensor B)_{-1} \times_{(L'' \tensor B)_{-1}} (L \tensor B)_{-1}$ 
the pullback of the diagram of sets:
\begin{equation} \label{eq:new:MCpb1}
(L' \tensor B)_{-1} \xto{g^B_\ast} (L'' \tensor B)_{-1} \xleftarrow{f^B_\ast} (L \tensor B)_{-1}. 
\end{equation}
Now let
\begin{equation*} 
\varphi \maps 
(L' \tensor B)_{-1} \times_{(L'' \tensor B)_{-1}} (L \tensor B)_{-1}  \to (L'\tensor B \oplus L \tensor B)_{-1} 
\end{equation*}
denote the function 
\begin{equation} \label{eq:varphi}
\varphi(a',a):= j^B_\ast(a',a) = \bigl (a', a - (\bs^{-1} \sigma \bs \tensor \id_B ) g^B_\ast(a') \bigr).
\end{equation}
We claim that $\varphi$ restricted to the pullback $E$ is indeed the inverse to $h^B_\ast$.
Let us first show that $\varphi \vert_E \maps E \to \MC(\ti{L} \tensor
B)$ is a well-defined function. We break the calculation into two steps:

\begin{pfsteps}[leftmargin=20pt]
\item \underline{Step (i)}: $\im \varphi \sse (\ti{L} \tensor B)_{-1}$. 

\myspace
\myindent Suppose $(a',a) \in (L' \tensor B)_{-1} \times_{(L'' \tensor B)_{-1}} (L \tensor B)_{-1}$, so 
that $g^B_\ast(a')=f^B_\ast(a)$. 
It will be convenient to express the relevant terms involved as sums i.e.\
$(a',a) = \sum_{i \geq 0} (a'_i,a_i)$ and
\[
\begin{split}
\varphi(a',a) = \sum_{i \geq 0} \varphi(a',a)_i, \quad
g^B_\ast(a') = \sum_{i \geq 0}g^B_\ast(a')_i =  \sum_{i \geq 0}f^B_\ast(a)_i=f^B_\ast(a), 
\end{split}
\]
where
$(a'_i,a_i) \in (L'_i \oplus L_i) \tensor B^{i+1},$  $\varphi(a',a)_i \in (L'_i \oplus L_i) \tensor B^{i+1}$, and 
\[
g^B_\ast(a')_i=f^B_\ast(a)_i \in L''_i \tensor B^{i+1}.
\]

\myindent We first consider the summand $\varphi(a',a)_0= \bigl(a'_0, a_0 -  (\bs^{-1} \sigma \bs \tensor \id_B ) g^B_\ast(a')_0 \bigr)$. 
Since all of the elements of $B$ appearing 
in the term $a' \in \bigoplus_{i \geq 0} L'_i \tensor B^{i+1}$ are positively graded, 
Eq.\ \ref{eq:Fstar} and Eq.\ \ref{eq:fB} imply that
\begin{equation} \label{eq:new:MCpb2}
g^B_\ast(a')_0  =(g_1\tensor \id_B)(a'_0) \in L''_{0}\tensor B^1.
\end{equation}
Recall from the definition of $\sigma \maps \bigoplus_{i \geq 1} \bs L''_i \to \bigoplus_{i \geq 1} \bs L_i$ in Eq.\ \ref{eq:sigma} that $\sigma \bs \vert_{L''_0} = \sigma \vert_{\bs L''_1}=0$. Hence, we deduce from Eq.\ \ref{eq:new:MCpb2} that
\begin{equation} \label{eq:new:MCpb3} 
\pr^B \varphi(a',a)_0 = a_0.
\end{equation}
On the other hand, the morphism $f$ is strict by hypothesis, and so from Eq.\ \ref{eq:new:MCpb3} we have
\[
(f_1 \tensor \id_B)\pr^B \varphi(a',a)_0= f^B_{\ast}(a)_0. 
\]
Therefore, by combining the above equality with Eq.\ \ref{eq:new:MCpb2}, we obtain:
\begin{equation} \label{eq:new:MCpb4} 
(f_1 \tensor \id_B)\pr^B \varphi(a',a)_0 = g^B_{\ast}(a')_0 = (g_1\tensor \id_B)\pr'^B \varphi(a',a)_0.  
\end{equation}
It follows from the definition of $\bs \ti{L}$ in Eq.\ \ref{eq:Ltilde} that $\ti{L}_0$ is the pullback of the linear maps $f_1$ and $g_1$. Therefore, since the functor $-\tensor_\kk B^1$ is exact, we conclude from Eq.\ \ref{eq:new:MCpb4}  that $\varphi(a',a)_0 \in \ti{L}_0 \tensor B^1.$

\myindent Now we consider the summands $\varphi(a',a)_{i} \in (L'_i \oplus L_i) \tensor B^{i+1}
$ for $i \geq 1$. 
It follows from the definition $\ti{L}_{i \geq 1}$ that it is sufficient to verify that
\begin{equation*} 
\pr^B \varphi(a',a)_{i} = a_i - (\bs^{-1} \sigma \bs \tensor \id_B ) g^B_\ast(a')_i
\in \ker f_1 \tensor B.
\end{equation*}
Again, $f^B_\ast=f_1 \tensor \id_B$ by hypothesis. Therefore:
\[
\begin{split}
(f_1 \tensor \id_B) \Bigl( \pr^B \varphi(a',a)_{i} \Bigr) &= f^B_\ast(a)_i
- (f_1\bs^{-1} \sigma \bs \tensor \id_B ) g^B_\ast(a')_i\\
&= f^B_\ast(a)_i - g^B_\ast(a')_i\\
&=0,
\end{split}
\]
where the last two equalities follow, respectively, from the definition of $\sigma$ 
in Eq.\ \ref{eq:sigma}, and the fact that $(a',a)$ is an element of the pullback of \eqref{eq:new:MCpb1}.

\myindent Hence, we conclude that $\varphi(a',a) \in (\ti{L} \tensor B)_{-1}.$

\item \underline{Step (ii)}: $\im \varphi \vert_E \sse \MC(\ti{L} \tensor B)$.
\myspace

\myindent To begin with, note that the pullback $E$ in Eq.\ \ref{eq:E} is the equalizer of the functions
\begin{equation*}
\begin{tikzpicture}[descr/.style={fill=white,inner sep=2.5pt},baseline=(current  bounding  box.center)]
\matrix (m) [matrix of math nodes, row sep=2em,column sep=5em,
  ampersand replacement=\&]
  {  
(L' \tensor B)_{-1} \times_{(L'' \tensor B)_{-1}} (L \tensor B)_{-1} \&
(L' \tensor B)_{-2} \times (L \tensor B)_{-2} \\
}; 

   \path[->,font=\scriptsize] 
($(m-1-1.east)+(0,0.25)$) edge node[above] {$\curv^{\prime B} \times \curv^B$} ($(m-1-2.west)+(0,0.25)$) 
($(m-1-1.east)+(0,-0.25)$) edge node[below] {$0$} ($(m-1-2.west)+(0,-0.25)$) 
;   
\end{tikzpicture}
\end{equation*} 
where $\curv^B$ and $\curv^{\prime B}$ are the curvature functions \eqref{eq:curv} for the tame $L_\infty$-algebras $L \tensor B$ and $L' \tensor B$, respectively. 

\myindent Suppose $(a',a) \in E$. 
We wish to show that  $\widetilde{\curv^B}(\varphi(a',a))=0$,
where $\widetilde{\curv^B}$ is the curvature function for
$\ti{L}\tensor B$.
Let $\delta_\oplus$ denote the codifferential that encodes the product $L_\infty$-structure on $L' \oplus L$. Then $(a',a) \in E$ implies that  
\begin{equation} \label{eq:MC-pb1}
\curv^B_\oplus(a',a) =  \bigl (\curv^{\prime B}(a'),\curv^B(a) \bigr) = 0.
\end{equation}
As in Claim \ref{claim:2}, let $\ti{\delta}=J \circ  \delta_\oplus
\circ  H$ denote the codifferential encoding the $L_\infty$ structure
on $\ti{L}$, and let $\ti{\delta}^B$ denote the induced codifferential 
for the $L_\infty$ structure on $\ti{L} \tensor B$. 
Passing to the completions, Eq.\ \ref{eq:exp1} implies that
\begin{equation} \label{eq:MC-pb2}
\ti{\delta}^B \bigl( \exp (\bs \varphi(a',a) ) - 1 \bigr) = \exp(\bs \varphi(a',a))(\bs \widetilde{\curv}^B(\varphi(a',a))),
\end{equation}
while Eq.\ \ref{eq:exp3} gives us
\begin{equation} \label{eq:MC-pb3}
\exp (\bs \varphi(a',a) ) - 1  = \exp (\bs j^B_\ast(a',a) ) - 1 = J^B \bigl (\exp (\bs (a',a) ) - 1 \bigr).
\end{equation}
A straightforward calculation shows that $\ti{\delta}^B = J^B \circ \delta^B_\oplus \circ H^B$.
Therefore, by combining Claim \ref{claim:1} with Eq.\ \ref{eq:MC-pb3} we obtain:
\begin{align*}
\ti{\delta}^B \bigl( \exp (\bs \varphi(a',a) ) - 1 \bigr)&= J^B
\delta^B_{\oplus} \bigl( (\exp (\bs (a',a) ) - 1 \bigr)\\
& = J^B \bigl (\exp(\bs (a',a))(\bs \curv^B_\oplus((a',a))) \bigr).
\end{align*}
It then follows from the above equality and Eq.\ \ref{eq:MC-pb1} that $\ti{\delta}^B \bigl( \exp (\bs \varphi(a',a) ) - 1 \bigr) =0$. Equation \ref{eq:MC-pb2} then implies that
 $\widetilde{\curv^B}(\varphi(a',a))=0$, and so $\varphi(a',a) \in \MC(\ti{L} \tensor B)$. 

\myindent Hence, we conclude that the function $\varphi  \vert_E \maps E \to \MC(\ti{L} \tensor B)$ well-defined.

\end{pfsteps}


Finally, we verify that $\varphi \vert_E$ is the inverse to $h^B_\ast$. Keeping in mind Remark \ref{rmk:coalg_tame}, it is straightforward to show that the functoriality described in Prop.\ \ref{prop:MC} and Lemma \ref{lem:MC_CE} holds for
the assignment of the graded coalgebra morphisms 
$H,J \maps  \S( \bs L' \oplus \bs L) \to \S( \bs L' \oplus \bs L)$ 
to the functions:
\[
h^B_{\ast}, j^{B}_\ast \maps (L'\tensor B \oplus  L \tensor B)_{-1} 
\to (L'\tensor B \oplus  L \tensor B)_{-1}.
\]
Hence, Claim \ref{claim:1} implies that $ h^B_{\ast}\circ  j^{B}_\ast =
j^B_{\ast}\circ h^{B}_\ast = \id_{(L'\tensor B \oplus  L \tensor B)_{-1}}$.
It then follows from the definition of $\varphi$ in Eq.\ \ref{eq:varphi}
that $\varphi \vert_{E}$ is the inverse of $h^{B}_{\ast}$.

\myindent This concludes the proof of the proposition in the special case when $f$ is a strict fibration.

\item \underline{Case 2}: $f \maps (L,\el) \to (L'',\el'')$  is an arbitrary fibration.
\myspace

\myindent First, we factor $f$ into a strict fibration followed by an isomorphism. Indeed,  
by Lemma \ref{lem:strict_fib}, there exists a Lie $n$-algebra $(\hat{L},\hat{\el})$ and an 
isomorphism $\psi \maps (\hat{L},\hat{\el}) \xto{\cong} (L,\el)$ such that
\[
f\psi \maps (\ha{L},\ha{\el}) \to (L'',\el'') 
\]
is a strict fibration with  $f\psi=(f\psi)_1 = f_1$. Let $g \maps (L',\el') \to (L'',\el'')$ be a morphism in $\lnaft$. As in the proof of Cor.\ \ref{cor:fib_pback}, we obtain a pair of pullback squares in $\lnaft$:
\[
\begin{tikzpicture}[descr/.style={fill=white,inner sep=2.5pt},baseline=(current  bounding  box.center)]
\matrix (m) [matrix of math nodes, row sep=2em,column sep=3em,
  ampersand replacement=\&]
  {  
(\ti{L},\ti{\el}) \& (\ha{L}, \hat{\el}) \\
(L_P,\el_P)  \& (L, \el) \\
(L',\el') \& (L'', \el'')\\
}; 
  \path[->,font=\scriptsize] 
   (m-1-1) edge node[auto] {$\pr h$} (m-1-2)
   (m-1-1) edge node[auto] {$\widetilde{\psi}$} node[sloped,below] {$\cong$} (m-2-1)
   (m-1-2) edge node[auto] {$\psi$} node[sloped,below] {$\cong$} (m-2-2)
   (m-2-1) edge node[auto] {$$} (m-2-2)
   (m-2-1) edge node[auto] {$q$} (m-2-2)
   (m-2-1) edge node[auto,swap] {$q'$} (m-3-1)
   (m-2-2) edge node[auto] {$f$} (m-3-2)
   (m-3-1) edge node[auto] {$g$} (m-3-2)
   (m-1-1) edge [bend right=60]  node[auto,swap] {$\pr'h$} (m-3-1)
   (m-1-2) edge [bend left=60]  node[auto] {$f\psi=f_1$} (m-3-2)
  ;

  \begin{scope}[shift=($(m-1-1)!.4!(m-2-2)$)]
  \draw +(-0.25,0) -- +(0,0)  -- +(0,0.25);
  \end{scope}

  \begin{scope}[shift=($(m-2-1)!.4!(m-3-2)$)]
  \draw +(-0.25,0) -- +(0,0)  -- +(0,0.25);
  \end{scope}
\end{tikzpicture}
\]
where $(\ti{L},\ti{\el})$ is now the pullback of $g$ and the strict fibration $f\psi \maps (\ha{L},\ha{\el}) \to (L'',\el'')$, and $(L_P,\el_P)$ is the pullback of $g$ and $f$.
Let 
\[
E: = \MC(L' \tensor B) \times_{\MC(L'' \tensor B)}  \MC(\ha{L} \tensor B)
\]
be the pullback of the diagram of sets 
$\MC(L' \tensor B) \xto{g^B_\ast} \MC(L'' \tensor B) \xleftarrow{(f\psi)^B_\ast} \MC (\ha{L} \tensor B).$
Then our result from Case 1 implies that there exists a bijection $\vphi \vert_E \maps E \xto{\cong} \MC(\ti{L} \tensor B)$ such that the following diagram commutes:
\begin{equation} \label{diag:new:MCpb3}
\begin{tikzpicture}[descr/.style={fill=white,inner sep=2.5pt},baseline=(current  bounding  box.center)]
\matrix (m) [matrix of math nodes, row sep=2em,column sep=4em,
  ampersand replacement=\&]
  {  
E \& [-0.5 cm] \& \\ [-0.5cm]
\& \MC(\ti{L}\tensor B) \& \MC(\ha{L} \tensor B) \\
\& \MC(L_P \tensor B) \& \MC(L \tensor B)\\
\& \MC(L'\tensor B) \& \MC(L''\tensor B)\\
}; 
  \path[->,font=\scriptsize] 
   (m-2-2) edge node[auto] {$(\pr h)^B_\ast$} (m-2-3)
   (m-2-2) edge node[auto,swap] {$$} (m-3-2)
   (m-2-2) edge node[auto] {$\widetilde{\psi}^B_\ast$} node[sloped,below] {$\cong$} (m-3-2)
   (m-2-3) edge node[auto] {$\psi^B_\ast$} node[sloped,below] {$\cong$}  (m-3-3)
   (m-3-3) edge node[auto] {$f^B_\ast$}   (m-4-3)
   (m-3-2) edge node[auto,swap] {$q'^B_\ast$}   (m-4-2)
   (m-3-2) edge node[auto] {$q^B_\ast$} (m-3-3)
   (m-4-2) edge node[auto] {$g^B_\ast$} (m-4-3)
   (m-1-1) edge node[auto] {$\vphi \vert_E$} node[sloped,below] {$\cong$} (m-2-2)
   (m-1-1) edge [bend left=20]  node[auto] {$$} (m-2-3)
   (m-1-1) edge [bend right=40]  node[auto,swap] {$$} (m-4-2)
  ;

\end{tikzpicture}
\end{equation}

Now, let 
\[
E_P: = \MC(L' \tensor B) \times_{\MC(L'' \tensor B)}  \MC(L \tensor B)
\]
be the pullback of the diagram of sets 
$\MC(L' \tensor B) \xto{g^B_\ast} \MC(L'' \tensor B) \xleftarrow{f^B_\ast} \MC (L \tensor B)$, and
let 
\[
\varphi_{E_P} \maps E_P \to \MC(L_P \tensor B)
\]
be the function
\begin{equation} \label{eq:varphiPB}
\begin{split}
\varphi_{E_P}(a',a)&:= (\widetilde{\psi}^{B}_{\ast} \circ \vphi \vert_E) \bigl( a', (\psi^{B}_{\ast})^{-1}(a) \bigr)\\
&~ =(\widetilde{\psi}^{B}_{\ast} \circ j^B_\ast ) \bigl( a', (\psi^{B}_{\ast})^{-1}(a) \bigr).
\end{split}
\end{equation}
Note that $\vphi_{E_P}$ is a bijection by construction. Using diagram
\eqref{diag:new:MCpb3}, it is easy to see that $q^B_\ast
\varphi_{E_P}(a',a)= a$ and $q'^B_\ast\varphi_{E_P}(a',a)= a'$. 
Hence, we conclude that \eqref{diag:MC-pullbacks} is a pullback square, and this completes the proof of the proposition.

\end{pfcases}

\end{proof}

\subsection{Maurer-Cartan sets with differentiable structure} \label{sec:MC_smooth}
We now consider the scenario in which all of the Maurer-Cartan sets in Prop.\ \ref{prop:MC-pullbacks} have geometric structure. Specifically, we are interested in the case when the geometry arises from a dg Banach algebra structure on $(B,d_B)$. Examples relevant to our applications in \cite{Rogers-Zhu:2018} include the cdgas $\Omega(\Delta^n)$ and $\Omega(\Lambda^{k}_{j})$: the dg Banach algebras of $r$-times continuously differentiable forms on the geometric $n$-simplex $\Delta^n$ and
the geometric horn $\Lambda^{k}_{j} \subseteq \Delta^k$, respectively. (See Sec.\ 5.1 of \cite{Henriques:2008}). 

So let $\kk = \R$ and suppose $(B,d_B) \in \bcdga$ is a fixed cdga equipped with the structure of a dg Banach algebra. If $(L,\el) \in \lnaft$, then since $L$ is finite type, the structure on $B$ naturally makes $(L \tensor B)$ into a graded Banach space. From Eqs.\ \ref{eq:curv} and \ref{eq:elB},   
we see that the curvature $\curv^{B} \maps (L \tensor B)_{-1} \to (L \tensor B)_{-2}$ is a polynomial and hence a smooth function between Banach manifolds, in the sense of \cite[Ch.\ I,3]{Lang:95}. 
Similarly, if $f \maps (L,\el) \to (L',\el')$ is a morphism in $\lnaft$, then it follows from 
Eq.\ \ref{eq:Fstar} and Eq.\ \ref{eq:fB} that $f^{B}_\ast \maps (L \tensor B)_{-1} \to (L' \tensor B)_{-1}$ is a smooth function.

We now restrict our focus to those $(L,\el) \in \lnaft$ which satisfy the following assumption:
\begin{ass} \label{ass:1}
The Maurer-Cartan set $\MC(L \tensor B) \sse (L \tensor B)_{-1}$ is a Banach submanifold \cite[Ch.\ II, 2]{Lang:95} and therefore  
\[
\begin{tikzpicture}[descr/.style={fill=white,inner sep=2.5pt},baseline=(current  bounding  box.center)]
\matrix (m) [matrix of math nodes, row sep=2em,column sep=2em,
  ampersand replacement=\&]
  {  
\MC({L}\tensor B) \&  ({L}\tensor B)_{-1} \& ({L}\tensor B)_{-2}\\
}; 
   \path[->,font=\scriptsize] 
($(m-1-2.east)+(0,0.25)$) edge node[above] {${\curv^B}$} ($(m-1-3.west)+(0,0.25)$) 
($(m-1-2.east)+(0,-0.25)$) edge node[below] {$0$} ($(m-1-3.west)+(0,-0.25)$) 
;   
   \path[right hook->,font=\scriptsize] 
   (m-1-1) edge node[auto] {$$} (m-1-2);
\end{tikzpicture}
\]
is an equalizer diagram in the category of Banach manifolds.
\end{ass}

\begin{cor} \label{cor:MC-pullbacks}
Suppose $B \in \bcdga$ has the structure of a dg Banach algebra. 
Let $f \maps (L,\el) \to (L'',\el'')$ be a fibration and $g \maps (L',\el') \to (L'',\el'')$ be a morphism in $\lnaft$, and let $(L_P,\el_P)$ 
be the pullback of the diagram
$(L',\el') \xto{g} (L'',\el'') \xleftarrow{f} (L,\el)$.

Assume that all of the aforementioned Lie $n$-algebra satisfy assumption \eqref{ass:1}.
If the pullback of the diagram 
\[
\MC(L' \tensor B) \xto{g^B_{\ast}} \MC(L'' \tensor B)
\xleftarrow{f^{B}_{\ast}} \MC(L \tensor B)
\]
exists as a Banach manifold, then the induced commutative diagram
\[
\begin{tikzpicture}[descr/.style={fill=white,inner sep=2.5pt},baseline=(current  bounding  box.center)]
\matrix (m) [matrix of math nodes, row sep=2em,column sep=3em,
  ampersand replacement=\&]
  {  
\MC(L_P \tensor B) \& \MC({L} \tensor B) \\
\MC(L' \tensor B) \& \MC(L'' \tensor B)\\
}; 
  \path[->,font=\scriptsize] 
   (m-1-1) edge node[auto] {$$} (m-1-2)
   (m-1-1) edge node[sloped,below] {$$} (m-2-1)
   (m-1-2) edge node[auto] {$f^B_\ast$}  (m-2-2)
   (m-2-1) edge node[auto] {$g^B_\ast$} (m-2-2)
  ;


\end{tikzpicture}
\]
is a pullback square in the category of Banach manifolds.
\end{cor}

\begin{proof}
First, we note that the underlying set of the pullback in the category of Banach manifolds is the usual fiber product \cite[Ch.II,2]{Lang:95}. Therefore, by hypothesis, the set $E_P: = \MC(L' \tensor B) \times_{\MC(L'' \tensor B)}  \MC(L \tensor B)$ admits the structure of a manifold. 

In the proof of Prop.\ \ref{prop:MC-pullbacks}, we explicitly constructed 
in Eq.\ \ref{eq:varphiPB} a bijection 
\[
\varphi_{E_P} \maps E_P \xto{\cong} \MC(L_P\tensor B)
\]  
which identified $\MC(L_P\tensor B)$ as the pullback in 
the category of sets. We observe that all functions appearing in the construction of $\varphi_{E_P}$, its inverse, and the relevant pullback diagrams are either:
\begin{enumerate}
\item{polynomial functions between Banach spaces of the form $(V\tensor B)_{-1}$, 
where $V$ is a finite-dimensional graded vector space, or }  

\item {polynomial functions between equalizers of polynomial functions 
of the above type.}

\end{enumerate}
By assumption \eqref{ass:1}, the equalizers themselves are submanifolds of 
Banach spaces. Hence, the function $\varphi_{E_P}$ respects the differentiable structures, as does its inverse. Therefore, $\varphi_{E_P}$ is a diffeomorphism.
\end{proof}

We end this section with some remarks concerning the validity of our assumption \eqref{ass:1}. 
It is not difficult to construct Lie $n$-algebras $(L,\el)$ and dg Banach algebras $(B,d_B)$  such that $\MC(L \tensor B)$ fails to satisfy this assumption. The following is a simple example.

\begin{example}
Let $L= L_1 \oplus L_2$ denote the  graded vector space: 
\begin{align*}
L_1 &:= \R e_1 \oplus \R e_2, & L_2 := \R\ti{e},\\
&\deg{e_1}=\deg{e_2}=1, &\deg{\ti{e}}=2.
\end{align*}
Equip $L$ with the degree 0 graded skew-symmetric bracket $\el_2 \maps L \tensor L \to L$
whose non-trivial values on generators are
\[
\el_2(e_1,e_1)=\ti{e}, \quad \el_2(e_2,e_2)=-\ti{e}.
\]  
It is easy to see that $\el_2$ satisfies the graded Jacobi identity. Therefore, $\el_2$ is an honest Lie bracket, and $(L,\el_2)$ is a Lie $3$-algebra.

Next, let $(B,d_B)$ denote the cdga $B:=\R[\theta] / (\theta^3)$, with $\deg{\tha}=2$, and trivial differential $d_B=0$. As a graded vector space, $B=B_0 \oplus B_2 \oplus B_4$, where 
\[
B_0 = \R, \quad B_2 = \R\tha, \quad  B_4=\R\tha^2.
\]
Since $B$ is a finite-dimensional $\R$-algebra, it admits the structure of a Banach algebra (by  embedding into $\End(\R^3)$, for example). Hence, the degree $-1$ piece of the tame $L_\infty$-algebra $(L \tensor B, \el^B_2)$ is just the Euclidean plane:
\[
(L \tensor B)_{-1} = \R (e_1 \tensor \tha) ~ \oplus ~ \R (e_2 \tensor \tha).
\]
If $a = xe_1 \tensor \tha + ye_2 \tensor \tha \in (L \tensor B)_{-1}$, then $\curv^B(a)=0$ if and only if
\[
\frac{1}{2}(x^2-y^2)(\ti{e}\tensor \tha^2)=0.
\]
Therefore $\MC(L\tensor B)$ is the zero locus of the polynomial $f(x,y)=x^2-y^2$, and so it
is not a submanifold of $(L \tensor B)_{-1} \cong \R^2$.
\end{example}
\myspace
On the other hand, the assumption \eqref{ass:1} is always satisfied in our main application of interest, when $B=\Omega(\Delta^n)$. Indeed, this follows from a result of
Henriques' \cite[Thm.\ 5.10]{Henriques:2008} and a result of \v{S}evera and \v{S}ira\v{n} \cite[Prop.\ 4.3]{Severa-Siran}. Furthermore, the hypothesis in Cor.\ \ref{cor:MC-pullbacks} concerning the existence of the pullback of Maurer-Cartan spaces is satisfied if the morphism $f \maps (L,\el) \to (L',\el')$ is a ``quasi-split fibration'' 
\cite[Sec.\ 6]{Rogers-Zhu:2018}. This is a special kind of fibration which we discuss in the next section.

\section{Postnikov towers for Lie $n$-algebras} \label{sec:postnikov}
In this last section, we analyze the functorial aspects of 
Henriques' Postnikov construction for Lie $n$-algebras \cite{Henriques:2008}. We show that in certain cases the Postnikov tower admits a convenient functorial decomposition. We use this in \cite{Rogers-Zhu:2018} to prove that the integration functor sends a certain distinguished class of fibrations in $\lnaft$ to fibrations between simplicial Banach manifolds. We call these distinguished fibrations ``quasi-split''.

\begin{definition}\label{def:split_fib}
A fibration of Lie $n$-algebras $f \maps (L,\el) \to (L',\el')$ is a \textbf{quasi-split fibration}
if: 
\begin{enumerate}

\item the induced map in homology $H(f_1) \maps H(L) \to H(L')$ is surjective in all degrees and,

\item 
$H_0(L) \cong \ker H_0(f_1) \oplus H_0(L')$ in the category of Lie algebras.
\end{enumerate}
\end{definition}

Note that every acyclic fibration in $\lnaft$ is a quasi-split fibration. More generally, $f \maps (L,\el) \to (L',\el')$ is a quasi-split fibration if $\ker H(f_1)$ is central and
$H(f_1) \maps H(L) \to H(L')$ is a split epimorphism in the category of $H_0(L)$-modules.

The string Lie 2-algebra $(\g \oplus \R[-1],\{\el_1,\el_2,\el_3\})$ associated to a simple Lie algebra $\g$ of compact type was the original motivation for Henriques' work in \cite{Henriques:2008}. It sits in a quasi-split fiber sequence of the form
\[
\begin{tikzpicture}[descr/.style={fill=white,inner sep=2.5pt},baseline=(current  bounding  box.center)]
\matrix (m) [matrix of math nodes, row sep=2em,column sep=3em,
  ampersand replacement=\&]
  {  
\R \& \R \& 0 \\
0 \& \g \& \g\\
}
; 
  \path[->,font=\scriptsize] 
   (m-1-1) edge node[auto] {$\id$} (m-1-2)
   (m-1-2) edge node[auto] {$$} (m-1-3)
   (m-1-1) edge node[auto] {$$} (m-2-1)
   (m-2-1) edge node[auto] {$$} ($(m-2-2.west)+(0,0.05)$)
   (m-1-3) edge node[auto] {$$} (m-2-3)
   (m-2-2) edge node[auto] {$$} (m-2-3)
   (m-1-2) edge node[auto] {$\el_1=0$} (m-2-2)
  ;
\end{tikzpicture}
\]
This is a special case of a more general fiber sequence called a "central $n$-extension". (See \cite[Sec.\ 6]{Rogers-Zhu:2018}.)

\subsection{Morphisms between towers}
Let $(L,\el)$ be a Lie $n$-algebra. Following \cite[Def.\ 5.6]{Henriques:2008},
we consider two different truncations
of the underlying chain complex $(L,d=\el_1)$. 
For any $m \geq 0$, denote by $\tau_{\leq m}L$ and $\tau_{< m}L$ the following $(m+1)$-term complexes:
\begin{equation*} 
\begin{split}
(\tau_{\leq m} L)_i=
\begin{cases}
L_i & \text{if $i < m$,}\\
\coker(d_{m+1}) & \text{if $i=m$,}\\
0 & \text{if $i >m$,}
\end{cases}
\qquad 
(\tau_{< m} L)_i=
\begin{cases}
L_i & \text{if $i < m$,}\\
\im (d_{m}) & \text{if $i=m$,}\\
0 & \text{if $i >m$.}
\end{cases}
\end{split}
\end{equation*}
In degree $m$, the differentials for $\tau_{\leq m}L$ and $\tau_{<  m}L$ are $d_{m} \maps L_m/\im(d_{m+1}) \to L_{m-1}$, and the inclusion $ \im(d_m) \emb L_{m-1}$, respectively. The homology complexes of
$\tau_{\leq m}L$ and $\tau_{<  m}L$ are
\[
H_{i}(\tau_{\leq m}L) = 
\begin{cases}
H_i(L) & \text{if $i \leq m$,}\\
0 & \text{if $i>m$,}
\end{cases}
\qquad
H_{i}(\tau_{< m}L) = 
\begin{cases}
H_i(L) & \text{if $i <  m$,}\\
0 & \text{if $i \geq m$.}
\end{cases}
\]
We have the following surjective chain maps
\begin{equation} \label{eq:projs}
\begin{split}
p_{\leq m} \maps L \to \tau_{\leq m}L \qquad p_{< m} \maps L \to \tau_{< m}L
\end{split}
\end{equation}
where in degree $m$, the map $p_{\leq m}$ is the surjection $L_m \to \coker(d_{m+1})$, and
$p_{< m}$ is the differential $d_{m} \maps L_m \to \im(d_{m})$. There are also the similarly defined
surjective chain maps
\begin{equation} \label{eq:projs1}
\begin{split}
q_{\leq m} \maps \tau_{\leq m}L \to \tau_{< m}L, \quad  q_{< m+1} \maps \tau_{< m +1}L \xto{\sim} \tau_{\leq m}L. 
\end{split}
\end{equation}
The map $q_{\leq m}$ in degree $m$ is the differential $d_m \maps \coker{d_{m+1}} \to \im d_m$, and the identity in all other degrees. The map $q_{<m+1}$ is the projection $L_m \to \coker d_{m+1}$ in degree $m$, the identity in all degrees $<m$, and the zero map in degree $m+1$.  
We note that $q_{<m+1}$  is a quasi-isomorphism of complexes.

\begin{proposition} \label{prop:lna_tower}
Let $(L,\el)$ be a Lie $n$-algebra. 
\begin{enumerate}

\item The Lie $n$-algebra structure on $(L,\el)$ induces Lie $(m+1)$-structures on the complexes
$\tau_{\leq m} L$ and $\tau_{<m} L$ whose brackets are given by
\[
\tau_{\leq m} \el_{k}(\bar{x}_1,\ldots,\bar{x}_k):= p_{\leq m}\el_k(x_1,\ldots,x_k), \quad
\tau_{< m} \el_{k}(\bar{y}_1,\ldots,\bar{y}_k):= p_{<m}\el_k(y_1,\ldots,y_k),
\] 
where $\bar{x}_i= p_{\leq m}(x_i)$ and $\bar{y}_i= p_{< m}(y_i)$.

\item The assignments $(L,\el) \mapsto (\tau_{\leq m} L,\tau_{\leq m}\el)$ and
$(L,\el) \mapsto (\tau_{< m} L,\tau_{< m}\el)$ are functorial.

\item An $L_\infty$-morphism $\phi \maps (L,\el) \to (L',\el')$ induces a morphism of towers of Lie $n$-algebras
\begin{equation} \label{diag:tower}
\begin{tikzpicture}[descr/.style={fill=white,inner sep=2.5pt},baseline=(current  bounding  box.center)]
\matrix (m) [matrix of math nodes, row sep=2em,column sep=1.4em,
  ampersand replacement=\&]
  {  
\cdots \tau_{\leq m-1} L \& \tau_{< m-1} L \& \tau_{\leq m-2}L \& ~ \cdots ~ \& 
\tau_{\leq 1 } L \& \tau_{< 1} L \& \tau_{\leq 0} L\\
\cdots \tau_{\leq m-1} L' \& \tau_{< m-1} L' \& \tau_{\leq m-2}L' \& ~ \cdots ~ \& 
\tau_{\leq 1 } L' \& \tau_{< 1} L' \& \tau_{\leq 0} L'\\
};
\path[->,font=\scriptsize] 
(m-1-1) edge node[auto] {$q_{\leq m-1}$} (m-1-2)
(m-1-2) edge node[auto] {$q_{<m-1}$} (m-1-3)
(m-1-3) edge node[auto] {$q_{ \leq m-2}$} (m-1-4)
(m-1-4) edge node[auto] {$$} (m-1-5)
(m-1-5) edge node[auto] {$q_{\leq 1}$} (m-1-6)
(m-1-6) edge node[auto] {$q_{< 1}$} (m-1-7)
(m-2-1) edge node[auto,swap] {$q'_{\leq m-1}$} (m-2-2)
(m-2-2) edge node[auto,swap] {$q'_{<m-1}$} (m-2-3)
(m-2-3) edge node[auto,swap] {$q'_{ \leq m-2}$} (m-2-4)
(m-2-4) edge node[auto] {$$} (m-2-5)
(m-2-5) edge node[auto,swap] {$q'_{\leq 1}$} (m-2-6)
(m-2-6) edge node[auto,swap] {$q'_{< 1}$} (m-2-7)

(m-1-1) edge node[auto,swap] {$$} (m-2-1)
(m-1-2) edge node[auto,swap] {$\tau_{< m-1}\phi$} (m-2-2)
(m-1-3) edge node[auto,swap] {$\tau_{\leq m-2}\phi$} (m-2-3)
(m-1-5) edge node[auto,swap] {$\tau_{\leq 1}\phi$} (m-2-5)
(m-1-6) edge node[auto,swap] {$\tau_{< 1}\phi$} (m-2-6)
(m-1-7) edge node[auto,swap] {$\tau_{\leq 0}\phi$} (m-2-7)
;
\end{tikzpicture}
\end{equation}
in which the horizontal arrows are the strict $L_\infty$-morphisms induced by the surjective chain maps \eqref{eq:projs1}.

\end{enumerate}
\end{proposition}

\begin{proof}
We prove (1) and (2) for $\tau_{\leq m}L$. The same arguments apply for $\tau_{<m}L$.  
First, we verify that the brackets $\tau_{\leq m} \el_{k}$ are well-defined.
For degree reasons,  the only non-trivial case to check is $\tau_{\leq m} \el_{2}(\bar{x}_1,\bar{x}_2)$ when $\bar{x}_1$ is in degree $m$ and $\bar{x}_2$ is in degree 0. Suppose $x_{1}=x'_{1} + d_{m+1}z$, where $d_{m+1}$ is the differential $\el_1$ in degree $m+1$. The Jacobi-like identities \eqref{eq:Jacobi} for the $L_\infty$-structure imply that the degree 0 bracket $\el_2$ satisfies $\el_1\el_2(x_1,x_2) = \el_2(\el_1x_1,x_2) + \el_2(\el_1x_1,x_2) = \el_2(x_1,\el_1x_2)$. Hence, $\el_2(x_1,x_2) = \el_2(x'_1,x_2)
+ d_{m+1}\el_2(z,x_2)$, and so  $\tau_{\leq m} \el_{2}$ is well-defined. The fact that the brackets $\el_k$ satisfy the identities \eqref{eq:Jacobi} immediately implies that the brackets $\tau_{\leq m} \el_{k}$ satisfy them as well.

Next, let $\phi \maps (L,\el) \to (L',\el')$ be a morphism in $\LnA{n}$.
Define maps $\tau_{\leq m} \phi_k \maps \Lambda^k \tau_{\leq m} L \to \tau_{\leq m}L'$ by
\[
\tau_{\leq m} \phi_k(\bar{x}_1,\ldots,\bar{x}_k):= p'_{\leq m}\phi_k(x_1,\ldots,x_k),
\]
where $p'_{\leq m} \maps L' \to \tau_{\leq m} L'$ is the projection \eqref{eq:projs}.
We verify that these are well-defined. Again, for degree reasons, 
the only non-trivial case to check is $\tau_{\leq m} \phi_1(\bar{x})$ 
with $\bar{x}$ in degree $m$. Recall that $\phi_1$ is a chain map (Remark \ref{rmk:h0}). Hence,
if $x_{1}=x'_{1} + d_{m+1}z$, then $\tau_{\leq m} \phi_1(\bar{x})=\tau_{\leq m} \phi_1(\bar{x'})$.
The fact that the maps $\phi_k$ satisfy the defining equations \eqref{eq:dgmap} immediately implies that the maps $\tau_{\leq m} \phi_{k}$ form an $L_\infty$-morphism $\tau_{\leq m} \phi \maps
(\tau_{\leq m} L, \tau_{\leq m}\el) \to (\tau_{\leq m}L', \tau_{\leq m}\el')$.

For statement (3), since the $L_\infty$ brackets for $\tau_{\leq m}L$ and
$\tau_{<m}L$ are defined using the projection maps \eqref{eq:projs},
it is easy to see that the horizontal projections in the diagram
\eqref{diag:tower} are strict $L_\infty$-morphisms. Since the vertical
morphisms $\tau_{\leq m}
\phi$ and $\tau_{<m}\phi$ are also defined using the projection maps \eqref{eq:projs}, the diagram indeed commutes. 
\end{proof}

\begin{remark} \label{rmk:tau0}
The Lie $n$-algebra $\tau_{\leq 0} L$ is just the Lie algebra $H_{0}(L)$ concentrated in degree zero. Given a morphism of Lie $n$-algebras 
$f \maps (L, \el) \to (L',\el')$, the induced morphism $\tau_{\leq 0} f \maps
H_{0}(L) \to H_{0}(L')$ of Lie algebras is the morphism $H_0(f_1)$ from Remark \ref{rmk:h0}.
\end{remark}

\subsection{A functorial decomposition of  towers} 
Let $f \maps (L,\el) \to (L',\el')$ be a quasi-split fibration \eqref{def:split_fib}. 
Our goal is to decompose the induced morphism between the Postnikov towers
associated to $L$ and $L'$. Let us make two simple initial observations.
First, recall that every fibration in $\LnA{n}$ can be factored into an isomorphism followed by a strict fibration (Prop.\ \ref{prop:strict_factor}). Hence, we restrict our discussion here to strict quasi-split fibrations. Second, it follows directly from the definition that every quasi-split fibration is an $L_\infty$-epimorphism (Def.\ \ref{def:quasi-iso}). Therefore, we present our results below in a slightly more general context for the case when $f$
is a strict $L_\infty$-epimorphism.  

We begin with the following useful lemma. A variation of this result arises in the construction of minimal models for $L_\infty$-algebras. 

\begin{lemma} \label{lem:min_mod1}
Let $f \maps (L, \el) \to (L',\el')$ be an acyclic fibration in $\LnA{n}$.
 Let $(\ker f_1, \el_1)$ denote the kernel of the chain map $f_1 \maps (L, \el_1) \to (L',\el'_1)$ considered as an abelian Lie $n$-algebra (Example \ \ref{ex:lna}). Then there exists a $L_\infty$-morphism
 \[
 r \maps (L, \el) \to (\ker f_1, \el_1)
 \]
such that the morphism induced via the universal property of the product:
\begin{equation*} 
\bigl(f,r \bigr) \maps (L,\el) \to (L' \oplus \ker f_1, \el' \oplus \el_{\ker f})
\end{equation*}
is an isomorphism of Lie $n$-algebras.

\end{lemma}

\begin{proof}
Since $f$ is an acyclic fibration, the chain map $f_1 \maps (L, \el_1) \to (L',\el'_1)$ is an acyclic fibration in $\Chain^{\proj}$. Therefore, there exists a chain map $\sigma \maps (L',\el'_1) \to (L,\el_1)$ such that $f_1 \sigma = \id_{L'}$. Moreover, since the complexes $(L, \el_1)$ and $(L',\el'_1)$ are bifibrant in $\Chain^{\proj}$, there exists a chain homotopy $h \maps L \to L$ such that $\id_L - \sigma f_1 = \el_1h + h \el_1$. We consider the following chain map:
\begin{equation} \label{eq:min_mod1.1.1}
g_1 \maps L \to \ker f_1, \quad g_1 := \id_L - \sigma f_1
\end{equation}
and for $k \geq 2$, define the degree $k-1$ multi-linear maps
\begin{equation} \label{eq:min_mod1.1.2}
g_k \maps \Lambda^k L \to \ker f_1, \quad g_k:= g_1 \circ h \circ \el_k.
\end{equation}
Let $G \maps \S(\bs L) \to \S( \bs \ker f_1)$ denote the coalgebra map associated to the maps $g_k$.
Let $\delta$ and $\delta_{\ker f_1}$ denote the codifferentials corresponding to the $L_\infty$-structures on $L$ and $\ker f_1$, respectively.
We verify that $(\delta_{\ker f_1}\circ G)^1_m$ equals  $(G \circ \delta)^1_m$ for all $m \geq 1$. Indeed, since $g_1$ is a chain map, we have $( \delta_{\ker f_1} \circ G)^1_1 = (G \circ \delta)^{1}_1$.
Now let $m \geq 2$. Equations \ref{eq:struc_skew} and \ref{eq:morph_eq1}  imply that
\begin{equation} \label{eq:min_mod1.1}
\bigl ( \delta_{\ker f_1} \circ G)^1_m = \sgn{m} \bs \bigl (\el_1  g_1h  \el_m \bigr) (\bs^{-1})^{\tensor m}
\end{equation}     
and
\begin{equation*}
\begin{split}
(G \circ \delta)^{1}_m =  G^1_1 \delta^1_m + \sum_{k \geq 2}^m  G^1_k \delta^k_m
= G^1_1 \delta^1_m + \sum_{k \geq 2}^m + \bs g_1 \circ h \bs^{-1} \delta^1_k \delta^k_m.
\end{split}
\end{equation*}
Since $\delta \circ \delta =0$, we use Eq.\ \ref{eq:codiff} to rewrite
the last term on the right hand side:  
\[
\begin{split}
(G \circ \delta)^{1}_m =  G^1_1 \delta^1_m  - \bs g_1h \bs^{-1} \delta^1_1 \delta^1_m
 = \sgn{m} \bs \bigl (g_1 - g_1 h \el_1 \bigr) \el_m (\bs^{-1})^{\tensor m}.
\end{split}
\]
By comparing the above equality with Eq.\ \ref{eq:min_mod1.1}, and using the fact
that $g_1 - g_1h\el_1 = \el_1 gh$, we conclude that $\bigl ( \delta_{\ker f_1} \circ G)^1_m =  
(G \circ \delta)^{1}_m$. Hence,  $g \maps (L, \el_k) \to (\ker f_1, \el_1)$ 
is an $L_\infty$-morphism. Finally, since the linear map $(f_1,g_1) \maps L \to L' \oplus \ker f_1$ is an isomorphism of complexes, it follows that the induced morphism
\[
\bigl(F,G \bigr) \maps L \to L' \oplus \ker f_1
\] 
is an $L_\infty$-isomorphism. 

\end{proof}

Our first decomposition result involves the commuting squares in \eqref{diag:tower} whose top edges are the strict acyclic fibrations  $q_{<m+1} \maps (\tlt{m+1} L, \tlt{m+1} \el) \to (\tleq{m} L, \tleq{m} \el)$ defined \eqref{eq:projs1}. 
Let $\ker q_{< m+1}$ denote the kernel of the strict acyclic fibration of the chain map $q_{< m+1}$.
Then $\ker q_{< m+1}$ is an abelian Lie $n$-algebra concentrated in degrees $m$ and $m+1$ with
\begin{equation*} 
(\ker q_{< m+1})_m=\im d_{m+1}, \quad  (\ker q_{< m+1})_{m+1} = \im d_{m+1}[-1]. 
\end{equation*}
The induced differential $\el^{\ker}=\el_1$ on $\ker q_{< m+1}$ is simply the desuspension isomorphism.

\begin{proposition} \label{prop:tower_decomp1}
Let $f \maps (L, \el) \to (L',\el')$ be a strict $L_\infty$-epimorphism between Lie $n$-algebras.
Then there exists morphisms in $\LnA{n}$
\[
r \maps (\tau_{< m +1}L, \tlt{m+1}\el) \to (\ker q_{<m+1},\el^{\ker}), \quad 
r' \maps (\tau_{< m +1}L', \tlt{m+1}\el') \to (\ker q'_{<m+1},\el^{\prime \ker})
\] 
inducing $L_\infty$-isomorphisms
\[
\begin{split}
\bigl(q_{< m+1},r \bigr) &\maps \bigl (\tau_{< m +1}L, \tlt{m+1}\el \bigr) \xto{\cong} \bigl( \tau_{\leq m} L \oplus \ker q_{<m+1}, \tleq{m}\el \oplus \el^{\ker} \bigr)    \\
\bigl(q'_{< m+1},r' \bigr)&\maps \bigl( \tau_{< m +1}L', \tlt{m+1}\el' \bigr) \xto{\cong} \bigl( \tau_{\leq m} L' \oplus \ker q'_{<m+1}, \tleq{m}\el' \oplus \el^{\prime \ker} \bigr)
\end{split}
\]
such that the following diagram commutes in $\LnA{n}$:
\begin{equation} \label{diag:tower_decomp1.1}
\begin{tikzpicture}[descr/.style={fill=white,inner sep=2.5pt},baseline=(current  bounding  box.center)]
\matrix (m) [matrix of math nodes, row sep=2em,column sep=5em,
  ampersand replacement=\&]
  {  
\tau_{< m +1 }L \& \tau_{\leq m }L \oplus \ker q_{<m+1}\\
\tau_{< m +1 }L' \& \tau_{\leq m }L' \oplus \ker q'_{<m+1}\\
}; 
\path[->,font=\scriptsize] 
(m-1-1) edge node[auto]{$\bigl(q_{< m+1},r \bigr)$} node[auto,below] {$\cong$} (m-1-2)
(m-1-1) edge node[auto,swap] {$\tau_{< m +1 } f$} (m-2-1)
(m-1-2) edge node[auto] {$\tau_{\leq m} f \oplus \tau_{< m +1 }f \vert_{\ker}$} (m-2-2)
(m-2-1) edge node[auto]{$\bigl(q'_{< m+1},r' \bigr)$} node[below] {$\cong$} (m-2-2)
;
\end{tikzpicture}
\end{equation}
\end{proposition}

\begin{proof}
Since $f$ is strict, we have $f=f_1$ and so
Prop.\ \ref{prop:lna_tower} implies that we have the following commutative diagram in $\LnA{n}$:
\begin{equation} \label{diag:tower_decomp1.2}
\begin{tikzpicture}[descr/.style={fill=white,inner sep=2.5pt},baseline=(current  bounding  box.center)]
\matrix (m) [matrix of math nodes, row sep=2em,column sep=2em,
  ampersand replacement=\&]
  {  
\tau_{< m +1 }L \& \tau_{\leq m }L\\
\tau_{< m +1 }L' \& \tau_{\leq m }L'\\
}; 
\path[->,font=\scriptsize] 
(m-1-1) edge node[auto] {$q_{< m+1}$} (m-1-2)
(m-1-1) edge node[auto,swap] {$\tlt{m+1}f_1$} (m-2-1)
(m-1-2) edge node[auto] {$\tleq{m}f_1$} (m-2-2)
(m-2-1) edge node[auto] {$q'_{< m+1}$} (m-2-2)
;
\end{tikzpicture}
\end{equation}
For the sake of brevity, let $V$ and $V'$ denote the abelian Lie $n$-algebras $\ker q_{<m+1}$ and $\ker q'_{<m+1}$, respectively.

Since $q_{< m+1}$ and $q'_{< m+1}$ are acyclic fibrations, Lemma \ref{lem:min_mod1} provides us 
with $L_\infty$-isomorphisms\\ $(q_{< m+1},r) \maps \tau_{< m +1 }L \xto{\cong} \tau_{\leq m }L \oplus V$ and
$(q'_{< m+1},r') \maps \tau_{< m +1 }L' \xto{\cong} \tau_{\leq m }L' \oplus V'$. 
We will show that in the proof of Lemma \ref{lem:min_mod1}, we can choose the morphisms $r$ and $r'$ such that diagram \eqref{diag:tower_decomp1.1} commutes. 

In degree $m$, \eqref{diag:tower_decomp1.2} corresponds to the following commutative diagram between short exact sequences of vector spaces:
\[
\begin{tikzpicture}[descr/.style={fill=white,inner sep=2.5pt},baseline=(current  bounding  box.center)]
\matrix (m) [matrix of math nodes, row sep=2em,column sep=2em,
  ampersand replacement=\&]
  {  
\im d_{m+1} \& L_{m} \& \coker d_{m+1} \\
\im d'_{m+1} \& L'_{m} \& \coker d'_{m+1} \\
}; 
\path[->,font=\scriptsize] 
 (m-1-1) edge node[auto] {$i$} (m-1-2)
 (m-2-1) edge node[auto] {$i'$} (m-2-2)
;
\path[->,font=\scriptsize] 
 (m-1-2) edge node[auto] {$\pi$} (m-1-3)
 (m-2-2) edge node[auto] {$\pi'$} (m-2-3)
 ;
\path[->,font=\scriptsize] 
 (m-1-1) edge node[auto,swap] {$\tlt{m+1}f_1 \vert_{ \im d}$} (m-2-1)
 (m-1-2) edge node[auto] {$\tlt{m+1}f_1$} (m-2-2)
 (m-1-3) edge node[auto] {$\tleq{m}f_1$} (m-2-3)
 ;
\end{tikzpicture}
\]
Since $f$ is an $L_\infty$-epimorphism, the maps $\tlt{m+1}f_1 \vert_{\im d}$ and 
$\tleq{m}f_1$ are surjections. This, along with the fact that the rows are exact, implies that there exists sections $s \maps \coker d_{m+1} \to L_m$, and $s' \maps \coker d'_{m+1} \to L'_m$, of $\pi$ and $\pi'$, respectively, such that
\[
\tlt{m+1}f_1 \circ s =s' \circ \tleq{m}f_1.
\] 
The linear maps $s$ and $s'$ induce sections $\sigma \maps \tleq{m}L \to 
\tlt{m+1}L$ and $\sigma' \maps \tleq{m}L' \to \tlt{m+1}L'$ in $\Chain$ of the 
chain maps $q_{< m+1}$ and $q'_{m+1}$, respectively. Explicitly, we have
\[
\sigma(x):=
\begin{cases}
x, & \text{if $\deg{x} < m$} \\
0, & \text{if $\deg{x} > m $} \\
s(x), & \text{if $\deg{x} =m $} \\
\end{cases}
\]
with an analogous formula for $\sigma'$. Moreover, it follows that 
\begin{equation} \label{eq:tower_decomp1.1}
\tlt{m+1}f_1 \circ \sigma = \sigma' \circ \tleq{m}f_1.
\end{equation}
We then construct chain homotopies $h \maps \tlt{m+1} L \to \tlt{m+1} L[1]$ and $h' \maps \tlt{m+1} L' \to \tlt{m+1} L'[1]$, as in the proof of Lemma \ref{lem:min_mod1}. 
Explicitly, we have
\[
h(x):=
\begin{cases}
0, & \text{if $\deg{x} < m$ or $\deg{x} > m$} \\
\bs (x- s \pi (x)) \in \im(d_{m+1})[1] , & \text{if $\deg{x} =m $} \\
\end{cases}
\]
with an analogous formula for $h'$. Hence, the homotopies satisfy
\begin{equation} \label{eq:tower_decomp1.2}
\tlt{m+1}f_1 \circ h = h' \circ \tlt{m+1}f_1. 
\end{equation}
We then use the chain maps $\sigma$, $\sigma'$, the homotopies $h$ and $h'$, and
the $L_\infty$-structures $\tlt{m+1} \el$ and $\tlt{m+1} \el'$
to construct $L_\infty$-morphisms
\[
r \maps \tau_{< m +1}L \to V, \quad r' \maps \tau_{< m +1}L' \to V'
\] 
via formulas \eqref{eq:min_mod1.1.1} and \eqref{eq:min_mod1.1.2} in the proof of Lemma \ref{lem:min_mod1}.

Using \eqref{eq:tower_decomp1.1}, \eqref{eq:tower_decomp1.2}, and the fact that
$\tlt{m+1} f_1$ is an $L_\infty$-morphism, a direct computation verifies that 
for all $k \geq 1$, we have 
\[
r'_k \circ \bigl (\tlt{m+1}f_1 \bigr)^{\tensor k} = 
\tlt{m+1}f_1 \vert_{V} \circ r_k.
\]
Hence, the diagram \eqref{diag:tower_decomp1.1} commutes. 
\end{proof}

Now let $m\geq 1$. We focus on those commuting squares in \eqref{diag:tower} whose top edges are the strict quasi-split fibrations  $q_{\leq m} \maps (\tleq{m} L, \tleq{m} \el) \to (\tlt{m} L, \tlt{m} \el)$ defined \eqref{eq:projs1}. 
As a chain map, $q_{\leq m}$ induces a short exact sequence of chain complexes
\[
H_m \xto{i} \tleq{m}L \xto{q_{\leq m}}  \tlt{m}L 
\]
where $H_n$ is the homology group $H_n(L)$ concentrated in degree $n$ with trivial differential.
Our second decomposition result is the following:

\begin{proposition} \label{prop:tower_decomp2}
Let $m \geq 1$ and let $f \maps (L, \el) \to (L',\el')$ be a strict $L_\infty$-epimorphism in $\LnA{n}$ such that the induced map in homology
\[
H(f_1) \maps H_{m} \to H'_m
\] 
is surjective in degree $m$. Then there exists $L_\infty$-structures $\hat{\el}$ and $\hat{\el}'$ on the graded vector spaces $\tlt{m} L \oplus H_m$ and $\tlt{m} L' \oplus H'_m$, respectively,  and $L_\infty$-isomorphisms
\[
\begin{split}
\hat{q} \maps (\tleq{m} L, \tleq{m}\el) &\xto{\cong} \bigl( \tlt{m}L \oplus H_m, \hat{\el} ~ \bigr)\\
\hat{q}' \maps (\tleq{m} L', \tleq{m}\el') & \xto{\cong} \bigl( \tlt{m}L' \oplus H'_m, \hat{\el'} ~ \bigr)
\end{split}
\]
such that the following diagram of strict $L_\infty$-morphisms commutes
\begin{equation}
\label{diag:tower_decomp2.1}
\begin{tikzpicture}[descr/.style={fill=white,inner sep=2.5pt},baseline=(current  bounding  box.center)]
\matrix (m) [matrix of math nodes, row sep=2em,column sep=2em,
  ampersand replacement=\&]
  {  
( \tau_{\leq m }L, \tleq{m} \el) \& (\tau_{\lt m }L \oplus H_m, \hat{\el})\\
(\tau_{\leq m }L', \tleq{m} \el') \& (\tau_{\lt m }L' \oplus H_m', \hat{\el}^{\prime})\\
}; 
\path[->,font=\scriptsize] 
(m-1-1) edge node[auto]{$\hat{q}$} node[auto,below] {$\cong$} (m-1-2)
(m-1-1) edge node[auto,swap] {$\tleq{m} f$} (m-2-1)
(m-1-2) edge node[auto] {$\tlt{m} f \oplus H(f)$} (m-2-2)
(m-2-1) edge node[auto]{$\hat{q}'$} node[below] {$\cong$} (m-2-2)
;
\end{tikzpicture}
\end{equation}
\end{proposition}

\begin{proof}
Since $f$ is strict, we have $f=f_1$ and so Prop.\ \ref{prop:lna_tower} implies that we have the following commutative diagram in $\LnA{n}$:
\begin{equation*} 
\begin{tikzpicture}[descr/.style={fill=white,inner sep=2.5pt},baseline=(current  bounding  box.center)]
\matrix (m) [matrix of math nodes, row sep=2em,column sep=2em,
  ampersand replacement=\&]
  {  
\tau_{\leq m }L \& \tau_{\lt m }L\\
\tau_{\leq m }L' \& \tau_{\lt m }L'\\
}; 
\path[->,font=\scriptsize] 
(m-1-1) edge node[auto] {$q_{\leq m}$} (m-1-2)
(m-1-1) edge node[auto,swap] {$\tleq{m}f_1$} (m-2-1)
(m-1-2) edge node[auto] {$\tlt{m}f_1$} (m-2-2)
(m-2-1) edge node[auto] {$q'_{\leq m}$} (m-2-2)
;
\end{tikzpicture}
\end{equation*}
In degree $m$, this corresponds to the following commutative diagram between short exact sequences of vector spaces:
\[
\begin{tikzpicture}[descr/.style={fill=white,inner sep=2.5pt},baseline=(current  bounding  box.center)]
\matrix (m) [matrix of math nodes, row sep=2em,column sep=2em,
  ampersand replacement=\&]
  {  
H_m \& \coker d_{m+1} \& \im d_m[-1] \\
H'_m \& \coker d'_{m+1} \& \im d'_m[-1] \\
}; 
\path[->,font=\scriptsize] 
 (m-1-1) edge node[auto] {$i$} (m-1-2)
 (m-2-1) edge node[auto] {$i'$} (m-2-2)
;
\path[->,font=\scriptsize] 
 (m-1-2) edge node[auto] {$d_m$} (m-1-3)
 (m-2-2) edge node[auto] {$d'_m$} (m-2-3)
 ;
\path[->,font=\scriptsize] 
 (m-1-1) edge node[auto,swap] {$H(f_1)$} (m-2-1)
 (m-1-2) edge node[auto] {$\tleq{m}f_1$} (m-2-2)
 (m-1-3) edge node[auto] {$\tlt{m}f_1$} (m-2-3)
 ;
\end{tikzpicture}
\]
By hypothesis, the vertical maps are surjections. Let 
\[
\mu \maps H'_m \to H_m, \quad \nu \maps \im d'_{m}[-1] \to \im d_{m}[-1], \quad 
\psi \maps \im d_{m}[-1] \to \coker d_{m+1},
\]
be sections of $H(f_1)$, $\tlt{m} f_1$, and $d_m$, respectively.
Denote by $s' \maps \im d'_m[-1] \to \coker d'_{m+1}$ the composition $s':= \tleq{m} f_1 \circ \psi \circ \nu$. In general, the linear map $\tleq{m}f_1 \psi - s' \tlt{m}f_1$ will not equal zero.
So let $s \maps \im d_m[-1] \to \coker d_{m+1}$ be the composition $s:=\psi - i \circ \mu \circ (\tleq{m}f_1 \psi - s' \tlt{m}f_1)$. Then $s$ and $s'$ are sections of $d_m$ and $d'_m$, respectively, and 
\[
\tleq{m }f_1 \circ s =s' \circ \tlt{m}f_1.
\] 
We use $s$ and $s'$ to define chain maps. Let $t \maps \tau_{< m}L \to \tau_{\leq m}L$ and $\hat{r} \maps \tleq{m} L \to H_m$ be the linear maps
\[
t(x):=
\begin{cases}
s(x), & \text{if $\deg{x} =m $} \\
x, & \text{if $\deg{x} < m$,} \\
\end{cases} \qquad \hat{r}:= \id -tq_{\leq m}
\] 
respectively. Then the isomorphism $\hat{q} \maps \tleq{m}L \to \tlt{m}L \oplus H_m$ 
is defined to be
\[
\hat{q}(z):=\bigl (q_{\leq m}(z), \hat{r}(z) \bigr).
\]
The map $\hat{q}'$ is defined in the analogous way, using the 
section $s'$ instead of $s$. 
Hence \eqref{diag:tower_decomp2.1} commutes as a diagram in the category $\Chain$.

Finally, we construct compatible $L_\infty$-structures on $\tlt{m}L \oplus H_m$ and $\tlt{m}L' \oplus H'_m$ via transfer across the isomorphisms $\hat{q}$ and $\hat{q}'$.
For each $k \geq 1$, we define
\begin{equation} \label{eq:tower_decomp2.2}
\hat{\el}_k := \hat{q} \circ \tleq{m} \el_k \circ \bigl( \hat{q}^{-1}\bigr)^{\tensor k},
\quad \hat{\el}'_k := \hat{q}' \circ \tleq{m} \el'_k \circ \bigl( \hat{q}^{\prime -1} \bigr)^{\tensor k}.
\end{equation}
Hence, by construction, the chain maps $\hat{q}$ and $\hat{q}'$
lift to $L_\infty$-isomorphisms.
\end{proof}

\subsubsection{The structure of the ``twisted product'' $(\tlt{m} L \oplus H_m, \hat{\el})$} \label{sec:tower_decomp2}

We emphasize that, in general, the Lie $n$-algebra $\bigl (\tlt{m}L \oplus H_m, \hat{\el} \bigr)$ 
defined in the above proof of Cor.\ \ref{prop:tower_decomp2} is not the categorical product of the Lie $n$-algebra $(\tlt{m}L, \tlt{m}\el)$ with the abelian Lie $n$-algebra $H_m$. This is in contrast with the decomposition $\tau_{< m +1}L \cong \tau_{\leq m} L \oplus \ker q_{<m+1}$ given in Prop.\ \ref{prop:tower_decomp1}. 

Indeed, it follows from the definition of $\hat{q}$ given in the above proof that  
the $L_\infty$-structure maps $\hat{\el}_k \maps \Lambda^k (\tlt{m}L \oplus H_m) \to \tlt{m}L \oplus H_m$ 
defined in Eq.\ \ref{eq:tower_decomp2.2} can be written as
\begin{multline*} 
\hat{\el}_k\bigl( (x_1,y_1), (x_2,y_2), \ldots, (x_k,y_k) \bigr) =  \\ 
 \Bigl( \tlt{m} \el_k \bigl (x_1,x_2,\ldots,x_k), ~ \hat{r}\circ \tleq{m} \el_k \bigl (t x_1 + y_1, t x_2 +y_2,\ldots, t x_k+ y_k) \Bigr).
\end{multline*}    
Note that there is only one non-trivial structure map 
$\hat{\el}_k$ that involves non-zero inputs from $H_m$, since 
$H_m$ is concentrated in top degree $m$. Namely:
\[
\hat{\el}_2\bigl( (x,0), (0,y) \bigr) = \bigl(0, \el_2(x,y) \bigr),
\]
where $x \in L_0=\tlt{m}L_0$ is an element of degree 0 and $y \in H_m$. 
This simple observation plays a key role in our study \cite{Rogers-Zhu:2018} of integrated quasi-split fibrations.

\end{document}